\UseRawInputEncoding
\documentclass[reqno]{amsart}

\usepackage{verbatim} 

\usepackage[utf8]{inputenc}
\usepackage[T1]{fontenc}
\usepackage[english]{babel}
\usepackage{lmodern}

\usepackage{eurosym}

\usepackage{enumitem}
\usepackage{graphicx}
\usepackage{float}
\usepackage{amsmath}
\allowdisplaybreaks[4]
\usepackage{mathtools}
\usepackage{amssymb}
\usepackage{amsthm}
\usepackage{bbm}

\newcommand{\N}{\mathbb{N}}

\newcommand{\R}{\mathbb{R}}

\newcommand{\ept}{\mathbb{E}}	

\theoremstyle{plain}
\newtheorem{theorem}{Theorem}
\newtheorem{proposition}[theorem]{Proposition}
\newtheorem{corollary}[theorem]{Corollary}
\newtheorem{lemma}[theorem]{Lemma}

\theoremstyle{definition}
\newtheorem{definition}[theorem]{Definition}
\newtheorem{remark}[theorem]{Remark}
\allowdisplaybreaks[4]

\usepackage{color}
\usepackage[square,sort,comma,numbers]{natbib}
\usepackage[a4paper,bindingoffset=0.2in,%
            left=0.5in,right=0.5in,top=1in,bottom=1in,%
            footskip=.25in ]{geometry}

\usepackage{setspace}

\usepackage{fancyhdr}

\usepackage{ifthen}

\allowdisplaybreaks[4]

\fancyheadoffset{\textwidth}
\usepackage{blindtext}
\usepackage[pdftex,pdfpagelabels=true]{hyperref}

\hypersetup{
    colorlinks=true,
    linkcolor=blue,
    filecolor=magenta,
    urlcolor=cyan,
}

\usepackage[title]{appendix}

\begin{document}

\title[]{Two-dimensional signal-dependent parabolic-elliptic Keller-Segel system and its mean-field derivation}
\date{\today}

\author{Lukas Bol}
\address{School of Business Informatics and Mathematics, Universität Mannheim, 68131, Mannheim, Germany}
\email{lukas.bol@uni-mannheim.de}

\author{Li Chen}
\address{School of Business Informatics and Mathematics, Universität Mannheim, 68131, Mannheim, Germany}
\email{li.chen@uni-mannheim.de}

\author{Yue Li}
\address{School of Mathematics and Statistics, Shandong Normal University, Jinan 250014, China}
\address{Institute of Analysis and Scientific Computing, Technische Universit\"at, Wiedner Hauptstra\ss e 8-10, 1040 Wien, Austria}
\email{liyue2011008@163.com}

\begin{abstract}
In this paper, the well-posedness of two-dimensional signal-dependent Keller-Segel system and its mean-field derivation from a interacting particle system on the whole space are investigated. The signal dependence effect is reflected by the fact that the diffusion coefficient in the particle system depends nonlinearly on the interactions between the individuals. Therefore, the mathematical challenge in studying the well-posedness of this system lies in the possible degeneracy and the aggregation effect when the concentration of signal becomes unbounded. The well-established method on bounded domains, to obtain the appropriate estimates for the signal concentration, is invalid for the whole space case. Motivated by the entropy minimization method and Onofri's inequality, which has been successfully applied for the parabolic-parabolic Keller-Segel system, we establish a complete entropy estimate benefited from the linear diffusion term, which plays an important role in obtaining the $L^p$ estimates for the solution. Furthermore, the upper bound for the concentration of signal is obtained.
Based on the estimates we obtained for the density of bacteria, the rigorous mean-field derivation is proved by introducing an intermediate particle system with a mollified interaction potential with logarithmic scaling.
By using this mollification, we obtain the convergence of the particle trajectories in expectation, which implies the weak propagation of chaos.
Additionally, under a regularity assumption of the initial data, we infer higher regularity for the solutions,
which allows us to use the relative entropy method to derive the strong $L^1$ convergence for the propagation of chaos.
\end{abstract}

\maketitle

\pagenumbering{arabic}

\section{Introduction}
In research article \cite{FTLHHL}, the authors introduced a signal-dependent Keller-Segel system which describes the emergence of stripe patterns via a so-called self-trapping mechanism.
This process, extensively examined through experimental techniques in synthetic biology, involves Escherichia coli secreting a signaling molecule, acyl-homoserine lactone (AHL).
At low AHL concentrations, these bacteria are highly mobile, engaging in random motion through standard swimming and tumbling behaviors with minimal disruption.
However, as AHL levels rise, the bacterial population's behavior undergoes notable shifts, culminating in a predominantly static state on a macroscopic scale.
The general signal-dependent Keller-Segel model introduced in the literature reads as
\begin{align}
	\label{signal-dependent}
	\begin{cases}
		\partial_t u=\Delta(\gamma(v)u),    \\
		\tau \partial_t v-\Delta v+v=u,
	\end{cases}
\end{align}
where $\tau=0,1$, $u$ and $v$ stand for the density of bacteria (or cells) and the concentration of chemical signals respectively,
and $\gamma(v)$ represents the signal-dependent motility which satisfies $\gamma'(v)\leq 0$. As one could immediately observe, due to the fact $\Delta(\gamma(v)u)=\nabla\cdot (\gamma(v)\nabla u+\gamma'(v)\nabla v u)$, the key challenge in analyzing this system lies in the possible degeneracy of $\gamma(v)$ and the aggregation effect from $\gamma'(v)\nabla v$, when the signal concentration $v$ and its gradient $\nabla v$ become unbounded.

System \eqref{signal-dependent} has been intensively studied in the last decade on bounded domains.
Under the assumption that there are upper and lower bounds for the signal-dependent motility $\gamma(v)$ and $\gamma'(v)$,
Tao and Winkler \cite{TW1} established the existence of global classical solutions for the two-dimensional case,
whereas the system admits weak solutions in higher dimensions.
By imposing additional constraints on the initial data, one can obtain classical solutions for the three-dimensional case.
Two different ways have been introduced in the literature to overcome the difficulty caused by the degenerate motility.
An upper bound for $v$ is usually obtained by using energy inequalities and the elliptic/parabolic regularity theory.
However, this approach has its limitations and is primarily applicable to specific cases of $\gamma(v)$ \cite{AY, YK, JW},
or when logistic source terms are present \cite{JKW, LY, WW}.
Alternatively, introducing an auxiliary elliptic problem satisfying the comparison principle can lead to an explicit upper bound for $v$
\cite{FJ1, FJ2, FS22, LJ}.
Furthermore, \cite{FJ3, J1, JL21} refined this approach by introducing the Alikakos-Moser-type iteration method to yield uniform-in-time
upper bounds for $v$.

Within the context of a two-dimensional bounded domain,
the well-posedness of solutions to the signal-dependent Keller-Segel system \eqref{signal-dependent} has received considerable attention.
Fujie and Jiang \cite{FJ1} emphasized that as long as $\gamma$ satisfies
\begin{align*}
0<\gamma(\cdot)\in C^3[0,\infty),\;\gamma'(\cdot)\leq 0\;{\rm{on}}\; (0,+\infty),\;\lim_{s\rightarrow +\infty}\gamma(s)=0,
\end{align*}
the system \eqref{signal-dependent} admits global classical solutions for any initial data.
Moreover, Fujie and Jiang \cite{FJ3} proved that if the motility decreases more slowly than the standard exponential level at high concentrations,
classical solutions of \eqref{signal-dependent} remain uniformly bounded in time.
In addition, the authors also obtained that if $\lim_{s\rightarrow +\infty}e^{\alpha s}\gamma(s)=+\infty$ for any $\alpha>0$, the classical solution to the system
\eqref{signal-dependent} exists globally and remains bounded uniformly in time.
Remarkably, in cases $\gamma(v)$ decays exponentially, i.e. $\gamma(v)=e^{-v}$,
the system \eqref{signal-dependent} exhibits a shared Lyapunov functional with the classical Keller-Segel model.
The authors \cite{FJ1, FJ2, JW} derived the critical mass for model \eqref{signal-dependent} with $\gamma(v)=e^{-\alpha v}$ $(\alpha>0)$ in the two-dimensional setting,
and proved that solutions remain uniformly bounded when the initial bacteria mass is below this critical mass.
Conversely, studies in \cite{BLT, FJ1, FJ2} indicate that solutions blow-up as time approaches infinity when the initial bacteria mass exceeds the critical mass.
Nevertheless, it is impossible to blow-up within finite time intervals, which is a difference from the Keller-Segel model.
There is so far, to the authors' knowledge, no analytical results for the signal-dependent Keller-Segel model on the whole space. The key motivation to study this model on the whole space case is to understand the signal-dependent effect on the microscopic level and to rigorously prove its macroscopic limit.

In this paper, we consider the case $\tau=0$, which represents very fast diffusion of the chemical signals.
Our goal is to present the rigorous derivation of the following signal-dependent Keller-Segel system in two dimensions:
\begin{align}
	\label{singalKellerSegel}
	\begin{cases}
		\partial_t u=\Delta( e^{-v}u+ u),    \\
		-\Delta v+v=\chi u, \\
		u(0,x)=u_0(x), \qquad x\in \R^2,\;\; t>0,
	\end{cases}
\end{align}
where $\chi$ is a positive constant describing the strength rate of signals.
An important feature of this system is the conservation of mass $\int_{\R^2}u(t,x)\,dx=\int_{\R^2}u_0(x)\,dx=1$, for any $t>0$, which satisfies the basic property in establishing a mean-field approximation of it.
The second equation in the system \eqref{singalKellerSegel} can be rewritten as $v=\Phi\ast u$, where $\Phi:=\chi\tilde{\Phi}$ is defined via the Yukawa
potential $\tilde{\Phi}$ verifying $\nabla\tilde{\Phi}\in L^q(\R^2)$ for $1\le q<2$ and $\tilde{\Phi}\in L^p(\R^2)$ for $1\le p<\infty$ \cite{LL}.
This also explains that the time evolution of bacteria is mainly driven by the interactions with other bacteria through a given signal potential. Since the potential $\Phi$ has strong singularity at the origin, we will use a mollification of it on the microscopic level.

We plan to derive \eqref{singalKellerSegel} rigorously from a system of interacting stochastic differential equations,
which describes the movements of $N$ particles $(X^\varepsilon_{N,i})_{1\leq i\leq N}$ in $\R^2$:
\begin{align}
\label{generalized_regularized_particle_model}
\begin{cases}
dX^\varepsilon_{N,i}
= \Big(2\exp\Big(-\displaystyle\frac{1}{N}\displaystyle\sum_{j=1}^N\Phi^\varepsilon(X_{N,i}^\varepsilon-X^\varepsilon_{N,j}) \Big)+2\Big)^{1/2} dB_i(t) , \\
X^\varepsilon_{N,i}(0) =\zeta_i,\qquad 1\leq i\leq N,
\end{cases}
\end{align}
where we use $\Phi^\varepsilon$ to denote the approximation of $\Phi$, i.e.  $\Phi^\varepsilon:= \Phi * j^\varepsilon$. Here $j^\varepsilon (x) := \frac{1}{\varepsilon^2 } j(x / \varepsilon)$ is a standard mollification kernel.
We take the standard setting from stochastic analysis. Let $(\Omega, \mathcal{F}, (\mathcal{F}_{t\geq 0}) , \mathbb{P})$ be a complete filtered probability space,
 $2$-dimensional $\mathcal{F}_t$-Brownian motions $\{ (B_i(t))_{t\geq 0} \}_{i=1}^\infty$ be assumed to be independent of each other, the initial data $\{\zeta_i\}_{i=1}^\infty$ be independent and identically distributed (i.i.d.) random variables, which are also independent of $\{ (B_i(t))_{t\geq 0} \}_{i=1}^\infty$. Let  $u_0$ be the common density of $\{\zeta_i\}_{i=1}^\infty$.

The concept of mean-field theory and propagation of chaos originated in statistical physics, particularly in the study of the Vlasov equation.
Reviews on this topic can be found in \cite{golse2016dynamics,jabin2014review}.
For an overview of mean-field theory of stochastic interacting particle systems, we refer to \cite{jabin2017mean,sznitman1991topics}.
In mean-field theory, especially when dealing with smooth interactions, the classical approach relies on direct trajectory estimates.
However, when dealing with singular interaction potentials,
additional assumptions on the initial data or cut-off potentials may be necessary to enhance the classical solution theory for the corresponding particle system.
When the interaction potential is of Coulomb type,
Lazarovici and Pickl \citep{lazarovici2017mean} studied a particle system with a cut-off potential and random initial data.
They established the convergence of particle trajectories in the sense of probability, thereby directly implying the propagation of chaos.
Further research on this topic, utilizing various cut-off techniques for different models, can be found in \citep{boers2016mean,bolley2011stochastic,garcia2017microscopic,chen2017mean,chen2020combined,godinho2015propagation,lazarovici2017mean}.
Oelschl\"ager explored moderately interacting particle systems in \citep{oelschlager1990large}, aiming to derive equations of the porous medium type.
Philipowski later extended this approach in \citep{philipowski2007interacting} by introducing logarithmic scaling to arrive at the porous medium equation.
This method has since been further generalized by \citep{chen2019rigorous} to derive a cross-diffusion system, by \citep{chen2021rigorous} to obtain the SKT system (where moderate interactions influence the diffusion coefficient), and by \citep{chen2021analysis} to formulate a non-local porous medium equation.
For a more comprehensive understanding of interacting particle systems with singular potentials and their generalizations,
refer to \citep{jabin2017mean} and the citations therein.

To derive rigorously \eqref{singalKellerSegel} from \eqref{generalized_regularized_particle_model},
we introduce an intermediate particle problem, which is formally viewed as a mean-field limit $N\rightarrow \infty$ in the system \eqref{generalized_regularized_particle_model} for fixed
$\varepsilon>0$, namely
\begin{align}
\label{generalized_intermediate_particle_model}
\begin{cases}
d\bar{X}^\varepsilon_{i} = \big(2\exp(-\Phi^\varepsilon\ast u^\varepsilon(t, \bar{X}^\varepsilon_{i}))+2\big)^{1/2} dB_i(t) ,\\
\bar{X}^\varepsilon_{i}(0) =\zeta_i,\qquad 1\leq i\leq N,
\end{cases}
\end{align}
where $ u^\varepsilon(t, x)$ is the probability density of $\bar{X}^\varepsilon_{i}$ and solves the following non-local partial differential equations
\begin{align}
\label{kellersegelmedium}
\begin{cases}\partial_t u^\varepsilon =  \Delta ( e^{-{v^\varepsilon}}u^\varepsilon+ u^\varepsilon),     \\
-\Delta {v^\varepsilon}+{v^\varepsilon}=\chi u^\varepsilon\ast j^\varepsilon,  \\
u^\varepsilon(0,x)=u_0(x), \qquad x\in \R^2,\;\; t>0.
\end{cases}
\end{align}
Since for any fixed $\varepsilon$, notice that $\|u^\varepsilon(t,\cdot)\|_{L^1(\R^2)}=1$, we have $e^{-v^\varepsilon}=e^{-\Phi^\varepsilon \ast u^\varepsilon}$ is bounded, thus it is standard to show, by the classical parabolic/elliptic theory, that system \eqref{kellersegelmedium} has a unique global smooth non-negative solution $(u^\varepsilon,{v^\varepsilon})$.
The intermediate Mckean-Vlasov system \eqref{generalized_intermediate_particle_model} is an approximation of the limiting Mckean-Vlasov system as $\varepsilon\rightarrow 0$:
\begin{align}
\label{generalized_particle_model}
\begin{cases} d\hat{X}_{i} =   \big(2\exp(-\Phi\ast u(t, \hat{X}_{i}))+2\big)^{1/2} dB_i(t),  \\
\hat{X}_{i}(0) =\zeta_i,\qquad 1\leq i\leq N,
\end{cases}
\end{align}
where $ u(t, x)$ is the probability density of $\hat{X}_{i}$ and solves the partial differential equations \eqref{singalKellerSegel}.

As can be expected, due to the singularity of $\Phi$, more estimates for $u$ should be obtained so that $\Phi*u$ can be Lipschitz continuous to solve the stochastic system \eqref{generalized_particle_model}. Therefore, the first result of this paper focus on the well-posedness result of \eqref{singalKellerSegel}. We give the definition of weak solution to this problem.
\begin{definition}[Definition of weak solutions for \eqref{singalKellerSegel}]
\label{definition}
Let $T>0$ be arbitrary. The couple $(u,v)$ is called a weak solution to the problem \eqref{singalKellerSegel} on $[0,T]$
if $(u,v)$ satisfies $0\leq u$, $0\leq v$ in $(0,T)\times\R^2$,
\begin{align*}
&\partial_t u\in L^2(0,T;H^{-1}(\R^2)),\quad u\log u\in L^\infty(0,T;L^1(\R^2)), \\
&u\in L^2(0,T;H^1(\R^2))\cap L^\infty(0,T;L^p(\R^2))\;(1\leq p<\infty),\quad |x|^2u\in L^\infty(0,T;L^1(\mathbb{R}^2))\\
&v\in L^2(0,T;W^{3,2}(\R^2))\cap L^\infty(0,T;W^{2,p}(\R^2))\;(1< p<\infty),
\end{align*}
the initial condition $u(0)=u_0$ in $L^2(\R^2)$, $(u,v)$ verifies the weak formulation
\begin{align*}
\int_0^T\langle\partial_t u,\varphi\rangle\,ds
+\int_0^T\int_{\R^2}\nabla(e^{-v}u+u)\cdot\nabla \varphi\,dxds=0,
\end{align*}
for any $\varphi\in L^2(0,T;H^1(\R^2))$.
Here $\langle\cdot,\cdot\rangle$ is the dual product between $H^{-1}(\R^2)$ and $H^1(\R^2)$.
In addition, $(u,v)$ solves
\begin{align*}
-\Delta v+v=\chi u\quad \mbox{ in the strong sense of }L^2(0,T;H^1(\R^2))\cap L^\infty(0,T;L^p(\R^2))\;(1< p<\infty).
\end{align*}
\end{definition}

The first main result of this paper is the well-posedness of the problem \eqref{singalKellerSegel}
and the error estimates between the local problem \eqref{singalKellerSegel} and the non-local problem \eqref{kellersegelmedium}.
\begin{theorem}
 \label{theorempde}
For any given $T>0$, let $\chi<4/c_*$ and the probability density $u_0\geq 0$ satisfy
\begin{gather*}
 u_0\log u_0\in L^1(\R^2),\quad
u_0\in L^1(\R^2, |x|^2\,dx)\cap L^p(\R^2)\;(1\leq p<\infty),
\end{gather*}
where $c_*$ is the optimal constant in the Sobolev inequality: $\|w\|_{L^4(\R^2)}^4\leq c_*\|w\|_{L^2(\R^2)}^2\|\nabla w\|_{L^2(\R^2)}^2$.
Then the problem \eqref{singalKellerSegel} possesses the weak solutions $(u,v)$ in $[0,T]$ in the sense of Definition \ref{definition}.
In addition, $(u,v)$ satisfies, for a.e. $t\in [0,T]$,
\begin{align}\label{ad1}
\mathcal{F}(u,v)(t)+\int_0^t\int_{\R^2} ue^{-v}|\nabla(\log u-v)|^2\,dxds\leq \mathcal{F}(u,v)(0),
\end{align}
where $\mathcal{F}(u,v)(t)=\int_{\R^2}(u\log u+\frac{|\nabla v|^2}{2}+\frac{v^2}{2}-\frac{uv}{2}-\chi\frac{uv}{2})\,dx$, and for any $p\in [2,\infty)$
\begin{align}\label{ad2}
\int_{\mathbb{R}^2}&u^p(t,x)\,dx
+\frac{2(p-1)}{p}\int_0^t\int_{\mathbb{R}^2}e^{-v}|\nabla u^{p/2}|^2\,dxds
+\frac{2(p-1)}{p}\int_0^t\int_{\mathbb{R}^2}|\nabla|\nabla v|^p|^2\,dxds\nonumber\\
&+2p\int_0^t\int_{\mathbb{R}^2}|\nabla v|^{2p}\,dxds
+\frac{2(p-1)}{p}\int_0^t\int_{\mathbb{R}^2}|\nabla u^{p/2}|^2\,dxds\leq \int_{\R^2}u_0^p\,dx+C,
\end{align}
where $C$ is a positive constant dependent on $u_0$.
Moreover, it holds
\begin{align}\label{a1}
\|u^\varepsilon- u\|_{L^{\infty} (0,T; L^1 (\R^2) ) }^2
+\|u^\varepsilon- u\|_{L^{\infty} (0,T; L^2 (\R^2) ) }
+\|\nabla(u^\varepsilon-u)\|_{L^2((0,T)\times\R^2)}
\leq C\varepsilon,
\end{align}
where $C$ is a positive constant independent of $\varepsilon$.
\end{theorem}

The difficulty in studying the well-posedness of the local system \eqref{singalKellerSegel} comes from the possible aggregation effect when the gradient of signal concentration $\nabla v$ becomes unbounded.
Since the well-established methods for the bounded domain case introduced for example in \cite{FJ2,FJ3,JL21} do not apply for the whole space case, new estimates with the benefit from linear diffusion need to be developed. In the current system, we keep the diffusion effect in order to dominate the aggregation effect for the reasonable initial data. In the whole space case, a positive lower bound for $v$ can not be expected, therefore one meets substantial difficulties to get well-posedness result without linear diffusion due to the aggregation effect. Actually, with linear diffusion, system \eqref{singalKellerSegel} has a similar Lyapunov functional to the one popularly used for the classical Keller-Segel model. We refer the derivation of the energy inequality fulfilled by the Keller-Segel model to \cite{CCC}.

The main proof strategy for Theorem \ref{theorempde} includes the following steps. Under the condition that the strength of signal is relatively weak, i.e. $\chi<4/c_*$, we can show that the entropy functional $\mathcal{F}(u,v)$ decays in time. Thus inspired by the entropy minimization method and Onofri's inequality in \cite{CCC}, we obtain the estimate that $u\log u\in L^1(\R^2)$. This, together with the generalization of the Gagliardo-Nirenberg inequality, provides the basis to obtain further the $L^p (\R^2)$ estimates for $u$, which makes it possible to obtain the upper bound for $v$ by using the classical elliptic theory.
The regularity we obtained for $(u,v)$ allows us to deduce the error estimates for $u^\varepsilon-u$. Actually, these estimates are done for $u^\epsilon$ uniformly, we refer the detailed estimates to Section 2. Then the well-posedness of \eqref{singalKellerSegel} is done by using the standard compactness argument.

Based on these, we obtain the second main result of this  paper, i.e. the following propagation of chaos result:
\begin{theorem}
\label{mean-field}
Under the assumptions of Theorem \ref{theorempde},
the problem \eqref{generalized_particle_model} has a unique square integrable solution $\hat{X}_i$ with $u$ as the density of its distribution.
Furthermore, for $N$ large enough,
there exists a constant $C> 0$ independent of $N$ such that
\begin{align}\label{r3}
\max_{i = 1, \dots, N} \ept \Big( \sup_{t \in [0,T]} \big| X^\varepsilon_{N,i} - \hat{X}_i \big|^2(t) \Big) \leq  C \varepsilon^{2},
\end{align}
where the cut-off parameter
$\varepsilon=(\lambda\log N)^{-\frac{1}{4}}$ with $0<\lambda\ll 1$.
\end{theorem}
The proof follows the well-established idea of proving mean-field limit for the moderate interacting systems, for example in \cite{philipowski2007interacting,chen2021rigorous,CGL}. Since the diffusion coefficients in the stochastic models \eqref{generalized_regularized_particle_model}, \eqref{generalized_intermediate_particle_model}, and \eqref{generalized_particle_model} depend nonlinearly on the interactions between the individuals,
 therefore it inquires the application of the Burkholder-Davis-Gundy inequality, as has been applied in \cite{chen2021rigorous} in deriving the cross-diffusion system for the logarithmic scaling. Because of this point, a propagation of chaos result with algebraic scaling such as in \cite{lazarovici2017mean}, where the maximum norm of trajectory has been used, is a challenging topic for future research.
We split the proof of Theorem \ref{mean-field} in three steps and introduce an intermediate particle system.
First, following similar strategy given in \cite{chen2021rigorous}, thanks to It\^o’s formula and the duality method, the solvabilities of stochastic differential equations \eqref{generalized_intermediate_particle_model}
and \eqref{generalized_particle_model} are given.
Second, we estimate the mean square error of the difference $X^\varepsilon_{N,i}-\bar{X}^\varepsilon_i$,
where $\bar{X}^\varepsilon_i$ is the solution to the intermediate particle system.
Finally, the estimate for $\bar{X}^\varepsilon_i-\hat{X}_i$ is done by exploiting the error estimates for $u^\varepsilon-u$.

As a consequence of Theorem \ref{mean-field}, we obtain the weak propagation of chaos result as a corollary, see \cite{sznitman1991topics}.
\begin{corollary}[Propagation of chaos in the weak sense]
\label{propagation_of_chaos}
Let $k\in\mathbb{N}$ and $P^\varepsilon_{N,k}(t)$ be the $k$-marginal of the joint distribution of  $(X^\varepsilon_{N,1},\cdots,X^\varepsilon_{N,N})$.
If the assumptions of Theorem \ref{mean-field} hold, then
\begin{align*}
  P^\varepsilon_{N,k}(t) \ \text{converges weakly to} \ P^{\otimes k}(t)\quad  \mbox{as } N \rightarrow \infty,
\end{align*}
where $P(t)$ is absolutely continuous with respect to the Lebesgue measure and has the probability density  $u(t,x)$.
\end{corollary}

Corollary \ref{propagation_of_chaos} gives the propagation of chaos in the weak sense, but we can further derive the quantitative propagation of chaos result in the strong sense by applying the relative entropy method \cite{jabin2018quantitative} recently improved in \cite{CHH}.
The main difficulty lies in the $L^\infty$ estimate for $\nabla\log u^\varepsilon$.
We use a modified Moser iteration and properties of the elliptic equations to get the higher regularity for $(u^\varepsilon,v^\varepsilon)$,
which allow us to obtain that there exists $T^*\in (0,T)$ such that the $L^\infty$ norm of $\nabla\log u^\varepsilon$ is bounded on a time interval $(0,T^*)$.
\begin{theorem}[Propagation of chaos in the strong sense]
\label{relative}
Under the assumptions of Theorem \ref{theorempde},
we suppose that $u_0\in L^\infty(\mathbb{R}^2)$ and
$\nabla\log u_0\in W^{1,p}(\mathbb{R}^2)$ $(p>2)$,
then there exist $T^*\in (0,T)$ and the constant $C>0$ independent of $\varepsilon$
such that
$\|\nabla\log u^\varepsilon\|_{L^\infty(0,T^*;W^{1,p}(\mathbb{R}^2))}
\leq C$.
Furthermore, let $k\in\N$ and $u^\varepsilon_{N,k}(t,x_1,\cdots,x_k)$ be the $k$-th marginal density of the joint density $u^\varepsilon_N(t,x_1,\cdots,x_N)$ of $(X^\varepsilon_{N,i})_{1\leq i\leq N}$.
Then we have
\begin{align*}
\|u^\varepsilon_{N,k}-u^{\otimes k}\|^2_{L^\infty(0,T^*;L^1(\mathbb{R}^{2k}))}\leq C(k) \varepsilon.
\end{align*}
\end{theorem}
\begin{remark}
Such a function $u_0$ does indeed exist. We give the following example:
Let $\varphi \in C^\infty (\R^2)$ verify $0\leq  \varphi \leq 1$ on $\R^2$, $\varphi = 0$ on $B_1(0)$ and $\varphi =1$ on $B_2(0)^c$. Then the function
$$ u_0(x) = \exp \left(- (1+|x|)^\alpha \varphi(x) \right) $$
satisfies all assumptions of Theorem \ref{relative} for $\alpha \in (0,1-\frac{2}{p})$.
\end{remark}
\begin{remark}
In fact, we can achieve the strong $L^\infty(0,T;L^1(\mathbb{R}^{2k}))$ convergence for the propagation of chaos
as long as $\|\nabla\log u^\varepsilon\|_{L^\infty((0,T)\times\mathbb{R}^2)}=O(\frac{1}{\varepsilon^\vartheta})$ with $\vartheta>0$.
\end{remark}
The rest of the paper is organized as follows. The proof of Theorem \ref{theorempde} is presented in Section \ref{final}.
The mean-field limit and propagation of chaos results are given in Sections \ref{The_proof_of_Theorem_3}.

\section{Existence and error estimates}
\label{final}

This section is devoted to the analysis of partial differential equations \eqref{singalKellerSegel} and \eqref{kellersegelmedium},
and the error estimates between their solutions.
For any given $\varepsilon>0$, it is standard to obtain a global non-negative classical solution $(u^\varepsilon,{v^\varepsilon})$
to the non-local system \eqref{kellersegelmedium}. Therefore, for simplicity, we only generate the uniform in $\varepsilon$ estimates for $u^\varepsilon$ and proceed the compactness argument.

\subsection{Well-posedness of problem \eqref{singalKellerSegel}}

We start with the energy inequality satisfied by $(u^\varepsilon,{v^\varepsilon})$.
\begin{proposition}\label{entropy1}{\rm{(Energy inequality for $(u^\varepsilon,{v^\varepsilon})$)}}.
Let $(u^\varepsilon,{v^\varepsilon})$ be a smooth non-negative solution to the system \eqref{kellersegelmedium}. Then $(u^\varepsilon,{v^\varepsilon})$ satisfies
\begin{align}\label{1}
\frac{d}{dt}F(u^\varepsilon,{v^\varepsilon})(t)
+\int_{\mathbb{R}^2}u^\varepsilon e^{-{v^\varepsilon}}|\nabla(\log u^\varepsilon-{v^\varepsilon}+\log(1+|x|^2))|^2\,dx
+(4-c_*\chi)\int_{\R^2}|\nabla\sqrt{u^\varepsilon}|^2\,dx\leq 9,
\end{align}
where
\begin{align*}
F(u^\varepsilon,{v^\varepsilon})(t):
=\int_{\mathbb{R}^2}\Big(u^\varepsilon\log u^\varepsilon
+\frac{|\nabla {v^\varepsilon}|^2}{2}+\frac{{v^\varepsilon}^2}{2}
-\frac{u^\varepsilon {v^\varepsilon}}{2}-\chi\frac{u^\varepsilon {v^\varepsilon}\ast j^\varepsilon }{2}
-u^\varepsilon\log H\Big)\,dx.
\end{align*}
Here $H(x)= 1/\pi(1+|x|^2)^2$.
\end{proposition}
\begin{proof}
Notice that the first equation of \eqref{kellersegelmedium} can be written as
\begin{align}\label{27}
\partial_t u^\varepsilon=\nabla\cdot( u^\varepsilon e^{-{v^\varepsilon}}\nabla(\log u^\varepsilon-{v^\varepsilon}))+\Delta u^\varepsilon.
\end{align}
A direct entropy identity of \eqref{27} can be easily obtained by using $(\log u^\varepsilon-{v^\varepsilon})$ as a test function:
\begin{align}\label{a2}
\frac{d}{dt}&\int_{\R^2}u^\varepsilon\log u^\varepsilon\,dx-\int_{\R^2}\partial_t u^\varepsilon {v^\varepsilon}\,dx
+\int_{\R^2}u^\varepsilon e^{-{v^\varepsilon}}|\nabla(\log u^\varepsilon-{v^\varepsilon})|^2\,dx\nonumber\\
&+4\int_{\R^2}|\nabla\sqrt{u^\varepsilon}|^2\,dx=\int_{\R^2}\nabla u^\varepsilon\cdot\nabla {v^\varepsilon}\,dx,
\end{align}
where
\begin{align}\label{a3}
\int_{\R^2}\partial_t u^\varepsilon {v^\varepsilon}\,dx
=&\int_{\R^2\times\R^2}\partial_t u^\varepsilon(t,x)u^\varepsilon(t,y)\Phi\ast j^\varepsilon(x-y)\,dydx\nonumber\\
=&\frac{1}{2}\frac{d}{dt}\int_{\R^2\times\R^2}u^\varepsilon(t,x)u^\varepsilon(t,y)\Phi\ast j^\varepsilon(x-y)\,dydx\nonumber\\
=&\frac{1}{2}\frac{d}{dt}\int_{\R^2}u^\varepsilon(t,x){v^\varepsilon}(t,x)\,dx.
\end{align}
From the second equation of \eqref{kellersegelmedium}, we obtain by taking $\partial_t {v^\varepsilon}$ as a test function that
\begin{align}\label{a4}
\frac{1}{2}\frac{d}{dt}\int_{\R^2}(|\nabla {v^\varepsilon}|^2+{v^\varepsilon}^2)\,dx
-\chi\int_{\R^2}\partial_t {v^\varepsilon} u^\varepsilon\ast j^\varepsilon\,dx=0.
\end{align}
We proceed similarly as \eqref{a3} to obtain that
\begin{align}\label{a5}
-\chi\int_{\R^2}\partial_t {v^\varepsilon} u^\varepsilon\ast j^\varepsilon\,dx
&=-\chi\int_{\R^2\times\R^2}\partial_t u^\varepsilon\ast j^\varepsilon(t,y)u^\varepsilon\ast j^\varepsilon(t,x)\Phi(x-y)\,dydx\nonumber\\
&=-\frac{\chi}{2}\frac{d}{dt}\int_{\R^2}u^\varepsilon(t,x){v^\varepsilon}\ast j^\varepsilon(t,x)\,dx.
\end{align}
Combining \eqref{a2}-\eqref{a5}, we infer that
\begin{align}\label{2}
\frac{d}{dt}&\int_{\mathbb{R}^2}\Big(u^\varepsilon\log u^\varepsilon+\frac{|\nabla {v^\varepsilon}|^2}{2}
+\frac{{v^\varepsilon}^2}{2}-\frac{u^\varepsilon {v^\varepsilon}}{2}-\chi\frac{u^\varepsilon {v^\varepsilon}\ast j^\varepsilon}{2} \Big)\,dx\nonumber\\
&+\int_{\mathbb{R}^2}u^\varepsilon e^{-{v^\varepsilon}}|\nabla(\log u^\varepsilon-{v^\varepsilon})|^2\,dx
+4\int_{\mathbb{R}^2}|\nabla\sqrt{u^\varepsilon}|^2\,dx
=\int_{\mathbb{R}^2}\nabla u^\varepsilon\cdot\nabla {v^\varepsilon}\,dx.
\end{align}
The second equation of \eqref{kellersegelmedium}, the Gagliardo-Nirenberg inequality and the non-negativity of $(u^\varepsilon,{v^\varepsilon})$ imply
\begin{align}\label{3}
\int_{\mathbb{R}^2}&\nabla u^\varepsilon\cdot \nabla {v^\varepsilon}\,dx
=-\int_{\mathbb{R}^2}u^\varepsilon {v^\varepsilon}\,dx+\chi\int_{\mathbb{R}^2}u^\varepsilon u^\varepsilon\ast j^\varepsilon\,dx\nonumber\\
&\leq \frac{\chi}{2}\|u^\varepsilon\|_{L^2(\mathbb{R}^2)}^2
+\frac{\chi}{2}\|u^\varepsilon\ast j^\varepsilon\|_{L^2(\mathbb{R}^2)}^2
\leq \chi\|u^\varepsilon\|_{L^2(\mathbb{R}^2)}^2
= \chi\|\sqrt{u^\varepsilon}\|_{L^4(\mathbb{R}^2)}^4\nonumber\\
&\leq c_*\chi\|\sqrt{u^\varepsilon}\|_{L^2(\R^2)}^2\|\nabla \sqrt{u^\varepsilon}\|_{L^2(\mathbb{R}^2)}^2
\leq c_*\chi\|\nabla \sqrt{u^\varepsilon}\|_{L^2(\mathbb{R}^2)}^2.
\end{align}
Plugging \eqref{3} into \eqref{2}, we arrive at
\begin{align}\label{28}
\frac{d}{dt}\mathcal{F}^\varepsilon(u^\varepsilon,{v^\varepsilon})(t)+\int_{\R^2}u^\varepsilon e^{-{v^\varepsilon}}|\nabla (\log u^\varepsilon-{v^\varepsilon})|^2\,dx
+(4-c_*\chi)\int_{\R^2}|\nabla\sqrt{u^\varepsilon}|^2\,dx
\leq 0,
\end{align}
where $\mathcal{F}^\varepsilon(u^\varepsilon,{v^\varepsilon})(t):
=\int_{\mathbb{R}^2}\big(u^\varepsilon\log u^\varepsilon
+\frac{|\nabla {v^\varepsilon}|^2}{2}+\frac{{v^\varepsilon}^2}{2}
-\frac{u^\varepsilon {v^\varepsilon}}{2}-\chi\frac{u^\varepsilon {v^\varepsilon}\ast j^\varepsilon }{2}\big)\,dx$.

We multiply \eqref{27} by $\log H(x)$ $(H(x)= 1/\pi(1+|x|^2)^2)$ to obtain that
\begin{align}\label{4}
\frac{d}{dt}\int_{\mathbb{R}^2}u^\varepsilon\log H\,dx
=&-\int_{\mathbb{R}^2} u^\varepsilon e^{-{v^\varepsilon}}\nabla(\log u^\varepsilon-{v^\varepsilon})\cdot\nabla\log H\,dx
+\int_{\mathbb{R}^2}u^\varepsilon\Delta \log H\,dx\nonumber\\
=&2\int_{\mathbb{R}^2}u^\varepsilon e^{-{v^\varepsilon}}\nabla (\log u^\varepsilon-{v^\varepsilon})\cdot\nabla\log (1+|x|^2)\,dx
-8\int_{\mathbb{R}^2}\frac{u^\varepsilon}{(1+|x|^2)^2}\,dx.
\end{align}
Combining \eqref{28} and \eqref{4}, it holds
\begin{align*}
&\frac{d}{dt}F(u^\varepsilon,{v^\varepsilon})(t)
+\int_{\mathbb{R}^2}u^\varepsilon e^{-{v^\varepsilon}}|\nabla(\log u^\varepsilon-{v^\varepsilon}+\log(1+|x|^2))|^2\,dx
+(4-c_*\chi)\int_{\R^2}|\nabla\sqrt{u^\varepsilon}|^2\,dx\\
\leq & 8\int_{\mathbb{R}^2}\frac{u^\varepsilon}{(1+|x|^2)^2}\,dx
+\int_{\mathbb{R}^2}u^\varepsilon e^{-{v^\varepsilon}}|\nabla\log (1+|x|^2)|^2\,dx\leq 9,
\end{align*}
where we have used
$|\nabla \log(1+|x|^2)|=\big|\frac{2x}{1+|x|^2}\big|\leq 1$
and $\int_{\mathbb{R}^2}u^\varepsilon\,dx=1$.
\end{proof}

The entropy appeared in \eqref{1} shares a similar structure to the entropy for the classical two-dimensional parabolic-parabolic Keller-Segel system in \cite{CCC}. This motivates us to apply the three lemmas in the appendix, which has played the key roles in deriving {\it a priori} estimates from entropy. Namely, the energy inequality \eqref{1} and Lemmas \ref{entropymini}-\ref{f-} allow us to derive the following estimates:
\begin{proposition}{\rm{(Estimates for $(u^\varepsilon,{v^\varepsilon})).$}}\label{energy1}
Let $(u^\varepsilon,{v^\varepsilon})$ be a smooth non-negative solution to the problem \eqref{kellersegelmedium}. Then for any $t\in (0,\infty)$,
\begin{gather}
\int_{\mathbb{R}^2}(u^\varepsilon {v^\varepsilon}+u^\varepsilon {v^\varepsilon}\ast j^\varepsilon)(t,x)\,dx\leq C(1+t),\label{5}\\
\|{v^\varepsilon}(t)\|_{H^1(\mathbb{R}^2)}\leq C(1+t),\label{6}\\
F(u^\varepsilon,{v^\varepsilon})(t)\geq -C,\label{7}
\end{gather}
where $C$ appeared in this subsection is a positive constant independent of $\varepsilon$.
\end{proposition}
\begin{proof}
Since $(u^\varepsilon,v^\varepsilon)$ is the non-negative classical solution to the non-local problem \eqref{kellersegelmedium},
for any fixed $\varepsilon>0$, it holds $v^\varepsilon\in L^\infty(0,T;H^1(\R^2))$ and $v^\varepsilon\in L^\infty(0,T;L^1(\R^2,H(x)dx))$,
which follows from the fact that $H\le 1$ on $\R^2$ and $v^\varepsilon \in L^\infty(0,T;L^1(\R^2))$.
Together with the $L^1(\R^2)$ bound for $j^\varepsilon$ and Young's convolution inequality,
we infer that $v^\varepsilon\ast j^\varepsilon\in L^\infty(0,T;H^1(\R^2))$ and $v^\varepsilon\ast j^\varepsilon\in L^\infty(0,T;L^1(\R^2,H(x)dx))$
for any fixed $\varepsilon>0$.
These allow us to use Lemma \ref{Onofri} to infer that
\begin{align*}
&\int_{\R^2}\exp\Big\{\Big(\frac{1}{2}+\delta\Big){v^\varepsilon}+\chi\Big(\frac{1}{2}+\delta\Big){v^\varepsilon}\ast j^\varepsilon\Big\}H(x)\,dx\\
\leq &\exp\Big\{\int_{\R^2}\Big(\Big(\frac{1}{2}+\delta\Big){v^\varepsilon}+\chi\Big(\frac{1}{2}+\delta\Big){v^\varepsilon}\ast j^\varepsilon \Big)H(x)\,dx
+\frac{1}{16\pi}\int_{\R^2}\Big|\Big(\frac{1}{2}+\delta\Big)\nabla {v^\varepsilon}
+\chi\Big(\frac{1}{2}+\delta\Big)\nabla {v^\varepsilon}\ast j^\varepsilon\Big|^2\,dx\Big\},
\end{align*}
where $\delta>0$ will be determined later.
Clearly, the integral on the left-hand side of the above inequality is nothing but the $L^1(\mathbb{R}^2)$ norm of
$\exp\{(\frac{1}{2}+\delta){v^\varepsilon}+\chi(\frac{1}{2}+\delta){v^\varepsilon}\ast j^\varepsilon+\log H\}$,
which implies that we can take
$\hat{w}=e^{\phi}/\int_{\mathbb{R}^2}e^\phi dx$ with
$\phi=(\frac{1}{2}+\delta){v^\varepsilon}+\chi(\frac{1}{2}+\delta){v^\varepsilon}\ast j^\varepsilon+\log H$ in Lemma \ref{entropymini}
to get
\begin{align*}
\mathcal{E}\Big(u^\varepsilon;\Big(\frac{1}{2}+\delta\Big)&{v^\varepsilon}+\chi\Big(\frac{1}{2}+\delta\Big){v^\varepsilon}\ast j^\varepsilon+\log H\Big)
=\mathcal{E}(u^\varepsilon;\phi)\geq \mathcal{E}(\hat{w};\phi)\\
=& \int_{\R^2}\frac{e^\phi}{\int_{\R^2}e^\phi\,dx}\Big(\log\Big(\frac{e^\phi}{\int_{\R^2}e^\phi\,dx}\Big)-\phi  \Big)\,dx \\
=&-\log \int_{\mathbb{R}^2}\exp\Big\{\Big(\frac{1}{2}+\delta\Big){v^\varepsilon}+\chi\Big(\frac{1}{2}+\delta\Big){v^\varepsilon}\ast j^\varepsilon\Big\}H(x)\,dx\\
\geq& -\Big(\frac{1}{2}+\delta\Big)\int_{\mathbb{R}^2}({v^\varepsilon}+\chi{v^\varepsilon}\ast j^\varepsilon)H\,dx
-\frac{(1+\chi^2)(1+2\delta)^2}{16\pi}\int_{\mathbb{R}^2}|\nabla {v^\varepsilon}|^2\,dx,
\end{align*}
where we have used Lemma \ref{Onofri}.
Therefore,
\begin{align*}
 F(u^\varepsilon,{v^\varepsilon})(t)
=&\mathcal{E}\Big(u^\varepsilon;\Big(\frac{1}{2}+\delta\Big){v^\varepsilon}
+\chi\Big(\frac{1}{2}+\delta\Big){v^\varepsilon}\ast j^\varepsilon+\log H \Big)+\delta\int_{\mathbb{R}^2}u^\varepsilon {v^\varepsilon}\,dx\\
&+\delta\chi\int_{\mathbb{R}^2}u^\varepsilon {v^\varepsilon}\ast j^\varepsilon\,dx
+\frac{1}{2}\int_{\mathbb{R}^2}(|\nabla {v^\varepsilon}|^2+{v^\varepsilon}^2)\,dx\\
\geq &-\Big(\frac{1}{2}+\delta \Big)\int_{\mathbb{R}^2}({v^\varepsilon}+\chi{v^\varepsilon}\ast j^\varepsilon)H\,dx
+\delta\int_{\mathbb{R}^2}u^\varepsilon({v^\varepsilon}+\chi{v^\varepsilon}\ast j^\varepsilon)\,dx\\
&+\frac{1}{2}\Big(1-\frac{(1+\chi^2)(1+2\delta)^2}{8\pi} \Big)\int_{\mathbb{R}^2}|\nabla {v^\varepsilon}|^2\,dx
+\frac{1}{2}\int_{\mathbb{R}^2}{v^\varepsilon}^2\,dx.
\end{align*}
In view of H\"older's inequality, it holds
\begin{align*}
F(u^\varepsilon,{v^\varepsilon})(t)\geq &-\frac{1+2\delta}{2\delta'}\int_{\mathbb{R}^2}H^2\,dx
+\frac{1}{2}\Big(1-\frac{(1+\chi^2)(1+2\delta)\delta'}{2} \Big)\int_{\mathbb{R}^2}{v^\varepsilon}^2\,dx\\
&+\delta\int_{\mathbb{R}^2}u^\varepsilon({v^\varepsilon}+\chi{v^\varepsilon}\ast j^\varepsilon)\,dx
+\frac{1}{2}\Big(1-\frac{(1+\chi^2)(1+2\delta)^2}{8\pi} \Big)\int_{\mathbb{R}^2}|\nabla {v^\varepsilon}|^2\,dx.
\end{align*}
Since $\chi<4/c_*<\sqrt{8\pi-1}$, we can choose $\delta>0$ small enough such that $\chi^2<\frac{8\pi}{(1+2\delta)^2}-1 $,
and $\delta'>0$ small enough such that $(1+\chi^2)(1+2\delta)\delta'<2$.
Together with the energy inequality \eqref{1} show \eqref{5}-\eqref{7}.
\end{proof}
\begin{proposition}{\rm{(Estimates for $u^\varepsilon$).}}\label{energy}
Let $(u^\varepsilon,{v^\varepsilon})$ be a smooth non-negative solution to the problem \eqref{kellersegelmedium}. Then for any $t\in (0,\infty)$,
\begin{gather}
\int_{\mathbb{R}^2}|x|^2u^\varepsilon (t,x)\,dx+\int_{\mathbb{R}^2}u^\varepsilon|\log u^\varepsilon|(t,x)\,dx\leq C(1+t).\label{8}
\end{gather}
\end{proposition}
\begin{proof}
We multiply \eqref{kellersegelmedium} by $|x|^2$ to get
\begin{align*}
\int_{\mathbb{R}^2}|x|^2u^\varepsilon (t,x)\,dx
=4\int_0^t\int_{\mathbb{R}^2}(e^{-v^\varepsilon}u^\varepsilon+u^\varepsilon)\,dxds
+\int_{\mathbb{R}^2}|x|^2u_0 \,dx\leq C(1+t).
\end{align*}
Therefore we obtain the first claim of \eqref{8}.

We notice that
\begin{align*}
\int_{\mathbb{R}^2}\Big(u^\varepsilon\log u^\varepsilon-\frac{u^\varepsilon {v^\varepsilon}}{2}-\chi\frac{u^\varepsilon {v^\varepsilon}\ast j^\varepsilon}{2}
-u^\varepsilon\log H\Big)\,dx
\leq F(u^\varepsilon,{v^\varepsilon})(t)\leq C(1+t).
\end{align*}
It follows that
\begin{align*}
\int_{\mathbb{R}^2}u^\varepsilon\log u^\varepsilon\,dx
\leq C(1+t)
+\int_{\mathbb{R}^2}\Big(\frac{u^\varepsilon {v^\varepsilon}}{2}+\chi\frac{u^\varepsilon {v^\varepsilon}\ast j^\varepsilon}{2}
+u^\varepsilon\log H\Big)\,dx\leq C(1+t).
\end{align*}
From
$\int_{\R^2}e^{-|x|^2}\,dx<\infty$,
we can take $\phi(x)=-|x|^2$ in Lemma \ref{f-} to get
\begin{align*}
\int_{\mathbb{R}^2}u^\varepsilon(\log u^\varepsilon)_-\,dx
\leq C+\int_{\mathbb{R}^2}|x|^2u^\varepsilon \,dx\leq C(1+t).
\end{align*}
Therefore,
\begin{align*}
\int_{\mathbb{R}^2}u^\varepsilon(\log u^\varepsilon)_+\,dx
\leq \int_{\mathbb{R}^2}u^\varepsilon\log u^\varepsilon\,dx+\int_{\mathbb{R}^2}u^\varepsilon(\log u^\varepsilon)_-\,dx
\leq C(1+t).
\end{align*}
Hence we complete the proof of the second claim of \eqref{8}.
\end{proof}

Motivated by the idea given in \cite{TW2}, by using the generalized Gagliardo-Nirenberg inequality Lemma \ref{GN}, detailed proof in the whole space case is given in the Appendix, we derive the $L^p$ $(1\leq p<\infty)$ estimates.
\begin{proposition}{\rm{($L^p$ estimates for $(u^\varepsilon,{v^\varepsilon})$).}}\label{pLp}
Let $({u^\varepsilon},{v^\varepsilon})$ be a smooth non-negative solution to the problem \eqref{kellersegelmedium}. Then for any $p\in [2,\infty)$, it holds
\begin{align}\label{Lp}
\frac{d}{dt}&\int_{\mathbb{R}^2}{u^\varepsilon}^p\,dx
+\frac{2(p-1)}{p}\int_{\mathbb{R}^2}e^{-{v^\varepsilon}}|\nabla {u^\varepsilon}^{p/2}|^2\,dx
+\frac{2(p-1)}{p}\int_{\mathbb{R}^2}|\nabla|\nabla {v^\varepsilon}|^p|^2\,dx\nonumber\\
&+2p\int_{\mathbb{R}^2}|\nabla {v^\varepsilon}|^{2p}\,dx
+\frac{2(p-1)}{p}\int_{\mathbb{R}^2}|\nabla {u^\varepsilon}^{p/2}|^2\,dx\leq C.
\end{align}
\end{proposition}
\begin{proof}
We multiply the first equation of \eqref{kellersegelmedium} by $p{u^\varepsilon}^{p-1}$ to get
\begin{align}\label{9}
&\frac{d}{dt}\int_{\mathbb{R}^2}{u^\varepsilon}^p\,dx\nonumber\\
=&-p(p-1)\int_{\mathbb{R}^2}e^{-{v^\varepsilon}}{u^\varepsilon}^{p-2}|\nabla {u^\varepsilon}|^2\,dx
+p(p-1)\int_{\mathbb{R}^2}e^{-{v^\varepsilon}}{u^\varepsilon}^{p-1}\nabla {v^\varepsilon}\cdot\nabla {u^\varepsilon}\,dx
-p(p-1)\int_{\mathbb{R}^2}{u^\varepsilon}^{p-2}|\nabla {u^\varepsilon}|^2\,dx\nonumber\\
\leq &-\frac{2(p-1)}{p}\int_{\mathbb{R}^2}e^{-{v^\varepsilon}}|\nabla {u^\varepsilon}^{p/2}|^2\,dx
+\frac{p(p-1)}{2}\int_{\mathbb{R}^2}e^{-{v^\varepsilon}}{u^\varepsilon}^p|\nabla {v^\varepsilon}|^2\,dx
-\frac{4(p-1)}{p}\int_{\mathbb{R}^2}|\nabla {u^\varepsilon}^{p/2}|^2\,dx.
\end{align}

Taking the first derivative on both sides of the second equation of \eqref{kellersegelmedium},
then multiplying the result by $2p|\nabla {v^\varepsilon}|^{2p-2}\nabla {v^\varepsilon}$, we obtain that
\begin{align}\label{10}
-2p\int_{\mathbb{R}^2}\nabla\Delta {v^\varepsilon}\cdot|\nabla {v^\varepsilon}|^{2p-2}\nabla {v^\varepsilon}\,dx
+2p\int_{\mathbb{R}^2}|\nabla {v^\varepsilon}|^{2p}\,dx
=2p\chi\int_{\mathbb{R}^2}|\nabla {v^\varepsilon}|^{2p-2}\nabla {u^\varepsilon}\ast j^\varepsilon\cdot \nabla {v^\varepsilon}\,dx.
\end{align}
Based on the fact that $2\nabla\Delta {v^\varepsilon}\cdot\nabla {v^\varepsilon}=\Delta|\nabla {v^\varepsilon}|^2-2|D^2{v^\varepsilon}|^2$
with $|D^2v^\varepsilon|^2:=\sum_{i,j=1}^2 |D_{ij}v^\varepsilon|^2$,
the first term on the left-hand side of \eqref{10} can be divided into two terms
\begin{align}\label{11}
-2p\int_{\mathbb{R}^2}\nabla \Delta {v^\varepsilon}\cdot |\nabla {v^\varepsilon}|^{2p-2}\nabla {v^\varepsilon}\,dx
&=-p\int_{\mathbb{R}^2}\Delta |\nabla {v^\varepsilon}|^2|\nabla {v^\varepsilon}|^{2p-2}\,dx
+2p\int_{\mathbb{R}^2}|\nabla {v^\varepsilon}|^{2p-2}|D^2{v^\varepsilon}|^2\,dx\nonumber\\
&=\frac{4(p-1)}{p}\int_{\mathbb{R}^2}|\nabla |\nabla {v^\varepsilon}|^p|^2\,dx
+2p\int_{\mathbb{R}^2}|\nabla {v^\varepsilon}|^{2p-2}|D^2{v^\varepsilon}|^2\,dx.
\end{align}
Taking into account the fact that $|\Delta {v^\varepsilon}|^2\leq 2|D^2 {v^\varepsilon}|^2$, the term on the right-hand side of \eqref{10} can be handled by
\begin{align}\label{12}
&2p\chi\int_{\mathbb{R}^2}|\nabla {v^\varepsilon}|^{2p-2}\nabla {u^\varepsilon}\ast j^\varepsilon\cdot \nabla {v^\varepsilon}\,dx\nonumber\\
=&-4p(p-1)\chi\int_{\mathbb{R}^2}|\nabla {v^\varepsilon}|^{2p-2}D^2 {v^\varepsilon} {u^\varepsilon}\ast j^\varepsilon\,dx
-2p\chi\int_{\mathbb{R}^2}|\nabla {v^\varepsilon}|^{2p-2}{u^\varepsilon}\ast j^\varepsilon\Delta {v^\varepsilon}\,dx\nonumber\\
\leq &p\int_{\R^2}|\nabla {v^\varepsilon}|^{2p-2}|D^2 {v^\varepsilon}|^2\,dx
+4p(p-1)^2\chi^2\int_{\R^2}|\nabla {v^\varepsilon}|^{2p-2}({u^\varepsilon}\ast j^\varepsilon)^2\,dx\nonumber\\
&+\frac{p}{2}\int_{\R^2}|\nabla {v^\varepsilon}|^{2p-2}|\Delta {v^\varepsilon}|^2\,dx
+2p\chi^2\int_{\R^2}|\nabla {v^\varepsilon}|^{2p-2}({u^\varepsilon}\ast j^\varepsilon)^2\,dx\nonumber\\
\leq &2p\int_{\mathbb{R}^2}|\nabla {v^\varepsilon}|^{2p-2}|D^2{v^\varepsilon}|^2\,dx
+\big(4p(p-1)^2+2p \big)\chi^2\int_{\mathbb{R}^2}|\nabla {v^\varepsilon}|^{2p-2}({u^\varepsilon}\ast j^\varepsilon)^2\,dx.
\end{align}
Inserting \eqref{11} and \eqref{12} into \eqref{10}, we have
\begin{align}\label{13}
\frac{4(p-1)}{p}\int_{\mathbb{R}^2}|\nabla |\nabla {v^\varepsilon}|^p|^2\,dx+2p\int_{\mathbb{R}^2}|\nabla {v^\varepsilon}|^{2p}\,dx
\leq \big(4p(p-1)^2+2p \big)\chi^2\int_{\mathbb{R}^2}|\nabla {v^\varepsilon}|^{2p-2}({u^\varepsilon}\ast j^\varepsilon)^2\,dx.
\end{align}

Combining \eqref{9} and \eqref{13} to infer that
\begin{align}\label{14}
&\frac{d}{dt}\int_{\mathbb{R}^2}{u^\varepsilon}^p\,dx+\frac{2(p-1)}{p}\int_{\mathbb{R}^2}e^{-{v^\varepsilon}}|\nabla {u^\varepsilon}^{p/2}|^2\,dx
+\frac{4(p-1)}{p}\int_{\mathbb{R}^2}|\nabla |\nabla {v^\varepsilon}|^p|^2\,dx\nonumber\\
&+2p\int_{\mathbb{R}^2}|\nabla {v^\varepsilon}|^{2p}\,dx+\frac{4(p-1)}{p}\int_{\mathbb{R}^2}|\nabla {u^\varepsilon}^{p/2}|^2\,dx\nonumber\\
\leq &\frac{p(p-1)}{2}\int_{\mathbb{R}^2}e^{-{v^\varepsilon}}{u^\varepsilon}^p|\nabla {v^\varepsilon}|^2\,dx
+\big(4p(p-1)^2+2p \big)\chi^2\int_{\mathbb{R}^2}|\nabla {v^\varepsilon}|^{2p-2}({u^\varepsilon}\ast j^\varepsilon)^2\,dx\nonumber\\
\leq &\nu_1\int_{\mathbb{R}^2}|\nabla {v^\varepsilon}|^{2p+2}\,dx+c_1\int_{\mathbb{R}^2}{u^\varepsilon}^{p+1}\,dx,
\end{align}
where $\nu_1>0$ will be determined later and $c_1$ is a positive constant dependent on $\nu_1$.

It is time to deal with the terms on the right-hand side of \eqref{14}.
It follows from \eqref{6} and the Gagliardo-Nirenberg inequality that
\begin{align*}
\nu_1\|\nabla {v^\varepsilon}\|^{2p+2}_{L^{2p+2}(\mathbb{R}^2)}&=\nu_1\||\nabla {v^\varepsilon}|^p \|^{2(p+1)/p}_{L^{2(p+1)/p}(\mathbb{R}^2)}
\leq c_3\nu_1\|\nabla|\nabla {v^\varepsilon}|^p\|^2_{L^2(\mathbb{R}^2)}\||\nabla {v^\varepsilon}|^p\|^{2/p}_{L^{2/p}(\mathbb{R}^2)}\\
&\leq c_3c_2\nu_1\|\nabla |\nabla {v^\varepsilon}|^p\|^2_{L^2(\mathbb{R}^2)}.
\end{align*}
Taking $\nu_1=\frac{2(p-1)}{pc_3c_2}$, we have
\begin{align}\label{15}
\nu_1\|\nabla {v^\varepsilon}\|^{2p+2}_{L^{2p+2}(\mathbb{R}^2)}
\leq\frac{2(p-1)}{p}\int_{\mathbb{R}^2}|\nabla |\nabla {v^\varepsilon}|^p|^2\,dx.
\end{align}
According to \eqref{8} and Lemma \ref{GN}, the second term on the right-hand side of \eqref{14} can be bounded by
\begin{align}\label{16}
c_1\int_{\mathbb{R}^2}{u^\varepsilon}^{p+1}\,dx
&=c_1\|{u^\varepsilon}^{p/2}\|^{2(p+1)/p}_{L^{2(p+1)/p}(\mathbb{R}^2)}\nonumber\\
&\leq c_1\nu_2\|\nabla {u^\varepsilon}^{p/2}\|^2_{L^2(\mathbb{R}^2)}
\big\|{u^\varepsilon}^{p/2}|\log {u^\varepsilon}^{p/2}|^{p/2}\big\|^{2/p}_{L^{2/p}(\mathbb{R}^2)}
+c_4\|{u^\varepsilon}^{p/2}\|^{2/p}_{L^{2/p}(\mathbb{R}^2)}\nonumber\\
&=\frac{c_1\nu_2p}{2}\|\nabla {u^\varepsilon}^{p/2}\|^2_{L^2(\mathbb{R}^2)}\int_{\R^2}{u^\varepsilon}|\log {u^\varepsilon}|\,dx
+c_4\int_{\R^2}{u^\varepsilon}\,dx\nonumber\\
&\leq \frac{c_1c_5\nu_2p}{2}\|\nabla {u^\varepsilon}^{p/2}\|^2_{L^2(\mathbb{R}^2)}+c_4\nonumber\\
&\leq \frac{2(p-1)}{p}\|\nabla {u^\varepsilon}^{p/2}\|^2_{L^2(\mathbb{R}^2)}+c_4,
\end{align}
where $\nu_2=\frac{4(p-1)}{p^2c_1c_5}$.
Inserting \eqref{15} and \eqref{16} into \eqref{14}, we have \eqref{Lp}.
\end{proof}

The estimates obtained in Proposition \ref{energy} and Proposition \ref{pLp} allow us to take the limit $\varepsilon\rightarrow 0$
in the weak formulation of problem \eqref{kellersegelmedium}.
\begin{proposition}{\rm{(Compactness argument).}}
There exists a subsequence $({u^\varepsilon},{v^\varepsilon})$ (not relabeled) satisfying
\begin{align}
&{u^\varepsilon}\rightharpoonup u\quad {\rm{in}}\quad L^2(0,T;H^1(\R^2)),\label{a6}\\
&{u^\varepsilon}\rightarrow  u\quad {\rm{in}}\quad L^p((0,T)\times\R^2),\quad p\in[2,\infty),\label{17}\\
&{v^\varepsilon}\stackrel{*}{\rightharpoonup}  v\quad {\rm{in}}\quad  L^2(0,T;W^{3,2}(\R^2))\cap L^\infty(0,T;W^{2,p}(\R^2)),\quad p\in(1,\infty),\label{18}\\
&{v^\varepsilon}\rightarrow v\quad {\rm{in}}\quad L^p((0,T)\times\R^2), \quad p\in [2,\infty ),  \label{19}\\
&e^{-{v^\varepsilon}}\rightarrow e^{-v}\quad {\rm{in}}\quad L^p((0,T)\times\R^2),\quad p\in[2,\infty).\label{24}
\end{align}
\end{proposition}
\begin{proof}
By \eqref{Lp} with $p=2$, we have \eqref{a6}. Due to \eqref{Lp} and the first equation of \eqref{kellersegelmedium}, it holds
\begin{align}\label{25}
\|\partial_t {u^\varepsilon}\|_{L^2(0,T;H^{-1}(\R^2))}\leq C.
\end{align}
Proposition \ref{energy} and Proposition \ref{pLp} show that $\|u^\varepsilon\|_{L^2(0,T;H^1(\R^2))} $ and $\|(1+|x|^2)u^\varepsilon\|_{L^\infty(0,T;L^1(\R^2))}$ are bounded. Since $H^1(\mathbb{R}^2)\cap L^1(\R^2,(1+|x|^2)dx)$ is embedded compactly in $L^2(\R^2)$, Aubin-Lions lemma implies that
\begin{align*}
{u^\varepsilon}\rightarrow u\quad {\rm{in}}\quad L^2((0,T)\times\R^2).
\end{align*}
With this observation at hand, and in accordance with \eqref{Lp}, we obtain \eqref{17}.

According to \eqref{6}, \eqref{Lp} and the second equation of \eqref{kellersegelmedium}, we infer that, for any $r\in [2,\infty)$,
\begin{align*}
\|\Delta {v^\varepsilon}\|_{L^\infty(0,T;L^r(\R^2))}
&\leq \|\Delta {v^\varepsilon}-{v^\varepsilon}\|_{L^\infty(0,T;L^r(\R^2))}+\|{v^\varepsilon}\|_{L^\infty(0,T;L^r(\R^2))}\\
&\leq \chi\|{u^\varepsilon}\|_{L^\infty(0,T;L^r(\R^2))}+C\|{v^\varepsilon}\|_{L^\infty(0,T;H^1(\R^2))}\leq C.
\end{align*}
In a similar way, it holds
\begin{align*}
\|\Delta {v^\varepsilon}\|_{L^\infty(0,T;L^1(\R^2))}
&\leq \chi\|{u^\varepsilon}\|_{L^\infty(0,T;L^1(\R^2))}+\|{v^\varepsilon}\|_{L^\infty(0,T;L^1(\R^2))}\leq C,
\end{align*}
where we have used
\begin{align*}
\|{v^\varepsilon}\|_{L^\infty(0,T;L^1(\R^2))}
\leq \|\Phi\|_{L^1(\R^2)}\|{u^\varepsilon}\|_{L^\infty(0,T;L^1(\R^2))}\|j^\varepsilon\|_{L^1(\R^2)}\leq C.
\end{align*}
Here the Yukawa potential $\tilde\Phi\in L^q(\mathbb{R}^2)$ $(1\leq q<\infty)$ \cite{LL}.
The estimates above allow us to apply the interpolation inequality to obtain that $\Delta v^\varepsilon$ is
uniformly bounded in $L^\infty(0,T;L^{q}(\R^2))$ $(1\leq q<\infty)$.
This yields, together with the $L^\infty(0,T;L^1(\R^2))$ bound for $v^\varepsilon$, that for $1<p<\infty$
\begin{align}\label{26}
\|{v^\varepsilon}\|_{L^\infty(0,T;W^{2,p}(\R^2))}\leq C.
\end{align}
In view of \eqref{6}, \eqref{Lp} with $p=2$ and the second equation of \eqref{kellersegelmedium}, we have
\begin{align*}
\|\nabla\Delta {v^\varepsilon}\|_{L^2((0,T)\times\R^2)}
\leq \|\nabla {v^\varepsilon}\|_{L^2((0,T)\times\R^2)}+\chi\|\nabla {u^\varepsilon}\|_{L^2((0,T)\times\R^2)}\leq C.
\end{align*}
Summarizing, this shows that
\begin{align}
\|{v^\varepsilon}\|_{L^2(0,T;W^{3,2}(\R^2))}\leq C,
\end{align}
which implies \eqref{18}.

It is easy to see
\begin{align*}
\|u\ast\Phi\|_{L^p((0,T)\times\R^2)}\leq \|u\|_{L^p((0,T)\times\R^2)}\|\Phi\|_{L^1(\R^2)}
\leq C.
\end{align*}
Thus, in accordance with \eqref{17}, we have
\begin{align*}
\|{v^\varepsilon}-v\|_{L^p((0,T)\times\R^2)}
&\leq \|{u^\varepsilon}\ast \Phi\ast j^\varepsilon-u\ast \Phi\ast j^\varepsilon\|_{L^p((0,T)\times\R^2)}
+\|u\ast \Phi\ast j^\varepsilon-u\ast \Phi\|_{L^p((0,T)\times\R^2)}\\
&\leq \|{u^\varepsilon}-u\|_{L^p((0,T)\times\R^2)}\|\Phi\|_{L^1(\R^2)}
+\|u\ast \Phi\ast j^\varepsilon-u\ast \Phi\|_{L^p((0,T)\times\R^2)}
\rightarrow 0,\quad {\rm{as}}\;\varepsilon\rightarrow 0,
\end{align*}
this implies \eqref{19}.
By \eqref{19}, the bound for $e^{-{v^\varepsilon}}$ in $L^\infty((0,T)\times\R^2)$ and mean value theorem, we get \eqref{24}.
\end{proof}

Taking into account the convergences \eqref{a6}-\eqref{24}, we can pass to the limit $\varepsilon\rightarrow 0$ in
the weak formulation of the problem \eqref{kellersegelmedium}, for the variables $({u^\varepsilon},{v^\varepsilon})$,
to conclude that the couple $(u,v)$ verifies the first equation of
\eqref{singalKellerSegel} in the sense of $L^2(0,T;H^{-1}(\R^2))$
and the second equation of \eqref{singalKellerSegel}
in the sense of $L^2(0,T;H^1(\R^2))\cap L^\infty(0,T;L^p(\R^2))$ $(1<p<\infty)$. \\

It remains to prove \eqref{ad1} and \eqref{ad2}.
Integrating \eqref{28} from $0$ to $t\in [0,T]$, we obtain that
\begin{align} \label{F^vareps_ineq_integrated}
\mathcal{F}^\varepsilon(u^\varepsilon,{v^\varepsilon})(t)+\int_0^t \int_{\R^2}u^\varepsilon e^{-{v^\varepsilon}}|\nabla (\log u^\varepsilon-{v^\varepsilon})|^2\,dxds
+(4-c_*\chi)\int_0^t \int_{\R^2}|\nabla\sqrt{u^\varepsilon}|^2\,dxds
\leq \mathcal{F}^\varepsilon(u^\varepsilon,{v^\varepsilon})(0).
\end{align}
Since $4-c_*\chi >0$, the term with this factor is non-negative and hence can be left out.
Now we take $\varepsilon \to 0$ and discuss every term separately.
We start with $\mathcal{F}^\varepsilon(u^\varepsilon,{v^\varepsilon})(t)$.
In view of \eqref{17}, we achieve that $u^\varepsilon (t) \to u(t)$ in $L^2(\R^2) $ for a.e. $t \in [0,T]$.
It follows from \eqref{8} that $\int_{\mathbb{R}^2}u^\varepsilon(t,x) |x|^2\,dx$ and
$\int_{\mathbb{R}^2}u^\varepsilon|\log u^\varepsilon|(t,x)\,dx $ are uniformly bounded w.r.t $\varepsilon$ for a.e. $t \in [0,T]$.
As the assumptions of Lemma \ref{liminf_ulogu} are satisfied, we receive that, for a.e. $t\in [0,T]$,
$$\int_{\R^2} u(t,x) \log u(t,x) \, dx\leq \liminf_{\varepsilon \to 0} \int_{\R^2} u^\varepsilon (t,x) \log u^\varepsilon (t,x) \, dx .$$
Since $\Phi \in W^{1,1}(\R^2)$ and $\nabla v^\varepsilon =u^\varepsilon \ast \nabla \Phi \ast j^\varepsilon$,
we proceed similarly as in \eqref{19} to get
\begin{align} \label{nabla_v^vareps_conv}
    \nabla v^\varepsilon \to \nabla v \quad\text{in}\quad L^p((0,T)\times\R^2),\quad  p \in [2, \infty),
\end{align}
which implies $\nabla v^\varepsilon (t) \to \nabla v(t) $ in $L^2(\R^2) $ for a.e. $t \in [0,T]$.
In a similar way, \eqref{19} and \eqref{17} show that $ v^\varepsilon (t) \to v(t)$ and $u^\varepsilon (t) \to u(t)$ in $L^2(\R^2) $ for a.e. $t \in [0,T]$.
These convergences lead to the convergence of the other terms in $\mathcal{F}^\varepsilon(u^\varepsilon,{v^\varepsilon})(t)$
and by the superadditivity of limit inferior we infer that, for a.e. $t \in [0,T]$,
$$ \mathcal{F}(u,v)(t) \leq  \liminf_{\varepsilon \to 0} \mathcal{F}^\varepsilon(u^\varepsilon,{v^\varepsilon})(t).$$
Next we dicuss $\mathcal{F}^\varepsilon(u^\varepsilon,{v^\varepsilon})(0)$.
Since $u^\varepsilon (0) = u_0$ and $v^\varepsilon (0)= v(0) \ast j^\varepsilon$, it holds
$$ \mathcal{F}^\varepsilon(u^\varepsilon,{v^\varepsilon})(0) = \int_{\mathbb{R}^2}\Big(u_0\log u_0
+\frac{|\nabla {v(0) \ast j^\varepsilon}|^2}{2}+\frac{{|v(0) \ast j^\varepsilon|}^2}{2}
-\frac{u_0 {v(0) \ast j^\varepsilon}}{2}-\chi\frac{u_0 {v(0) \ast j^\varepsilon}\ast j^\varepsilon }{2}\Big)\,dx.$$
By assumption $u_0 \in L^2(\R^2)$ (see Theorem \ref{theorempde}) and $\Phi\in W^{1,1}(\R^2)$, we have $v(0) = \Phi \ast u_0 \in W^{1,2}(\R^2)$.
Together with the properties of the mollification, we conclude that
$$ \lim_{\varepsilon \to 0} \mathcal{F}^\varepsilon(u^\varepsilon,{v^\varepsilon})(0) =  \mathcal{F}(u,{v})(0).$$
Now we fix $t \in [0,T]$. For the second term of \eqref{F^vareps_ineq_integrated}, a simple calculation yields
\begin{align} \label{pf_ad1_reform}
\int_0^t\int_{\R^2}u^\varepsilon e^{-v^\varepsilon}|\nabla(\log u^\varepsilon-v^\varepsilon)|^2\,dxds
=\int_0^t\int_{\R^2}\big|e^{-\frac{v^\varepsilon}{2}}(2\nabla\sqrt {u^\varepsilon}-\sqrt{u^\varepsilon}\nabla v^\varepsilon)\big|^2\,dxds.
\end{align}
It follows from \eqref{17} that $\|\sqrt{u^\varepsilon}\|_{L^4((0,T)\times\R^2)}$ converges to $\|\sqrt{u}\|_{L^4((0,T)\times\R^2)}$
and $u^\varepsilon(t,x)$ converges to $u(t,x)$ for a.e. $(t,x)\in(0,T)\times\R^2$.
As it is well-known, those convergences lead to
\begin{align} \label{conv_sqrt_u^vareps}
\sqrt{u^\varepsilon} \to \sqrt{u} \quad\text{in}\quad L^4((0,T)\times\R^2).
\end{align}
Since $\mathcal{F}^\varepsilon(u^\varepsilon,{v^\varepsilon})(0)$ is uniformly bounded w.r.t. $\varepsilon$,
\eqref{F^vareps_ineq_integrated} and \eqref{7} give an uniform bound for $\nabla\sqrt{u^\varepsilon}$ in $L^2((0,t)\times\R^2)$ w.r.t $\varepsilon$.
This together with \eqref{conv_sqrt_u^vareps} allow us to identify the limit, resulting in
\begin{align} \label{weak_conv_nabla_sqrt_u^vareps}
\nabla\sqrt{u^\varepsilon} \rightharpoonup \nabla\sqrt{u} \quad\text{in}\quad L^2((0,t)\times\R^2).
\end{align}
We combine this result and $e^{-{\frac{v^\varepsilon}{2}}}\rightarrow e^{-\frac{v}{2}}$ in $L^2((0,T)\times\R^2)$ to derive that $e^{-{\frac{v^\varepsilon}{2}}} \nabla\sqrt{u^\varepsilon}  \rightharpoonup e^{-\frac{v}{2}} \nabla\sqrt{u}$ in $L^1((0,t)\times\R^2)$.
Thus, in accordance with the uniform bound for $e^{-{\frac{v^\varepsilon}{2}}} \nabla\sqrt{u^\varepsilon}$ in $L^2((0,t)\times\R^2)$, we obtain that
\begin{align} \label{weak_conv_e^-(v^vareps/2)_nabla_sqrt_u^vareps}
e^{-{\frac{v^\varepsilon}{2}}} \nabla\sqrt{u^\varepsilon}  \rightharpoonup e^{-\frac{v}{2}} \nabla\sqrt{u} \quad\text{in}\quad L^2((0,t)\times\R^2).
\end{align}
The convergences \eqref{nabla_v^vareps_conv} and \eqref{conv_sqrt_u^vareps} give $e^{-\frac{v^\varepsilon}{2}}\sqrt{u^\varepsilon}\nabla v^\varepsilon \to e^{-\frac{v}{2}}\sqrt{u}\nabla v$ in $L^2((0,T)\times\R^2)$.
Together with \eqref{pf_ad1_reform}, \eqref{weak_conv_e^-(v^vareps/2)_nabla_sqrt_u^vareps} and the weak lower semicontinuity of norms, we have
\begin{align*}
\int_0^t\int_{\R^2}u e^{-v}|\nabla(\log u-v)|^2\,dxds \leq \liminf_{\varepsilon \to 0} \int_0^t\int_{\R^2}u^\varepsilon e^{-v^\varepsilon}|\nabla(\log u^\varepsilon-v^\varepsilon)|^2\,dxds .
\end{align*}
Finally, taking into account the monotonicity and superadditivity of the limit inferior, the convergences of all the terms lead to \eqref{ad1}. 

Let $p \in [2,\infty )$. To show \eqref{ad2}, we integrate \eqref{Lp} from $0$ to $t$ and obtain that, for $t \in [0,T]$,
\begin{align} \label{p_geq_2_vareps_ineq_integrated}
&\int_{\mathbb{R}^2}{u^\varepsilon (t)}^p\,dx
+\frac{2(p-1)}{p} \int_0^t \int_{\mathbb{R}^2}e^{-{v^\varepsilon}}|\nabla {u^\varepsilon}^{p/2}|^2\,dxds
+\frac{2(p-1)}{p} \int_0^t \int_{\mathbb{R}^2}|\nabla|\nabla {v^\varepsilon}|^p|^2\,dxds\nonumber\\
&+2p\int_0^t\int_{\mathbb{R}^2}|\nabla {v^\varepsilon}|^{2p}\,dxds
+\frac{2(p-1)}{p} \int_0^t \int_{\mathbb{R}^2}|\nabla {u^\varepsilon}^{p/2}|^2\,dxds \leq \int_{\mathbb{R}^2}u_0^p\,dx + CT.
\end{align}
By \eqref{17}, we deduce that $u^\varepsilon(t) \to u(t)$ in $L^p(\R^2)$ for a.e. $t \in [0,T]$.
It follows from $u^\varepsilon \geq 0$ and Fatou's lemma that, for a.e. $t\in [0,T]$,
$$ \int_{\mathbb{R}^2}{u (t)}^p\,dx \leq \liminf_{\varepsilon \to 0} \int_{\mathbb{R}^2}{u^\varepsilon (t)}^p\,dx .$$
Since the right-hand side of \eqref{p_geq_2_vareps_ineq_integrated}  doesn't depend on $\varepsilon$, it provides an uniform bound for $\nabla {u^\varepsilon}^{p/2}$ in $L^2((0,t)\times\R^2) $ w.r.t $\varepsilon$.
Thus, in accordance with ${u^\varepsilon}^{p/2} \to {u}^{p/2}$ in $L^2((0,T)\times\R^2)$, which follows from \eqref{17},
we obtain that
$\nabla {u^\varepsilon}^{p/2} \rightharpoonup \nabla {u}^{p/2} $ in $L^2((0,t)\times\R^2) $.
Using the weak lower semicontinuity of norms, we achieve that
$$ \int_0^t \int_{\mathbb{R}^2}|\nabla {u}^{p/2}|^2\,dxds \leq \liminf_{\varepsilon \to 0} \int_0^t \int_{\mathbb{R}^2}|\nabla {u^\varepsilon}^{p/2}|^2\,dxds .$$
Similar to \eqref{weak_conv_e^-(v^vareps/2)_nabla_sqrt_u^vareps}, we infer that $e^{-{\frac{v^\varepsilon}{2}}}\nabla {u^\varepsilon}^{p/2} \rightharpoonup e^{-{\frac{v}{2}}}\nabla {u}^{p/2} $ in $L^2((0,T)\times\R^2)$, which provides
$$  \int_0^t \int_{\mathbb{R}^2}e^{-{v}}|\nabla {u}^{p/2}|^2\,dxds \leq \liminf_{\varepsilon \to 0} \int_0^t \int_{\mathbb{R}^2}e^{-{v^\varepsilon}}|\nabla {u^\varepsilon}^{p/2}|^2\,dxds .$$
From \eqref{nabla_v^vareps_conv} we know that  $\nabla v^\varepsilon \to \nabla v$ in $L^{2p}((0,T)\times\R^2) $, which yields
$$ \int_0^t \int_{\mathbb{R}^2}|\nabla {v}|^{2p}\,dxds = \lim_{\varepsilon \to 0} \int_0^t \int_{\mathbb{R}^2}|\nabla {v^\varepsilon}|^{2p}\,dxds. $$
Taking into account \eqref{nabla_v^vareps_conv},
we proceed similarly as the proof of \eqref{conv_sqrt_u^vareps} to infer that $|\nabla v^\varepsilon |^p \to |\nabla v |^p$ in $L^{2}((0,T)\times\R^2) $.
Further \eqref{p_geq_2_vareps_ineq_integrated} allows us to derive a uniform bound for $\nabla |\nabla v^\varepsilon |^p $ in $L^{2}((0,t)\times\R^2) $ w.r.t $\varepsilon$.
Following the arguments for \eqref{weak_conv_nabla_sqrt_u^vareps} we get $\nabla |\nabla v^\varepsilon |^p \rightharpoonup \nabla |\nabla v |^p  $ in $L^{2}((0,t)\times\R^2) $.
Applying the weak lower semicontinuity of norms once more, we deduce that
$$ \int_0^t \int_{\mathbb{R}^2}|\nabla|\nabla {v}|^p|^2\,dxds \leq \liminf_{\varepsilon \to 0} \int_0^t \int_{\mathbb{R}^2}|\nabla|\nabla {v^\varepsilon}|^p|^2\,dxds. $$
Hence, we complete the proof of \eqref{ad2}.

The boundness of $|x|^2u$ in $L^\infty(0,T;L^1(\mathbb{R}^2))$ follows directly from \eqref{8} and the fact that $u^\varepsilon$ converges to $u$ a.e. in $(0,T)\times\mathbb{R}^2$ by using Fatou's lemma.

\subsection{The error estimates for ${u^\varepsilon}-u$}
First of all, we give the following lemma which will be used later.
\begin{lemma}\label{error}
Let $\varepsilon>0$, $q\in [1,\infty]$ and $f\in W^{1,q}(\R^2)$. Then it holds
\begin{align*}
\|j^\varepsilon\ast f-f\|_{L^q(\R^2)}
\leq \varepsilon\|\nabla f\|_{L^q(\R^2)}.
\end{align*}
\end{lemma}
\begin{proof}
For the case $q=\infty$,
\begin{align*}
\|j^\varepsilon\ast f-f\|_{L^\infty(\R^2)}
=\sup_{x\in \R^2}\Big|\int_{B_\varepsilon(0)}|y|j^\varepsilon(y)\frac{f(x-y)-f(x)}{|y|}\,dy \Big|
\leq \varepsilon\|\nabla f\|_{L^\infty(\R^2)}.
\end{align*}
Now we focus on the case $q\in [1,\infty)$. For any $g\in L^{q'}(\R^2)$,
\begin{align*}
\Big|\int_{\R^2}(j^\varepsilon\ast f-f)g(x)\,dx \Big|
&=\Big|\int_{\R^2}\int_{B_\varepsilon(0)}j^\varepsilon(y)\big(f(x-y)-f(x)\big)\,dy g(x)\,dx \Big|\\
&=\Big|\int_{\R^2}\int_{B_\varepsilon(0)}j^\varepsilon(y)\int_0^1\nabla f(x-ry)y\,drdy g(x)\,dx \Big|\\
&\leq \varepsilon\int_0^1\int_{B_\varepsilon(0)}j^\varepsilon(y)\int_{\R^2}|\nabla f(x-ry)||g(x)|\,dxdydr\\
&\leq \varepsilon\|\nabla f\|_{L^q(\R^2)}\|g\|_{L^{q'}(\R^2)}.
\end{align*}
Hence we complete the proof of Lemma \ref{error}.
\end{proof}
Using ${u^\varepsilon}-u\in L^2(0,T;H^1(\R^2))$ to test the following system:
\begin{align*}
\partial_t({u^\varepsilon}-u)=\Delta(e^{-{v^\varepsilon}}{u^\varepsilon}-e^{-v}u)
+\Delta({u^\varepsilon}-u),
\end{align*}
we get
\begin{align}\label{20}
\frac{1}{2}\frac{d}{dt}&\int_{\R^2}|{u^\varepsilon}-u|^2\,dx+\int_{\R^2}|\nabla({u^\varepsilon}-u)|^2\,dx
=-\int_{\R^2}\nabla \big(e^{-\Phi\ast {u^\varepsilon}\ast j^\varepsilon}{u^\varepsilon}-e^{-\Phi\ast u}u \big)\cdot
\nabla({u^\varepsilon}-u)\,dx\nonumber\\
=&-\int_{\R^2}\nabla\big(e^{-\Phi\ast {u^\varepsilon}\ast j^\varepsilon}({u^\varepsilon}-u) \big)\cdot
\nabla({u^\varepsilon}-u)\,dx
-\int_{\R^2}\nabla\big((e^{-\Phi\ast {u^\varepsilon}\ast j^\varepsilon}-e^{-\Phi\ast u\ast j^\varepsilon})u \big)\cdot
\nabla({u^\varepsilon}-u)\,dx\nonumber\\
&-\int_{\R^2}\nabla\big((e^{-\Phi\ast u\ast j^\varepsilon}-e^{-\Phi\ast u})u \big)\cdot
\nabla({u^\varepsilon}-u)\,dx\nonumber\\
=:&J_1+J_2+J_3.
\end{align}
In view of \eqref{26} and Sobolev embedding theorem, the term $J_1$ can be bounded by
\begin{align}\label{21}
J_1&\leq
\int_{\R^2}e^{-\Phi\ast {u^\varepsilon}\ast j^\varepsilon}({u^\varepsilon}-u)\nabla\Phi\ast {u^\varepsilon}\ast j^\varepsilon\cdot\nabla({u^\varepsilon}-u)\,dx\nonumber\\
&\leq \frac{1}{6}\int_{\R^2}|\nabla({u^\varepsilon}-u)|^2\,dx
+C\|\nabla {v^\varepsilon}\|_{L^\infty((0,T)\times\R^2)}^2\int_{\R^2}|{u^\varepsilon}-u|^2\,dx\nonumber\\
&\leq \frac{1}{6}\int_{\R^2}|\nabla({u^\varepsilon}-u)|^2\,dx
+C\int_{\R^2}|{u^\varepsilon}-u|^2\,dx,
\end{align}
where $C$ appeared in this subsection is a positive constant independent of $\varepsilon$.
We split the term $J_2$ into four parts,
\begin{align*}
J_2\leq& \frac{1}{6}\int_{\R^2}|\nabla({u^\varepsilon}-u)|^2\,dx\\
&+C\int_{\R^2}|e^{-\Phi\ast {u^\varepsilon}\ast j^\varepsilon}-e^{-\Phi\ast u\ast j^\varepsilon}|^2|\nabla u|^2\,dx\\
&+ C\int_{\R^2}|e^{-\Phi\ast {u^\varepsilon}\ast j^\varepsilon}\nabla(\Phi\ast {u^\varepsilon}\ast j^\varepsilon-\Phi\ast u\ast j^\varepsilon)|^2u^2\,dx\\
&+C\int_{\R^2}|(e^{-\Phi\ast {u^\varepsilon}\ast j^\varepsilon}-e^{-\Phi\ast u\ast j^\varepsilon})\nabla\Phi\ast u\ast j^\varepsilon|^2u^2\,dx\\
=:&\frac{1}{6}\int_{\R^2}|\nabla({u^\varepsilon}-u)|^2\,dx+J_{21}+J_{22}+J_{23}.
\end{align*}
According to mean value theorem, the term $J_{21}$ can be bounded by
\begin{align*}
J_{21}&\leq C\big\|e^{-\Phi\ast {u^\varepsilon}\ast j^\varepsilon}-e^{-\Phi\ast u\ast j^\varepsilon} \big\|_{L^\infty(\R^2)}^2
\|\nabla u\|_{L^2(\R^2)}^2
\leq C\|\Phi\ast ({u^\varepsilon}-u)\|_{L^\infty(\R^2)}^2\|\nabla u\|_{L^2(\R^2)}^2\\
&\leq C\|\nabla u\|_{L^2(\R^2)}^2\|\Phi\|_{L^2(\R^2)}^2\|{u^\varepsilon}-u\|_{L^2(\R^2)}^2
\leq C\|\nabla u\|_{L^2(\R^2)}^2\|{u^\varepsilon}-u\|_{L^2(\R^2)}^2.
\end{align*}
Taking into account the bound for $\nabla\Phi$ in $L^q(\R^2)$ for any $q\in (1,2)$,
the term $J_{22}$ can be handled by
\begin{align*}
J_{22}\leq C\|\nabla\Phi\ast ({u^\varepsilon}-u)\|_{L^4(\R^2)}^2\|u\|_{L^\infty(0,T;L^4(\R^2)}^2
\leq C\|\nabla\Phi\|_{L^{4/3}(\R^2)}^2\|{u^\varepsilon}-u\|_{L^2(\R^2)}^2
\leq C\|{u^\varepsilon}-u\|_{L^2(\R^2)}^2.
\end{align*}
The convergence \eqref{18} implies, for any $p>2$,
\begin{align*}
\|\nabla v\|_{L^\infty((0,T)\times\R^2)}\leq C\|\nabla v\|_{L^\infty(0,T;W^{1,p}(\R^2)}
\leq C\liminf_{\varepsilon\rightarrow 0}\|\nabla {v^\varepsilon}\|_{L^\infty(0,T;W^{1,p}(\R^2)}\leq C.
\end{align*}
For the term $J_{23}$,
\begin{align*}
J_{23}&\leq C\|\Phi\ast ({u^\varepsilon}-u)\ast j^\varepsilon\|_{L^\infty(\R^2)}^2
\|\nabla v\|_{L^\infty((0,T)\times\R^2)}^2\|u\|_{L^\infty(0,T;L^2(\R^2))}^2
\leq C\|\Phi\ast ({u^\varepsilon}-u)\|_{L^\infty(\R^2)}^2\\
&\leq C\|\Phi\|_{L^2(\R^2)}^2\|{u^\varepsilon}-u\|_{L^2(\R^2)}^2
\leq C\|{u^\varepsilon}-u\|_{L^2(\R^2)}^2.
\end{align*}
It follows that
\begin{align}\label{22}
J_2\leq \frac{1}{6}\int_{\R^2}|\nabla({u^\varepsilon}-u)|^2\,dx+C\int_{\R^2}|{u^\varepsilon}-u|^2\,dx.
\end{align}

With the help of mean value theorem, the term $J_3$ can be divided into four parts
\begin{align*}
J_3\leq &\frac{1}{6}\int_{\R^2}|\nabla({u^\varepsilon}-u)|^2\,dx
+C\int_{\R^2}|\Phi\ast u\ast j^\varepsilon-\Phi\ast u|^2|\nabla u|^2\,dx\\
&+C\int_{\R^2}u^2|\Phi\ast u\ast j^\varepsilon-\Phi\ast u|^2|\nabla\Phi\ast u\ast j^\varepsilon|^2\,dx
+C\int_{\R^2}u^2|\nabla\Phi\ast u\ast j^\varepsilon-\nabla\Phi\ast u|^2\,dx\\
=:&\frac{1}{6}\int_{\R^2}|\nabla({u^\varepsilon}-u)|^2\,dx+J_{31}+J_{32}+J_{33}.
\end{align*}
By Lemma \ref{error}, we obtain that
\begin{align*}
J_{31}\leq &C\|\Phi\ast u\ast j^\varepsilon-\Phi\ast u\|_{L^\infty((0,T)\times\R^2)}^2\|\nabla u\|_{L^2(\R^2)}^2
\leq C\varepsilon^2\|\nabla\Phi\ast u\|_{L^\infty((0,T)\times\R^2)}^2\|\nabla u\|_{L^2(\R^2)}^2\\
\leq & C\varepsilon^2\|\nabla v\|_{L^\infty((0,T)\times\R^2)}^2\|\nabla u\|_{L^2(\R^2)}^2\leq C\|\nabla u\|_{L^2(\R^2)}^2\varepsilon^2.
\end{align*}
Similarly,
\begin{align*}
J_{32}\leq &C\|v\ast j^\varepsilon -v\|_{L^\infty((0,T)\times\R^2))}^2\|u\|_{L^\infty(0,T;L^4(\R^2))}^2\|\nabla\Phi\ast u\ast j^\varepsilon\|_{L^\infty(0,T;L^4(\R^2))}^2\\
\leq &C\varepsilon^2\|\nabla v\|_{L^\infty((0,T)\times\R^2)}^2\|\nabla v\|_{L^\infty(0,T;L^4(\R^2))}^2\leq C\varepsilon^2.
\end{align*}
For the term $J_{33}$,
\begin{align*}
J_{33}\leq & C\|\nabla\Phi\ast u\ast j^\varepsilon-\nabla\Phi\ast u\|_{L^\infty(0,T;L^4(\R^2))}^2\|u\|_{L^\infty(0,T;L^4(\R^2))}^2\\
\leq & C\varepsilon^2\|D^2\Phi\ast u\|_{L^\infty(0,T;L^4(\R^2))}^2\leq C\varepsilon^2\|D^2v\|_{L^\infty(0,T;L^4(\R^2))}^2\leq C\varepsilon^2.
\end{align*}
Therefore, we have
\begin{align}\label{23}
J_3\leq \frac{1}{6}\int_{\R^2}|\nabla({u^\varepsilon}-u)|^2\,dx+C\Big(1+\int_{\R^2}|\nabla u|^2\,dx\Big)\varepsilon^2.
\end{align}

Plugging \eqref{21}-\eqref{23} into \eqref{20}, we infer that
\begin{align*}
\frac{d}{dt}\int_{\R^2}|{u^\varepsilon}-u|^2\,dx+\int_{\R^2}|\nabla({u^\varepsilon}-u)|^2\,dx
\leq C\varepsilon^2\Big(1+\int_{\R^2}|\nabla u|^2\,dx\Big)+C\int_{\R^2}|{u^\varepsilon}-u|^2\,dx.
\end{align*}
Together with the fact that $u^\varepsilon$ and $u$ share the same initial data and $\|\nabla u\|_{L^2(0,T;L^2(\R^2))}$ is bounded, Gronwall's inequality implies that
\begin{align*}
&\|{u^\varepsilon}-u\|_{L^\infty(0,T;L^2(\R^2))}\leq C\varepsilon,\\
&\|\nabla({u^\varepsilon}-u)\|_{L^2((0,T)\times\R^2)}\leq C\varepsilon.
\end{align*}
Since the second moment is bounded, showed in Proposition \ref{energy}, we have that
\begin{align*}
\int_{\mathbb{R}^2}|u^\varepsilon-u|\,dx&\leq \int_{B_{R}(0)}|u^\varepsilon-u|\,dx
+\int_{B_{R}(0)^c}\frac{|x|^2}{R^2}|u^\varepsilon-u|\,dx\\
&\leq CR\Big(\int_{\mathbb{R}^2}|u^\varepsilon-u|^2\,dx\Big)^{1/2}
+\frac{1}{R^2}\int_{\mathbb{R}^2}|x|^2|u^\varepsilon-u|\,dx
\leq C R\varepsilon+\frac{C}{R^2}.
\end{align*}
Taking $R=\sqrt{\varepsilon}^{-1}$, we get
\begin{align*}
\|u^\varepsilon-u\|_{L^\infty(0,T;L^1(\mathbb{R}^2))}\leq C\sqrt{\varepsilon}.
\end{align*}
Thus we have finished the proof of theorem \ref{theorempde}.

\section{Mean-field limit}
\label{The_proof_of_Theorem_3}
In this section, we focus on the mean-field limit part and give the proof of Theorem \ref{mean-field} and Theorem \ref{relative}. First of all, based on the results given in Theorem \ref{theorempde}, the proof of solvability of McKean-Vlasov problems \eqref{generalized_intermediate_particle_model} and \eqref{generalized_particle_model} can be done exactly the same as in  \cite{chen2021rigorous}. Namely, we can easily obtain that \eqref{generalized_intermediate_particle_model} has a unique square integrable solution $\bar{X}_i^\varepsilon$ with $u^\varepsilon$ as the density of its distribution, and \eqref{generalized_particle_model} has a unique square integrable solution $\hat{X}_i$ with $u$ as the density of its distribution.

The propagation of chaos result is given in this section both in the weak sense by studying the difference of trajectories in expectation, and the strong $L^1$ sense under additional condition on the PDE solution by applying the relative entropy method. Notice that the application of the Burkholder-Davis-Gundy inequality doesn't allow one to estimate the maximum norm of the trajectories, as has been often used in the models with drift interaction. The current proof follows the idea from \cite{chen2021rigorous}. Due to its different interaction from those in \cite{chen2021rigorous}, we show a different convergence rate. Furthermore, the estimate for difference of trajectories in expectation provides us the idea to close the relative entropy estimate.

\subsection{Propagation of chaos in the weak sense}
The proof of Theorem \ref{mean-field} is done in two steps.
First, we give the uniform in $\varepsilon$ estimate for the difference $X^\varepsilon_{N,i} - \bar{X}^\varepsilon_i$.
Then, the estimate for the difference $\bar{X}^\varepsilon_i-\hat{X}_i$ is obtained with the help of the error estimates for $u^\varepsilon-u$ given in Theorem \ref{theorempde}.
\begin{proposition} \label{mean-field_est_1}
Let $X^\varepsilon_{N,i}$ and $\bar{X}^\varepsilon_{i}$ be square integrable solutions to \eqref{generalized_regularized_particle_model} and \eqref{generalized_intermediate_particle_model}, respectively.
Under the assumptions of Theorem \ref{theorempde}, for any $\beta\in (0,1)$ and $N$ large enough,
there exists a constant $C> 0$ independent of $N$ such that
\begin{align}\label{r1}
\max_{i=1,\dots,N} \ept \Big( \sup_{t\in [0,T] } | X^\varepsilon_{N,i} - \bar{X}^\varepsilon_i |^2(t) \Big)
\leq CN^{-\beta},
\end{align}
where the cut-off parameter satisfies
\begin{align*}
\varepsilon=(\lambda\log N)^{-\frac{1}{4}}{\mbox{ with }}  -1+\delta+\bar{C}T\lambda\leq -\beta.
\end{align*}
Here $\delta>0$ is an arbitrary small constant and $\bar C$ is a positive constant independent of $N$.
\end{proposition}

\begin{proof}
For the difference $X^\varepsilon_{N,i} - \bar{X}^\varepsilon_i$, it follows from the Burkholder-Davis-Gundy inequality and mean value theorem that
\begin{align}\label{A4r}
&\ept \Big(\sup_{t\in [0,T]}|X^\varepsilon_{N,i}-\bar{X}^\varepsilon_{i}|^2 \Big)\nonumber\\
\leq&C \ept\Big(\int_0^T \Big|\Big(2\exp\Big(-\frac{1}{N}\sum_{\substack{j=1 }}^N\Phi^\varepsilon
(X^\varepsilon_{N,i}-X^\varepsilon_{N,j}) \Big)+2\Big)^{1/2}
-\big(2\exp(-\Phi^\varepsilon\ast u^\varepsilon(t,\bar{X}^\varepsilon_i))+2 \big)^{1/2} \Big|^2\,dt \Big)\nonumber\\
\leq& C \ept  \Big(\int_0^T \Big| \frac{1}{N} \sum_{\substack{j=1 }}^N \Phi^\varepsilon (X^\varepsilon_{N,i} - X^\varepsilon_{N,j}) -  \Phi^\varepsilon \ast u^\varepsilon(t,\bar{X}^\varepsilon_i)  \Big|^2 \,dt  \Big)\nonumber\\
\leq& C\ept \Big(\int_0^T (A_1^2+A_2^2+A_3^2)(t) \,dt\Big),
\end{align}
where
\begin{align*}
A_1(t)&:=  \frac{1}{N} \sum_{\substack{j=1}}^N \Phi^\varepsilon (X^\varepsilon_{N,i} - X^\varepsilon_{N,j})  -   \frac{1}{N} \sum_{\substack{j=1}}^N \Phi^\varepsilon (\bar{X}^\varepsilon_{i} - X^\varepsilon_{N,j}), \\
A_2(t)&:=   \frac{1}{N} \sum_{\substack{j=1}}^N \Phi^\varepsilon (\bar{X}^\varepsilon_{i} - X^\varepsilon_{N,j}) -  \frac{1}{N} \sum_{\substack{j=1}}^N \Phi^\varepsilon (\bar{X}^\varepsilon_{i} - \bar{X}^\varepsilon_{j}), \\
A_3(t)&:= \frac{1}{N} \sum_{\substack{j=1}}^N \Phi^\varepsilon (\bar{X}^\varepsilon_{i} - \bar{X}^\varepsilon_{j}) - \Phi^\varepsilon \ast u^\varepsilon(t,\bar{X}^\varepsilon_i).
\end{align*}
It is time to deal with the first term
\begin{align}\label{A1r}
\ept \Big( \int_0^T  A_1^2(t)\, dt \Big)
\leq&\ept\Big(\int_0^T \Big|\frac{1}{N} \sum_{\substack{j=1}}^N \big| \Phi^\varepsilon (X^\varepsilon_{N,i} - X^\varepsilon_{N,j}) -  \Phi^\varepsilon (\bar{X}^\varepsilon_{i} - X^\varepsilon_{N,j}) \big|\Big|^2 \,dt\Big)\nonumber\\
\leq & \|\nabla\Phi^\varepsilon\|_{L^\infty(\mathbb{R}^2)}^2 \ept \Big( \int_0^T  \left| X^\varepsilon_{N,i} - \bar{X}^\varepsilon_{i} \right|^2 \,dt \Big) \nonumber\\
\leq &  \frac{C}{\varepsilon^4} \int_0^T \ept \Big(  \sup_{s \in [0,t]} \big| X^\varepsilon_{N,i} - \bar{X}^\varepsilon_{i} \big|^2(s) \Big)\, dt,
\end{align}
where we have used
\begin{align*}
\|\nabla\Phi^\varepsilon\|_{L^\infty(\mathbb{R}^2)}\leq \|\nabla\Phi\|_{L^1(\mathbb{R}^2)}\|j^\varepsilon\|_{L^\infty(\mathbb{R}^2)}\leq \frac{C}{\varepsilon^2}.
\end{align*}
The second term can be calculated as
\begin{align}\label{A2r}
\ept \Big( \int_0^T  A_2^2(t) \,dt \Big)
\leq&\ept\Big(\int_0^T \Big|\frac{1}{N} \sum_{\substack{j=1}}^N \big|\Phi^\varepsilon (\bar{X}^\varepsilon_{i} - X^\varepsilon_{N,j})
-  \Phi^\varepsilon (\bar{X}^\varepsilon_{i} - \bar{X}^\varepsilon_{j})\big|\Big|^2 \,dt\Big)\nonumber\\
\leq & \|\nabla\Phi^\varepsilon\|_{L^\infty(\mathbb{R}^2)}^2 \ept
\Big( \int_0^T \frac{1}{N} \sum_{\substack{j=1}}^N \big| X^\varepsilon_{N,j} - \bar{X}^\varepsilon_{j} \big|^2\, dt \Big) \nonumber\\
\leq &   \frac{C}{N\varepsilon^4} \int_0^T \sum_{\substack{j=1}}^N\ept \Big(  \sup_{s \in [0,t]} \big| X^\varepsilon_{N,j}
 - \bar{X}^\varepsilon_{j} \big|^2(s) \Big) \,dt.
\end{align}
To deal with the third term, we introduce
$$ Z_{i,j}:=\Phi^\varepsilon (\bar{X}^\varepsilon_{i} - \bar{X}^\varepsilon_{j}) - \Phi^\varepsilon \ast u^\varepsilon(t,\bar{X}^\varepsilon_i), $$
which implies
\begin{align*}
|A_3|^2 = \Big|\frac{1}{N}\sum_{j=1}^N Z_{i,j}\Big|^2=\frac{1}{N^2} \sum_{j,l=1}^N Z_{i,j} Z_{i,l}.
\end{align*}
We start with the case $j \neq l$, $i \neq j$ and $i \neq l$.
\begin{align*}
\ept (Z_{i,j}Z_{i,l})
=&\ept \big((\Phi^\varepsilon(\bar{X}^\varepsilon_i-\bar{X}^\varepsilon_j)-\Phi^\varepsilon\ast u^\varepsilon(t,\bar{X}^\varepsilon_i) )
(\Phi^\varepsilon(\bar{X}^\varepsilon_i-\bar{X}^\varepsilon_l)-\Phi^\varepsilon\ast u^\varepsilon(t,\bar{X}^\varepsilon_i) ) \big)\\
=&\int_{\mathbb{R}^2}\int_{\mathbb{R}^2}\int_{\mathbb{R}^2} (\Phi^\varepsilon(x-y)-\Phi^\varepsilon\ast u^\varepsilon(t,x))
(\Phi^\varepsilon(x-z)-\Phi^\varepsilon\ast u^\varepsilon(t,x))u^\varepsilon(t,x)u^\varepsilon(t,y)u^\varepsilon(t,z)\,dxdydz=0.
\end{align*}
For $j\neq l$ and $i=j$, we have
\begin{align*}
\ept(Z_{i,i}Z_{i,l})
=\ept \big((\Phi^\varepsilon(0)-\Phi^\varepsilon\ast u^\varepsilon(t,\bar{X}^\varepsilon_i))
(\Phi^\varepsilon(\bar{X}^\varepsilon_i-\bar{X}^\varepsilon_l)-\Phi^\varepsilon\ast u^\varepsilon(t,\bar{X}^\varepsilon_i)) \big)=0.
\end{align*}
Similarly, for $j\neq l$ and $i=l$, we have
\begin{align*}
\ept(Z_{i,j}Z_{i,i})=0.
\end{align*}
Together with the fact that
\begin{align*}
|Z_{i,j}|^2=|\Phi^\varepsilon (\bar{X}^\varepsilon_{i} - \bar{X}^\varepsilon_{j}) - \Phi^\varepsilon \ast u^\varepsilon(t,\bar{X}^\varepsilon_i)|^2
\leq \frac{C}{\varepsilon^4},
\end{align*}
we obtain that
\begin{align}\label{A3r}
\ept \Big(\int_0^T A_3^2(t)\,dt \Big)=\frac{1}{N^2}\ept \Big(\int_0^T \sum_{j,l=1}^N Z_{i,j}Z_{i,l}\,dt \Big)
\leq \frac{1}{N^2}\ept \Big(\int_0^T \sum_{j=1}^N |Z_{i,j}|^2\,dt \Big)
+\frac{1}{N^2}\ept \Big(\int_0^T |Z_{i,i}|^2\,dt \Big)\leq \frac{C}{N\varepsilon^4}.
\end{align}
Inserting \eqref{A1r}-\eqref{A3r} into \eqref{A4r}, we derive that
\begin{align*}
\max_{i=1,\dots,N} \ept \Big( \sup_{t\in [0,T] } | X^\varepsilon_{N,i} - \bar{X}^\varepsilon_i |^2(t) \Big)
\leq  \frac{C}{N\varepsilon^{4}}  + \frac{\bar C}{\varepsilon^4} \int_0^T \max_{i=1,\dots,N} \ept \Big(  \sup_{s \in [0,t]} | X^\varepsilon_{N,i} - \bar{X}^\varepsilon_{i} |^2(s) \Big) dt.
\end{align*}
By Gronwall's inequality, it holds
$$ \max_{i=1,\dots,N} \ept \Big( \sup_{t\in [0,T] } \left| X^\varepsilon_{N,i} - \bar{X}^\varepsilon_i \right|^2(t) \Big)
\leq  \frac{C}{N\varepsilon^{4}} \exp \Big({ \frac{\bar{C}T}{\varepsilon^4}  } \Big). $$
Taking $\varepsilon=(\lambda\log N)^{-\frac{1}{4}}$ with $\lambda>0$, we obtain that for big enough $N$
\begin{align*}
\frac{1}{N\varepsilon^4}\exp \Big({ \frac{\bar{C}T}{\varepsilon^4}  } \Big)=\lambda \log N N^{-1+\bar{C}T\lambda}
\leq \lambda N^{-1+\delta+\bar{C}T\lambda},
\end{align*}
where $0<\delta\ll 1$ is  arbitrary. Therefore, it is enough to choose $\lambda$ such that for any fixed $\beta\in (0,1)$
\begin{align*}
 -1+\delta+\bar{C}T\lambda\leq -\beta.
\end{align*}
This finishes the proof of Proposition \ref{mean-field_est_1}.
\end{proof}

Based on the error estimates between $u^\varepsilon$ and $u$, we infer the estimate for $\bar{X}^\varepsilon_i - \hat{X}_i$ as follows.
\begin{proposition} \label{mean-field_est_2}
Let $\bar{X}^\varepsilon_{i}$ and $\hat{X}_i$ be solutions to \eqref{generalized_intermediate_particle_model} and \eqref{generalized_particle_model}, respectively.
Under the assumptions of Theorem \ref{theorempde}, there exists $C >0$ independent of $\varepsilon$ such that for any small $\varepsilon >0 $
\begin{align*}
\max_{i=1,\dots,N} \ept \Big( \sup_{t\in [0,T] } | \bar{X}^\varepsilon_i - \hat{X}_i |^2(t) \Big) \leq C \varepsilon^2.
\end{align*}
\end{proposition}

\begin{proof}
In view of the Burkholder-Davis-Gundy inequality and  mean value theorem
\begin{align}\label{A5r}
\ept \Big( \sup_{t\in [0,T] } | \bar{X}^\varepsilon_i - \hat{X}_i |^2(t) \Big)
\leq & C \ept \Big( \int_0^T \Big| \sqrt{2 \exp \big( -\Phi^\varepsilon \ast u^\varepsilon(t,\bar{X}^\varepsilon_i) \big)+ 2  } - \sqrt{2 \exp \big(- \Phi \ast u(t,\hat{X}_i) \big)+2  } \Big|^2 d t \Big) \nonumber\\
\leq & C \ept \Big( \int_0^T \big| \Phi^\varepsilon \ast u^\varepsilon(t,\bar{X}^\varepsilon_i) -   \Phi \ast u(t,\hat{X}_i) \big|^2 d t \Big)\nonumber\\
\leq &C \ept \Big(\int_0^T (A_4^2+A_5^2+A_6^2)(t)\,dt \Big),
\end{align}
where
\begin{align*}
A_4 := &   \Phi^\varepsilon \ast u^\varepsilon(t,\bar{X}^\varepsilon_i)-\Phi^\varepsilon \ast u(t,\bar{X}^\varepsilon_i)  , \\
A_5 := & \Phi^\varepsilon \ast u(t,\bar{X}^\varepsilon_i)- \Phi^\varepsilon \ast u(t,\hat{X}_i)  , \\
A_6 := & \Phi^\varepsilon \ast u(t,\hat{X}_i)- \Phi \ast u(t,\hat{X}_i).
\end{align*}
By \eqref{a1}, \eqref{Lp} and the bound for $\Phi$ in $L^{4/3}(\mathbb{R}^2)$, we deduce that
\begin{align}\label{A6r}
\ept \Big(\int_0^T A_4^2(t)\,dt \Big)
& = \int_0^T \int_{\R^2} \left( \Phi^\varepsilon \ast u^\varepsilon(t,x) -  \Phi^\varepsilon \ast u(t,x)  \right)^2 u^\varepsilon(t,x) \,dx d t \nonumber\\
& \leq \|u^\varepsilon\|_{L^1(0,T;L^2(\mathbb{R}^2))}\|\Phi^\varepsilon\ast (u^\varepsilon-u)\|_{L^\infty(0,T;L^4(\mathbb{R}^2))}^2\nonumber\\
&\leq C\|\Phi^\varepsilon\|_{L^{4/3}(\mathbb{R}^2)}^2\|u^\varepsilon-u\|_{L^\infty(0,T;L^2(\mathbb{R}^2))}^2\leq C\varepsilon^2.
\end{align}
The second term can be calculated as
\begin{align}\label{A7r}
\ept \Big(\int_0^T A_5^2(t)\,dt \Big)
\leq  & C\|\nabla\Phi^\varepsilon\ast u\|_{L^\infty((0,T)\times\mathbb{R}^2)}^2
 \ept \Big( \int_0^T |\bar{X}^\varepsilon_i- \hat{X}_i  |^2 \,d t \Big)
\leq C\ept \Big( \int_0^T | \bar{X}^\varepsilon_i -\hat{X}_i |^2 \,d t \Big),
\end{align}
where we have used
\begin{align*}
\|\nabla\Phi^\varepsilon\ast u\|_{L^\infty((0,T)\times\mathbb{R}^2)}
\leq \|\nabla\Phi\|_{L^{4/3}(\mathbb{R}^2)}\|j^\varepsilon\|_{L^1(\mathbb{R}^2)}\|u\|_{L^\infty(0,T;L^4(\mathbb{R}^2))}\leq C.
\end{align*}
Applying Lemma \ref{error}, it holds
\begin{align}\label{A8r}
\ept \Big(\int_0^T A_6^2(t)\,dt \Big)
 &= \int_0^T \int_{\R^2} \left( \Phi^\varepsilon \ast u(t,x) -  \Phi\ast u(t,x) \right)^2 u \, dx dt \nonumber\\ 
 &\leq \|u\|_{L^1(0,T;L^2(\mathbb{R}^2))}\|(\Phi^\varepsilon-\Phi)\ast u\|_{L^\infty(0,T;L^4(\mathbb{R}^2))}^2\nonumber\\
 &\leq C\|\Phi^\varepsilon-\Phi\|_{L^{4/3}(\mathbb{R}^2)}^2\|u\|_{L^\infty(0,T;L^2(\mathbb{R}^2))}^2\nonumber\\
 &\leq C\varepsilon^2\|\nabla\Phi\|_{L^{4/3}(\mathbb{R}^2)}^2\leq C\varepsilon^2.
\end{align}
Inserting \eqref{A6r}-\eqref{A8r} into \eqref{A5r}, we infer that
 $$\max_{i=1,\dots,N}\ept \Big( \sup_{t\in [0,T] }| \bar{X}^\varepsilon_i - \hat{X}_i |^2(t) \Big)
 \leq C \varepsilon^2 + C \int_0^T  \max_{i=1,\dots,N}\ept \Big( \sup_{s\in [0,t] } |\bar{X}^\varepsilon_i - \hat{X}_i|^2(s)  \Big) dt .$$
Combining Gronwall's inequality, we obtain that
$$ \max_{i=1,\dots,N}\ept \Big( \sup_{t\in [0,T] } | \bar{X}^\varepsilon_i - \hat{X}_i |^2(t)\Big) \leq C \varepsilon^2.$$
Hence we complete the proof of Proposition \ref{mean-field_est_2}.
\end{proof}
It follows from Proposition \ref{mean-field_est_1} and Proposition \ref{mean-field_est_2} that
\begin{align*}
& \max_{i = 1, \dots, N} \ept \Big( \sup_{t \in [0,T]} | X^\varepsilon_{N,i} - \hat{X}_i |^2(t) \Big) \\
\leq & 2\max_{i = 1, \dots, N} \ept \Big( \sup_{t \in [0,T]} | X^\varepsilon_{N,i} -  \bar{X}^\varepsilon_i |^2(t)  \Big)
+ 2\max_{i = 1, \dots, N} \ept \Big( \sup_{t \in [0,T]} |  \bar{X}^\varepsilon_i - \hat{X}_i |^2(t) \Big) \\
\leq & C N^{-\beta} +  C \varepsilon^{2}\leq C\varepsilon^2,
\end{align*}
which implies \eqref{r3}.

\subsection{Propagation of chaos in the strong sense}\label{relative entropy}
This subsection is devoted to the proof of Theorem \ref{relative}, which can be divided into two parts.
In the first part, we derive the $L^\infty$ bound for $\nabla\log u^\varepsilon$, which will be used in the second part for the discussion of propagation of chaos.
\subsubsection{$L^\infty$ estimate for $\nabla\log u^\varepsilon$}
First of all, the $L^\infty$ estimate for $u^\varepsilon$ is shown by using a modified Moser iteration.
\begin{lemma}\label{Linfty}
Under the assumptions of Theorem \ref{theorempde},
we suppose that $u_0\in L^\infty(\mathbb{R}^2)$ and
$u^\varepsilon$ is a weak solution to the system \eqref{kellersegelmedium},
then there is a constant $C$ independent of $\varepsilon$ such that
\begin{align}
\|u^\varepsilon\|_{L^\infty((0,T)\times\R^2)}\leq C.
\end{align}
\end{lemma}
\begin{proof}
A multiplication of \eqref{kellersegelmedium} by
$p_k{u^\varepsilon}^{p_k-1}$ $(p_k=2^k \mbox{ for } k=1,2,\ldots)$ and integration over $\mathbb{R}^2$ leads to
\begin{align*}
\frac{d}{dt}\int_{\mathbb{R}^2}{u^\varepsilon}^{p_k}dx
=&-p_k(p_k-1)\int_{\mathbb{R}^2}{u^\varepsilon}^{p_k-2}|\nabla u^\varepsilon|^2 dx\\
&+p_k(p_k-1)\int_{\mathbb{R}^2}e^{-v^\varepsilon}{u^\varepsilon}^{p_k-1}\nabla v^\varepsilon\cdot\nabla u^\varepsilon dx
-p_k(p_k-1)\int_{\mathbb{R}^2}e^{-v^\varepsilon}{u^\varepsilon}^{p_k-2}|\nabla u^\varepsilon|^2 dx\\
\leq& -\frac{3p_k(p_k-1)}{4}\int_{\mathbb{R}^2}{u^\varepsilon}^{p_k-2}|\nabla u^\varepsilon|^2 dx
+p_k(p_k-1)\|\nabla v^\varepsilon\|^2_{L^\infty(\mathbb{R}^2)}\int_{\mathbb{R}^2}{u^\varepsilon}^{p_k} dx\\
\leq& -\frac{3(p_k-1)}{p_k}\int_{\mathbb{R}^2}|\nabla {u^\varepsilon}^{p_k/2}|^2 dx
+p_k(p_k-1)\|\nabla v^\varepsilon\|^2_{L^\infty(\mathbb{R}^2)}\int_{\mathbb{R}^2}{u^\varepsilon}^{p_k} dx.
\end{align*}
We deduce from the Nash inequality that
\begin{align}\label{re1}
\|u^\varepsilon\|_{L^{p_k}(\mathbb{R}^2)}^{p_k}&=\|{u^\varepsilon}^{p_k/2}\|_{L^2(\mathbb{R}^2)}^2
\leq C\|{u^\varepsilon}^{p_k/2}\|_{L^1(\mathbb{R}^2)}\|\nabla {u^\varepsilon}^{p_k/2}\|_{L^2(\mathbb{R}^2)}\nonumber\\
&=C\|u^\varepsilon\|^{p_{k-1}}_{L^{p_{k-1}}(\mathbb{R}^2)}\|\nabla {u^\varepsilon}^{p_k/2}\|_{L^2(\mathbb{R}^2)},
\end{align}
this together with $L^\infty((0,T)\times\mathbb{R}^2)$ bound for $\nabla v^\varepsilon$, which follows from \eqref{26} and the continuous embedding
$W^{1,p}(\mathbb{R}^2)\hookrightarrow L^\infty(\mathbb{R}^2)$ (for $p>2$), show that
\begin{align}\label{re2}
\frac{d}{dt}\int_{\mathbb{R}^2}{u^\varepsilon}^{p_k}dx
&\leq -\frac{3(p_k-1)}{p_k}\int_{\mathbb{R}^2}|\nabla {u^\varepsilon}^{p_k/2}|^2dx
+Cp_k(p_k-1)\|u^\varepsilon\|_{L^{p_{k-1}}(\mathbb{R}^2)}^{p_{k-1}}
\|\nabla {u^\varepsilon}^{p_k/2}\|_{L^2(\mathbb{R}^2)}\nonumber\\
&\leq -\frac{p_k-1}{p_k}\int_{\mathbb{R}^2}|\nabla {u^\varepsilon}^{p_k/2}|^2dx
+\bar{C}\big(p_k(p_k-1)\|u^\varepsilon\|_{L^{p_{k-1}}(\mathbb{R}^2)}^{p_{k-1}} \big)^2,
\end{align}
where $\bar{C}>0$ appeared in this proof is a constant independent of $p_k$.
Taking into account \eqref{re1}, we have
\begin{align*}
\|u^\varepsilon\|_{L^{p_k}(\mathbb{R}^2)}^{p_k}
&\leq \frac{p_k-1}{p_k}\int_{\mathbb{R}^2} |\nabla {u^\varepsilon}^{p_k/2}|^2 dx
+\frac{p_k}{4(p_k-1)}\big(C\|u^\varepsilon\|_{L^{p_{k-1}}(\mathbb{R}^2)}^{p_{k-1}} \big)^2\nonumber\\
&\leq \frac{p_k-1}{p_k}\int_{\mathbb{R}^2} |\nabla {u^\varepsilon}^{p_k/2}|^2 dx
+\bar{C}\|u^\varepsilon\|^{2p_{k-1}}_{L^{p_{k-1}}(\mathbb{R}^2)}.
\end{align*}
We combine this result and \eqref{re2} to find that
\begin{align}\label{re20}
\frac{d}{dt}\|u^\varepsilon\|_{L^{p_k}(\mathbb{R}^2)}^{p_k}\leq
-\|u^\varepsilon\|_{L^{p_k}(\mathbb{R}^2)}^{p_k}+\bar{C}p_k^4\|u^\varepsilon\|_{L^{p_{k-1}}(\mathbb{R}^2)}^{2p_{k-1}}.
\end{align}
Let $y_k(t):=\|u^\varepsilon\|_{L^{p_k}(\mathbb{R}^2)}^{p_k}$. The inequality \eqref{re20} can be rewritten into
\begin{align*}
(e^t y_k(t))'\leq \bar{C}p_k^4e^ty_{k-1}^2(t)
\leq \bar{C} 16^k e^t\sup_{t\geq 0}y^2_{k-1}(t)
\leq a_ke^t\sup_{t\geq 0}y^2_{k-1}(t),
\end{align*}
where $a_k:=\bar{C}16^k$. This implies that
\begin{align*}
y_k(t)\leq 2a_k\max\{\sup_{t\geq 0}y^2_{k-1}(t),y_k(0) \}.
\end{align*}
Here
$$y_k(0)=\|u_0\|_{L^{p_k}(\mathbb{R}^2)}^{p_k}
\leq \max\{\|u_0\|_{L^1(\mathbb{R}^2)},\|u_0\|_{L^\infty(\mathbb{R}^2)} \}^{p_k}\leq D^{2^k}, $$
where $D:=\max\{1,\|u_0\|_{L^\infty(\mathbb{R}^2)} \}$.
After $k-1$ times iteration, we have
\begin{align*}
y_k(t)&\leq 2a_k\max\{\sup_{t\geq 0}y^2_{k-1}(t),D^{2^k} \}\\
&\leq (2a_k)(2a_{k-1})^2(2a_{k-2})^{2^2}\cdots(2a_1)^{2^{k-1}}\max\{\sup_{t\geq 0}y_0^{2^k}(t),D^{2^k} \}\\
&\leq (2\bar{C})^{2^k-1}16^{2^{k+1}-k-2}\max\{\sup_{t\geq 0}y_0^{2^k}(t),D^{2^k} \}.
\end{align*}
Taking the power $\frac{1}{p_k}=\frac{1}{2^k}$ to derive that
\begin{align*}
\|u^\varepsilon\|_{L^{p_k}(\mathbb{R}^2)}\leq 2\bar{C}16^2\max\{\sup_{t\geq 0}y_0(t),D \}\leq\bar{C}\max\{\sup_{t\geq 0}y_0(t),D \}.
\end{align*}
Therefore, we complete the proof of Lemma \ref{Linfty}.
\end{proof}
Lemma \ref{Linfty} allows us to conclude some estimates.
\begin{lemma}\label{uH2}
Under the assumptions of Theorem \ref{theorempde},
we suppose that $u_0\in L^\infty(\mathbb{R}^2)$ and
$(u^\varepsilon,v^\varepsilon)$ is a solution to the system \eqref{kellersegelmedium},
then there is a constant $C$ independent of $\varepsilon$ such that
\begin{align}
&\|u^\varepsilon\|_{L^\infty(0,T;H^2(\mathbb{R}^2))}+\|u^\varepsilon\|_{L^2(0,T;H^3(\mathbb{R}^2))}\leq C,\label{re8}\\
&\|v^\varepsilon\|_{L^\infty(0,T;H^4(\mathbb{R}^2))}+\|v^\varepsilon\|_{L^2(0,T;H^5(\mathbb{R}^2))}\leq C.\label{re9}
\end{align}
\end{lemma}
\begin{proof}
Let $\alpha$ be a multi-index of order $|\alpha|=1, 2$.
We apply $D^\alpha$ to \eqref{kellersegelmedium}, multiply the resulting equation by $D^\alpha u^\varepsilon$ and integrate over $\mathbb{R}^2$ to conclude that
\begin{align}\label{re7}
&\frac{1}{2}\frac{d}{dt}\int_{\mathbb{R}^2}|D^\alpha u^\varepsilon|^2 dx+\int_{\mathbb{R}^2}|\nabla D^\alpha u^\varepsilon|^2 dx\nonumber\\
=&-\int_{\mathbb{R}^2}\big(\nabla D^\alpha(e^{-v^\varepsilon}u^\varepsilon)
-e^{-v^\varepsilon}\nabla D^\alpha u^\varepsilon \big)\cdot \nabla D^\alpha u^\varepsilon dx
-\int_{\mathbb{R}^2}e^{-v^\varepsilon}|\nabla D^\alpha u^\varepsilon|^2 dx\nonumber\\
\leq& \frac{1}{2}\int_{\mathbb{R}^2}|\nabla D^\alpha u^\varepsilon|^2 dx
+C\int_{\mathbb{R}^2}|\nabla D^\alpha(e^{-v^\varepsilon}u^\varepsilon)-e^{-v^\varepsilon}\nabla D^\alpha u^\varepsilon|^2 dx\nonumber\\
\leq& \frac{1}{2}\int_{\mathbb{R}^2}|\nabla D^\alpha u^\varepsilon|^2 dx
+C\|\nabla(e^{-v^\varepsilon})\|^2_{L^\infty(\mathbb{R}^2)}\|D^\alpha u^\varepsilon\|^2_{L^2(\mathbb{R}^2)}
+C\|\nabla D^\alpha(e^{-v^\varepsilon})\|^2_{L^2(\mathbb{R}^2)}\|u^\varepsilon\|^2_{L^\infty(\mathbb{R}^2)}.
\end{align}

First, we deal with the case of $|\alpha|=1$.
Estimate \eqref{26} and the continuous embedding $W^{1,p}(\mathbb{R}^2)\hookrightarrow L^4(\mathbb{R}^2)$ with $p=4/3$ give
\begin{align*}
\|\nabla D^\alpha(e^{-v^\varepsilon})\|_{L^2(\mathbb{R}^2)}^2
=\|e^{-v^\varepsilon}(D^\alpha v^\varepsilon\nabla v^\varepsilon-\nabla D^\alpha v^\varepsilon)\|_{L^2(\mathbb{R}^2)}^2
\leq \|D^\alpha v^\varepsilon\|_{L^4(\mathbb{R}^2)}^4+\|\nabla D^\alpha v^\varepsilon\|_{L^2(\mathbb{R}^2)}^2\leq C.
\end{align*}
We combine this result and \eqref{re7}, Lemma \ref{Linfty} and a uniform $L^\infty((0,T)\times\mathbb{R}^2)$ bound for $\nabla (e^{-v^{\varepsilon}})$  to find that
\begin{align}\label{re4}
\frac{1}{2}\frac{d}{dt}\int_{\mathbb{R}^2}|D^\alpha u^\varepsilon|^2 dx+ \frac{1}{2}\int_{\mathbb{R}^2}|\nabla D^\alpha u^\varepsilon|^2 dx
\leq
C\int_{\mathbb{R}^2}|D^\alpha u^\varepsilon|^2 dx+C.
\end{align}
Because of the inequality \eqref{re4}, we can apply Gronwall's inequality to infer a uniform $L^\infty(0,T;L^2(\mathbb{R}^2))$ bound for $D^\alpha u^\varepsilon$
and a uniform $L^2((0,T)\times\mathbb{R}^2)$ bound for
 $\nabla D^\alpha u^\varepsilon$. Thus, in accordance with \eqref{Lp} with $p=2$, we obtain that
\begin{align}\label{re5}
\|u^\varepsilon\|_{L^\infty(0,T;H^1(\mathbb{R}^2))}+\|u^\varepsilon\|_{L^2(0,T;H^2(\mathbb{R}^2))}\leq C,
\end{align}
which implies that, for $|\beta|=3$,
\begin{align}\label{re6}
\|D^\beta v^\varepsilon\|_{L^\infty(0,T;L^2(\mathbb{R}^2))}+\|\nabla D^{\beta}v^\varepsilon\|_{L^2((0,T)\times\mathbb{R}^2)}
\leq& C\|\Delta D^\alpha v^\varepsilon\|_{L^\infty(0,T;L^2(\mathbb{R}^2))}
+C\|\Delta \nabla D^{\alpha} v^\varepsilon\|_{L^2((0,T)\times\mathbb{R}^2)}\nonumber\\
\leq& C\|D^\alpha v^\varepsilon\|_{L^\infty(0,T;L^2(\mathbb{R}^2))}
+C\|D^\alpha u^\varepsilon\ast j^\varepsilon\|_{L^\infty(0,T;L^2(\mathbb{R}^2))}\nonumber\\
&+C\|\nabla D^\alpha v^\varepsilon\|_{L^2((0,T)\times\mathbb{R}^2)}
+C\|\nabla D^\alpha u^\varepsilon\ast j^\varepsilon\|_{L^2((0,T)\times\mathbb{R}^2)}\nonumber\\
\leq& C\|D^\alpha v^\varepsilon\|_{L^\infty(0,T;L^2(\mathbb{R}^2))}
+C\|D^\alpha u^\varepsilon\|_{L^\infty(0,T;L^2(\mathbb{R}^2))}\nonumber\\
&+C\|\nabla D^\alpha v^\varepsilon\|_{L^2((0,T)\times\mathbb{R}^2)}+C\|\nabla D^\alpha u^\varepsilon\|_{L^2((0,T)\times\mathbb{R}^2)}
\leq C.
\end{align}

Now we turn to the case of $|\alpha|=2$.
The estimates \eqref{26} and \eqref{re6} give a uniform $L^\infty(0,T;L^2(\mathbb{R}^2))$ bound for $\nabla D^{\alpha}(e^{-v^\varepsilon})$.
Together with \eqref{re7} and Lemma \ref{Linfty}, we infer that
\begin{align}\label{re21}
\frac{1}{2}\frac{d}{dt}\int_{\mathbb{R}^2}|D^\alpha u^\varepsilon|^2 dx+\frac{1}{2}\int_{\mathbb{R}^2}|\nabla D^\alpha u^\varepsilon|^2 dx
\leq
C\int_{\mathbb{R}^2}|D^\alpha u^\varepsilon|^2 dx+C.
\end{align}
Thanks to \eqref{re5}, \eqref{re21} and Gronwall's inequality, we get \eqref{re8}.
Following the same way as \eqref{re6}, we deduce from \eqref{26}, \eqref{re8} and \eqref{re6} that, for $|\beta|=4$,
\begin{align*}
\|D^{\beta}v^\varepsilon\|_{L^\infty(0,T;L^2(\mathbb{R}^2))}+\|\nabla D^{\beta}v^\varepsilon\|_{L^2((0,T)\times\mathbb{R}^2)}
\leq& C\|D^\alpha v^\varepsilon\|_{L^\infty(0,T;L^2(\mathbb{R}^2))}+C\|D^\alpha u^\varepsilon\|_{L^\infty(0,T;L^2(\mathbb{R}^2))}\\
&+C\|\nabla D^\alpha v^\varepsilon\|_{L^2((0,T)\times\mathbb{R}^2)}+C\|\nabla D^\alpha u^\varepsilon\|_{L^2((0,T)\times\mathbb{R}^2)}\leq C,
\end{align*}
achieving \eqref{re9}.
\end{proof}
Based on Lemmas \ref{Linfty} and \ref{uH2}, we derive the uniform $L^\infty$ bound for $\nabla\log u^\varepsilon$ in the following discussion.
For notational simplicity, we use $\partial_i$ $(i=1,2)$ to denote $\partial_{x_i}$ $(i=1,2)$ in this subsection.

Since $\nabla\log u_0\in W^{1,p}(\mathbb{R}^2)$ $(p>2)$, we know by Sobolev embedding theorem that $\nabla\log u_0\in L^\infty(\R^2)$.
Now we denote $\|\nabla\log u_0\|_{W^{1,p}(\R^2)}=M_0$, we will prove that there exists
 a time $T^*\in (0,T)$ such that $\|\nabla\log u^\varepsilon\|_{L^\infty(0,T^*:W^{1,p}(\R^2))}$ is bounded uniformly in $\varepsilon$. Here
$\partial_i\log u^\varepsilon=\partial_i\rho^\varepsilon$ $(i=1,2)$ is the solution to the following problem
on the time interval $(0,T)$ for any fixed $T>0$
\begin{align}\label{re10}
\begin{cases}
\partial_t\partial_i\rho^\varepsilon(t,x)=&(e^{-v^\varepsilon}+1)\Delta\partial_i\rho^\varepsilon-e^{-v^\varepsilon}\partial_i v^\varepsilon\Delta\rho^\varepsilon
+2(e^{-v^\varepsilon}+1)\nabla\rho^\varepsilon\cdot\partial_i\nabla\rho^\varepsilon-e^{-v^\varepsilon}\partial_i v^\varepsilon|\nabla\rho^\varepsilon|^2\\
&+2e^{-v^\varepsilon}\partial_i v^\varepsilon\nabla v^\varepsilon\cdot\nabla\rho^\varepsilon
-2e^{-v^\varepsilon}\partial_i\nabla v^\varepsilon\cdot \nabla\rho^\varepsilon
-2e^{-v^\varepsilon}\nabla v^\varepsilon\cdot\partial_i\nabla\rho^\varepsilon-e^{-v^\varepsilon}\partial_i v^\varepsilon|\nabla v^\varepsilon|^2\\
&+2e^{-v^\varepsilon}\nabla v^\varepsilon\cdot\partial_i\nabla v^\varepsilon
+e^{-v^\varepsilon}\partial_i v^\varepsilon\Delta v^\varepsilon-e^{-v^\varepsilon}\Delta\partial_i v^\varepsilon,\\
\partial_i\rho^\varepsilon(0,x)=&\partial_i\log u_0.
\end{cases}
\end{align}
Or equivalently, for any fixed $T>0$, we can find a small $M_0$, so that  $\|\nabla\log u^\varepsilon\|_{L^\infty(0,T:W^{1,p}(\R^2))}$ is uniformly bounded in $\varepsilon$.

We consider the following initial problem for a vector valued function $\mathbf{p}=(p_1,p_2):[0,T]\times\R^2\rightarrow\R^2$:
\begin{align}\label{pproblem}
	\begin{cases}
		\partial_t \mathbf{p}=&(e^{-v^\varepsilon}+1)\Delta\mathbf{p}-e^{-v^\varepsilon}\nabla v^\varepsilon\nabla\cdot \mathbf{p}
		+2(e^{-v^\varepsilon}+1)\mathbf{p}\cdot\nabla \mathbf{p}-e^{-v^\varepsilon}\nabla v^\varepsilon|\mathbf{p}|^2\\
		&+2e^{-v^\varepsilon}\nabla v^\varepsilon\nabla v^\varepsilon\cdot\mathbf{p}
		-2e^{-v^\varepsilon}D^2 v^\varepsilon\cdot \mathbf{p}
		-2e^{-v^\varepsilon}\nabla v^\varepsilon\cdot\nabla\mathbf{p}+\mathbf{f}(v^\varepsilon),\\
	 (\mathbf{f}(v^\varepsilon))_i=&-e^{-v^\varepsilon}\partial_i v^\varepsilon|\nabla v^\varepsilon|^2+2e^{-v^\varepsilon}\nabla v^\varepsilon\cdot\partial_i\nabla v^\varepsilon
		+e^{-v^\varepsilon}\partial_i v^\varepsilon\Delta v^\varepsilon-e^{-v^\varepsilon}\Delta\partial_i v^\varepsilon,\\
		\mathbf{p}|_{t=0}=&\nabla\log u_0=\mathbf{p}_0, \qquad \|\mathbf{p}_0\|_{W^{1,p}(\mathbb{R}^2)}=M_0, \quad p>2.
	\end{cases}
\end{align}
Obviously by using the estimates obtained in \eqref{re9} for $v^\varepsilon$,
we deduce that $\mathbf{f}(v^\varepsilon)$ is uniformly bounded in $L^\infty(0,T;L^{p}(\R^2))$.

Next we wish to prove the following well-posedness result for the problem \eqref{pproblem}.
\begin{lemma}\label{W1p}
	There exists a time $T^*>0$ such that the initial value problem of PDE system \eqref{pproblem} has a unique weak solution $\mathbf{p}$
    in $L^\infty(0,T^*;W^{1,p}(\R^2))$ $(p>2)$ with estimate
	\begin{equation}\label{M0estimate}
	\|\mathbf{p}\|_{L^\infty(0,T^*;W^{1,p}(\mathbb{R}^2))}\leq 2M_0.
	\end{equation}
\end{lemma}

\begin{remark}
Since $\nabla \rho^{\varepsilon}$ is a solution of \eqref{pproblem}, one directly obtains that $\mathbf{p}=\nabla \rho^{\varepsilon}$.
Therefore, the estimate \eqref{M0estimate} and  Sobolev embedding theorem imply that
\begin{align}\label{re23}
	\|\nabla\rho^\varepsilon\|_{L^\infty((0,T^*)\times\mathbb{R}^2)}\leq 2\tilde{C}M_0,
\end{align}
where $\tilde{C}$ is the optimal constant in the Sobolev inequality.
\end{remark}

\begin{proof}
	To use Banach fixed point theorem, we consider the following complete space
	\begin{equation}
		\label{Xspace}
		\mathcal{X}_{\bar{T}}=\{\mathbf{p}: \|\mathbf{p}\|_{L^\infty(0,\bar{T};W^{1,p}(\mathbb{R}^2))}\leq 2M_0 \}
	\end{equation}
with the metric, $\forall \mathbf{p},\mathbf{q}\in\mathcal{X}_{\bar{T}}$,
$$
d(\mathbf{p},\mathbf{q}):=\|\mathbf{p}-\mathbf{q}\|_{L^\infty(0,\bar{T};W^{1,p}(\mathbb{R}^2))}.
$$
In the above space, $\bar T$ is to be determined.
First for given $\mathbf{q}\in\mathcal{X}_{\bar{T}}$, we consider
\begin{align}\label{pqproblem}
	\begin{cases}
		\partial_t \mathbf{p}=&(e^{-v^\varepsilon}+1)\Delta\mathbf{p}-e^{-v^\varepsilon}\nabla v^\varepsilon\nabla\cdot \mathbf{p}
		+2(e^{-v^\varepsilon}+1)\mathbf{q}\cdot\nabla \mathbf{p}-e^{-v^\varepsilon}\nabla v^\varepsilon\mathbf{q}\cdot\mathbf{p}\\
		&+2e^{-v^\varepsilon}\nabla v^\varepsilon\nabla v^\varepsilon\cdot\mathbf{p}
		-2e^{-v^\varepsilon}D^2 v^\varepsilon\cdot \mathbf{p}
		-2e^{-v^\varepsilon}\nabla v^\varepsilon\cdot\nabla\mathbf{p}+\mathbf{f}(v^\varepsilon),\\
		\mathbf{p}|_{t=0}=&\nabla\log u_0=\mathbf{p}_0, \qquad \|\mathbf{p}_0\|_{W^{1,p}(\mathbb{R}^2)}=M_0, \quad p>2.
	\end{cases}
\end{align}
We use $C$ to denote a positive constant independent of $\varepsilon$ in the following discussion.
Sobolev embedding theorem yields that $\|\mathbf{q}\|_{L^\infty((0,\bar{T})\times\R^2)}\leq CM_0$.
Multiplying the equation \eqref{pqproblem} by  $|p_i|^{p-2}p_i$ $(i=1,2)$ and integrating over $\mathbb{R}^2$ leads to
\begin{align}
	\frac{1}{p}\frac{d}{dt}\int_{\mathbb{R}^2}|p_i|^p dx
	=&\int_{\mathbb{R}^2}(e^{-v^\varepsilon}+1)\Delta p_i |p_i|^{p-2}p_i dx
	-\int_{\mathbb{R}^2}e^{-v^\varepsilon}\partial_i v^\varepsilon\nabla\cdot \mathbf{p}|p_i|^{p-2}p_i dx\nonumber\\
	&+2\int_{\mathbb{R}^2}(e^{-v^\varepsilon}+1)\mathbf{q}\cdot\nabla p_i|p_i|^{p-2}p_i dx
	-\int_{\mathbb{R}^2}e^{-v^\varepsilon}\partial_i v^\varepsilon\mathbf{q}\cdot\mathbf{p}|p_i|^{p-2}p_i dx\nonumber\\
	&+2\int_{\mathbb{R}^2}e^{-v^\varepsilon}\partial_i v^\varepsilon\nabla v^\varepsilon\cdot\mathbf{p}|p_i|^{p-2}p_i dx
	-2\int_{\mathbb{R}^2}e^{-v^\varepsilon}\partial_i\nabla v^\varepsilon\cdot\mathbf{p}|p_i|^{p-2}p_i dx\nonumber\\
	&-2\int_{\mathbb{R}^2}e^{-v^\varepsilon}\nabla v^\varepsilon\cdot\nabla p_i|p_i|^{p-2}p_i dx
	+\int_{\mathbb{R}^2}(\mathbf{f}(v^\varepsilon))_i|p_i|^{p-2}p_i dx	=\sum_{l=1}^{8}I_l.\label{Lpest}
\end{align}

Now we are going to estimate the eight terms on the right hand side of \eqref{Lpest}.
The estimate \eqref{26} and the continuous embedding $W^{1,p}(\mathbb{R}^2)\hookrightarrow L^\infty(\mathbb{R}^2)$ (for $p>2$) give the $L^\infty((0,T)\times\mathbb{R}^2)$ bound for $\nabla v^\varepsilon$.
By doing integral by parts and using H\"older's inequality, we have
\begin{align*}
	I_1=&\int_{\mathbb{R}^2}e^{-v^\varepsilon}\nabla v^\varepsilon\cdot \nabla p_i|p_i|^{p-2}p_i dx
	-\int_{\R^2}(e^{-v^\varepsilon}+1)\nabla p_i\cdot \nabla(|p_i|^{p-2}p_i) dx\\
     =&\frac{2}{p}\int_{\R^2}e^{-v^\varepsilon}\nabla v^\varepsilon\cdot |p_i|^{\frac{p}{2}}\nabla|p_i|^{\frac{p}{2}} dx
     -(p-1)\int_{\mathbb{R}^2}(e^{-v^\varepsilon}+1)|p_i|^{p-2}|\nabla p_i|^2 dx\\
	\leq& \frac{p-1}{2p^2}\int_{\mathbb{R}^2}\Big|\nabla|p_i|^{\frac{p}{2}}\Big|^2 dx
	+C\|\nabla v^\varepsilon\|^2_{L^\infty(\mathbb{R}^2)}\int_{\mathbb{R}^2}|p_i|^p dx
	-\frac{4(p-1)}{p^2}\int_{\mathbb{R}^2}\Big|\nabla|p_i|^{\frac{p}{2}}\Big|^2 dx\\
	\leq& -\frac{7(p-1)}{2p^2}\int_{\mathbb{R}^2}\Big|\nabla|p_i|^{\frac{p}{2}}\Big|^2 dx
	+C\int_{\mathbb{R}^2}|p_i|^p dx.
\end{align*}
It follows from \eqref{re9} and the continuous embedding $H^2(\mathbb{R}^2)\hookrightarrow L^\infty(\mathbb{R}^2)$ that
$\nabla\partial_i v^\varepsilon$ is bounded in $L^\infty((0,T)\times\mathbb{R}^2)$.
Thanks to integration by parts, H\"older's inequality and Young's inequality, we deduce that
\begin{align*}
	I_2=&
	(p-1)\int_{\mathbb{R}^2}e^{-v^\varepsilon}\partial_i v^\varepsilon \mathbf{p}\cdot |p_i|^{p-2}\nabla p_i dx
	-\int_{\mathbb{R}^2}e^{-v^\varepsilon}\nabla v^\varepsilon\partial_i v^\varepsilon\cdot\mathbf{p}|p_i|^{p-2}p_i dx
	+\int_{\mathbb{R}^2}e^{-v^\varepsilon}\nabla\partial_i v^\varepsilon\cdot\mathbf{p}|p_i|^{p-2}p_i dx\\
	\leq &\frac{p-1}{8}\int_{\mathbb{R}^2}|p_i|^{p-2}|\nabla p_i|^2 dx
	+C\|\nabla v^\varepsilon\|_{L^\infty(\mathbb{R}^2)}^2\int_{\mathbb{R}^2}|\mathbf{p}|^p dx
	+\|\nabla\partial_i v^\varepsilon\|_{L^\infty(\mathbb{R}^2)}\int_{\mathbb{R}^2}|\mathbf{p}|^p dx\\
	\leq &\frac{p-1}{2p^2}\int_{\mathbb{R}^2}\Big|\nabla|p_i|^{\frac{p}{2}}\Big|^2 dx
	+C\int_{\mathbb{R}^2}|\mathbf{p}|^p dx.
\end{align*}
Similarly, \eqref{re9} and \eqref{re23} allow us to use H\"older's inequality and Young's inequality to obtain that
\begin{align*}
	I_3&=\frac{4}{p}\int_{\mathbb{R}^2}(e^{-v^\varepsilon}+1)\mathbf{q}\cdot\nabla |p_i|^{\frac{p}{2}}|p_i|^{\frac{p}{2}} dx\leq \frac{p-1}{2p^2}\int_{\mathbb{R}^2}\Big|\nabla|p_i|^{\frac{p}{2}}\Big|^2 dx
	+C\|\mathbf{q}(t,\cdot)\|^2_{L^\infty(\mathbb{R}^2)}\int_{\mathbb{R}^2}|\mathbf{p}|^p dx\\
	&\leq \frac{p-1}{2p^2}\int_{\mathbb{R}^2}\Big|\nabla|p_i|^{\frac{p}{2}}\Big|^2 dx
	+CM_0^2\int_{\mathbb{R}^2}|\mathbf{p}|^p dx,\\
	I_4&\leq \|\partial_i v^\varepsilon\|_{L^\infty(\mathbb{R}^2)}\|\mathbf{q}(t,\cdot)\|_{L^\infty(\mathbb{R}^2)}
	\int_{\mathbb{R}^2}|\mathbf{p}|^p dx
	\leq CM_0\int_{\mathbb{R}^2}|\mathbf{p}|^p dx,\\
	I_5+I_6&\leq 2(\|\nabla v^\varepsilon\|^2_{L^\infty(\mathbb{R}^2)}+\|\partial_i \nabla v^\varepsilon\|_{L^\infty(\mathbb{R}^2)})\int_{\mathbb{R}^2}|\mathbf{p}|^p dx
	\leq C\int_{\mathbb{R}^2}|\mathbf{p}|^p dx,\\
	I_7&\leq \frac{p-1}{2p^2}\int_{\mathbb{R}^2}\Big|\nabla|p_i|^{\frac{p}{2}}\Big|^2 dx
	+C\|\nabla v^\varepsilon\|^2_{L^\infty(\mathbb{R}^2)}\int_{\mathbb{R}^2}|\mathbf{p}|^p dx
	\leq \frac{p-1}{2p^2}\int_{\mathbb{R}^2}\Big|\nabla|p_i|^{\frac{p}{2}}\Big|^2 dx
	+C\int_{\mathbb{R}^2}|\mathbf{p}|^p dx,\\
	I_8&\leq \int_{\mathbb{R}^2}|p_i|^p dx+C\int_{\mathbb{R}^2}|\mathbf{f}(v^\varepsilon)|^{p} dx
	\leq \int_{\mathbb{R}^2}|p_i|^p dx+C.
\end{align*}
We combine the estimates for $I_1-I_8$ and \eqref{Lpest}, and do summation $\sum_{i=1}^2$ to obtain that
\begin{align}\label{re12}
	\frac{1}{p}\frac{d}{dt}\int_{\mathbb{R}^2}|\mathbf{p}|^p dx
	+\frac{2(p-1)}{p^2}\int_{\mathbb{R}^2}\sum_{i=1}^{2}\Big|\nabla|p_i|^{\frac{p}{2}}\Big|^2 dx
	\leq C(M_0^2+1)\int_{\mathbb{R}^2}|\mathbf{p}|^p dx+C.
\end{align}

Next we do the $L^p$ estimate for $\nabla\mathbf{p}$.
Applying $\partial_j$ $(j=1,2)$ to \eqref{pqproblem}, multiplying the resulting equation by $|\partial_j p_i|^{p-2}\partial_j p_i$ $(i,j=1,2)$
and integrating over $\mathbb{R}^2$, we deduce that
\begin{align}\label{Lpjest}
	\frac{1}{p}\frac{d}{dt}\int_{\mathbb{R}^2}|\partial_j p_i|^p dx
	=&\int_{\mathbb{R}^2}\partial_j\big((e^{-v^\varepsilon}+1)\Delta p_i \big)|\partial_j p_i|^{p-2}\partial_j p_i dx
	-\int_{\mathbb{R}^2}\partial_j(e^{-v^\varepsilon}\partial_i v^\varepsilon\nabla\cdot \mathbf{p})|\partial_j p_i|^{p-2}\partial_j p_i dx\nonumber\\
	&+2\int_{\mathbb{R}^2}\partial_j\big((e^{-v^\varepsilon}+1)\mathbf{q}\cdot\nabla p_i \big)|\partial_j p_i|^{p-2}\partial_j p_idx
	-\int_{\mathbb{R}^2}\partial_j(e^{-v^\varepsilon}\partial_i v^\varepsilon\mathbf{q}\cdot\mathbf{p})|\partial_j p_i|^{p-2}\partial_j p_i dx\nonumber\\
	&+2\int_{\mathbb{R}^2}\partial_j(e^{-v^\varepsilon}\partial_i v^\varepsilon\nabla v^\varepsilon\cdot \mathbf{p})|\partial_j p_i|^{p-2}\partial_j p_i dx
	-2\int_{\mathbb{R}^2}\partial_j(e^{-v^\varepsilon}\partial_i\nabla v^\varepsilon\cdot \mathbf{p})|\partial_j p_i|^{p-2}\partial_j p_i dx\nonumber\\
	&-2\int_{\mathbb{R}^2}\partial_j(e^{-v^\varepsilon}\nabla v^\varepsilon\cdot\nabla p_i)|\partial_j p_i|^{p-2}\partial_j p_i dx
	-\int_{\mathbb{R}^2}\partial_j(\mathbf{f}(v^\varepsilon))_i|\partial_j p_i|^{p-2}\partial_j p_i dx
    =\sum_{l=1}^{8}K_l.
\end{align}
The first term on the right-hand side of \eqref{Lpjest} can be estimated by \eqref{re9} and H\"older's inequality
\begin{align}\label{re14}
	K_1=&\int_{\mathbb{R}^2}\partial_j\big(e^{-v^\varepsilon}\Delta p_i \big)|\partial_j p_i|^{p-2}\partial_j p_i dx
	-(p-1)\int_{\mathbb{R}^2}|\partial_j p_i|^{p-2}|\nabla(\partial_j p_i)|^2 dx\nonumber\\
	=&-\int_{\mathbb{R}^2}\partial_jv^\varepsilon e^{-v^\varepsilon}\Delta p_i |\partial_j p_i|^{p-2}\partial_j p_i dx\nonumber\\
	&+\int_{\mathbb{R}^2}e^{-v^\varepsilon}\nabla v^\varepsilon\cdot\nabla (\partial_j p_i) |\partial_j p_i|^{p-2}\partial_j p_idx
-\frac{4(p-1)}{p^2}\int_{\mathbb{R}^2}(e^{-v^\varepsilon}+1)\Big|\nabla |\partial_j p_i|^{\frac{p}{2}}\Big|^2 dx\nonumber\\
	=&\int_{\mathbb{R}^2}\nabla(\partial_jv^\varepsilon e^{-v^\varepsilon})\cdot\nabla p_i |\partial_j p_i|^{p-2}\partial_j p_i dx
+\int_{\mathbb{R}^2}(\partial_jv^\varepsilon e^{-v^\varepsilon})\nabla p_i\cdot \nabla(|\partial_j p_i|^{p-2}\partial_j p_i) dx\nonumber\\
	&+\int_{\mathbb{R}^2}e^{-v^\varepsilon}\nabla v^\varepsilon\cdot\nabla (\partial_j p_i) |\partial_j p_i|^{p-2}\partial_j p_idx
-\frac{4(p-1)}{p^2}\int_{\mathbb{R}^2}(e^{-v^\varepsilon}+1)\Big|\nabla |\partial_j p_i|^{\frac{p}{2}}\Big|^2 dx\nonumber\\
	\leq& (\|\nabla(\partial_j v^\varepsilon)\|_{L^\infty(\R^2)}+\|\nabla v^\varepsilon\|^2_{L^\infty(\R^2)})\int_{\mathbb{R}^2}|\nabla p_i|^p dx
+\|\nabla v^\varepsilon\|_{L^\infty(\R^2)}\int_{\mathbb{R}^2}|\nabla p_i \cdot\nabla(|\partial_j p_i|^{p-2}\partial_j p_i)| dx\nonumber\\
	&+\|\nabla v^\varepsilon\|_{L^\infty(\R^2)}\int_{\mathbb{R}^2}\big|\nabla (\partial_j p_i) |\partial_j p_i|^{p-2}\partial_j p_i\big|dx
-\frac{4(p-1)}{p^2}\int_{\mathbb{R}^2}\Big|\nabla |\partial_j p_i|^{\frac{p}{2}}\Big|^2 dx\nonumber\\
\leq& C\int_{\mathbb{R}^2}|\nabla p_i|^p dx
+C(p-1)\int_{\mathbb{R}^2}\big||\partial_j p_i|^{p-2}\nabla(\partial_jp_i)\cdot\nabla p_i\big| dx\nonumber\\
	&+C\int_{\mathbb{R}^2}\Big||\partial_j p_i|^{\frac{p}{2}}\nabla|\partial_j p_i|^{\frac{p}{2}}\Big|dx
-\frac{4(p-1)}{p^2}\int_{\mathbb{R}^2}\Big|\nabla |\partial_j p_i|^{\frac{p}{2}}\Big|^2 dx\nonumber\\
	\leq &-\frac{2(p-1)}{p^2}\int_{\mathbb{R}^2}\Big|\nabla |\partial_j p_i|^{\frac{p}{2}}\Big|^2 dx
	+C\int_{\mathbb{R}^2}|\nabla p_i|^p dx.
\end{align}
The second term $K_2$ can be estimated by H\"older's inequality and Young's inequality
\begin{align*}
K_2=&(p-1)\int_{\R^2}e^{-v^\varepsilon}\partial_i v^\varepsilon\nabla\cdot\mathbf{p}|\partial_jp_i|^{p-2}\partial_j(\partial_j p_i)dx\\
\leq&\frac{p-1}{16}\int_{\R^2}\big||\partial_j p_i|^{\frac{p-2}{2}}\partial_j(\partial_j p_i) \big|^2 dx
+C\|\partial_i v^\varepsilon\|_{L^\infty(\R^2)}^2\int_{\R^2}|\nabla\cdot\mathbf{p}|^2|\partial_jp_i|^{p-2}dx\\
\leq& \frac{p-1}{4p^2}\int_{\mathbb{R}^2}\Big|\nabla |\partial_j p_i|^{\frac{p}{2}}\Big|^2 dx +C\int_{\mathbb{R}^2}|\nabla \mathbf{p}|^p dx.
\end{align*}
The terms $K_3$ and $K_4$, which depend on $\mathbf{q}$, are needed to be estimated carefully by using  $\|\mathbf{q}\|_{L^\infty((0,\bar{T})\times\mathbb{R}^2)}\leq CM_0$. Namely,
\begin{align*}
	K_3+K_4=&-2(p-1)\int_{\R^2}(e^{-v^\varepsilon}+1)\mathbf{q}\cdot\nabla p_i|\partial_jp_i|^{p-2}\partial_j(\partial_jp_i)dx
    +(p-1)\int_{\R^2}e^{-v^\varepsilon}\partial_iv^\varepsilon\mathbf{q}\cdot\mathbf{p}|\partial_jp_i|^{p-2}\partial_j(\partial_jp_i)dx\\
    \leq &\frac{p-1}{16}\int_{\R^2}|\partial_jp_i|^{p-2}|\partial_j(\partial_jp_i)|^2 dx
    +C\|\mathbf{q}\|^2_{L^\infty(\R^2)}\int_{\R^2}|\nabla p_i|^2 dx
    +C\|\partial_i v\|^2_{L^\infty(\R^2)}\|\mathbf{q}\|^2_{L^\infty(\R^2)}\int_{\R^2}|\mathbf{p}|^2|\partial_jp_i|^{p-2}dx\\
	\leq & \frac{p-1}{4p^2}\int_{\mathbb{R}^2}\Big|\nabla |\partial_j p_i|^{\frac{p}{2}}\Big|^2 dx +CM_0^2\int_{\mathbb{R}^2}|\nabla \mathbf{p}|^p dx
     +CM_0^2\int_{\mathbb{R}^2}|\mathbf{p}|^p dx.
\end{align*}
We combine Young's inequality and a uniform $L^\infty(0,T;L^{p}(\R^2))$ bound for $\mathbf{f}(v^\varepsilon)$,
which follows from \eqref{re9}, to infer that
\begin{align*}
K_8
&=\int_{\R^2}(\mathbf{f}(v^\varepsilon))_i\partial_j (|\partial_j p_i|^{p-2}\partial_j p_i) dx
=\frac{2(p-1)}{p}\int_{\R^2}(\mathbf{f}(v^\varepsilon))_i|\partial_jp_i|^{\frac{p}{2}-2}\partial_jp_i\partial_j |\partial_j p_i|^{\frac{p}{2}} dx\\
		\leq &\frac{p-1}{4p^2}\int_{\mathbb{R}^2}\Big|\nabla |\partial_j p_i|^{\frac{p}{2}}\Big|^2 dx
+C\int_{\R^2}|\mathbf{f}(v^\varepsilon)|^2|\partial_jp_i|^{p-2}dx\\
			\leq &\frac{p-1}{4p^2}\int_{\mathbb{R}^2}\Big|\nabla |\partial_j p_i|^{\frac{p}{2}}\Big|^2 dx +C\int_{\R^2}|(\mathbf{f}(v^\varepsilon))|^pdx+C\int_{\R^2}|\partial_jp_i|^{p}dx\\
				\leq &\frac{p-1}{4p^2}\int_{\mathbb{R}^2}\Big|\nabla |\partial_j p_i|^{\frac{p}{2}}\Big|^2 dx +C\int_{\R^2}|\partial_jp_i|^{p}dx+C.		
\end{align*}
Thanks to \eqref{re9}, \eqref{re23}, by using Young's inequality and H\"older's inequality,
the rest terms on the right-hand side of \eqref{Lpjest} can be estimated as follows
\begin{align*}
K_5+K_6+K_7
=&2\int_{\R^2}\big(-e^{-v^\varepsilon}\partial_i v^\varepsilon\nabla v^\varepsilon\cdot\mathbf{p}
+e^{-v^\varepsilon}\partial_i\nabla v^\varepsilon\cdot\mathbf{p}
+e^{-v^\varepsilon}\nabla v^\varepsilon\cdot\nabla p_i\big)\partial_j(|\partial_jp_i|^{p-2}\partial_jp_i)dx\\
=&\frac{4(p-1)}{p}\int_{\R^2}\big(-e^{-v^\varepsilon}\partial_i v^\varepsilon\nabla v^\varepsilon\cdot\mathbf{p}
+e^{-v^\varepsilon}\partial_i\nabla v^\varepsilon\cdot\mathbf{p}
+e^{-v^\varepsilon}\nabla v^\varepsilon\cdot\nabla p_i\big)|\partial_jp_i|^{\frac{p}{2}-2}\partial_ip_i\partial_j|\partial_jp_i|^{\frac{p}{2}}dx\\
\leq&\frac{p-1}{4p^2}\int_{\mathbb{R}^2}\Big|\nabla|\partial_jp_i|^{\frac{p}{2}}\Big|^2 dx
+C\|\nabla v^\varepsilon\|^4_{L^\infty(\R^2)}\int_{\R^2}|\mathbf{p}|^2|\partial_jp_i|^{p-2}dx
+C\|\partial_i\nabla v^\varepsilon\|^2_{L^\infty(\R^2)}\int_{\R^2}|\mathbf{p}|^2|\partial_jp_i|^{p-2}dx\\
&+C\|\nabla v^\varepsilon\|^2_{L^\infty(\R^2)}\int_{\R^2}|\nabla p_i|^2|\partial_jp_i|^{p-2}dx\\
\leq&\frac{p-1}{4p^2}\int_{\mathbb{R}^2}\Big|\nabla|\partial_jp_i|^{\frac{p}{2}}\Big|^2 dx
+C\int_{\mathbb{R}^2}|\nabla\mathbf{p}|^p dx+C\int_{\mathbb{R}^2}|\mathbf{p}|^p dx.
\end{align*}
We combine the estimates for $K_1-K_8$ and \eqref{Lpjest}, and take $\sum_{i,j=1}^2$ to obtain that
\begin{align}\label{re15}
	&\frac{1}{p}\frac{d}{dt}\int_{\mathbb{R}^2}|\nabla\mathbf{p}|^p dx
   +\frac{p-1}{p^2}\int_{\mathbb{R}^2}\sum_{i,j=1}^2\Big|\nabla |\partial_j p_i|^{\frac{p}{2}}\Big|^2  dx
	\leq C(M_0^2+1)\int_{\mathbb{R}^2}|\nabla\mathbf{p}|^p dx
   +C(M_0^2+1)\int_{\mathbb{R}^2}|\mathbf{p}|^p dx+C.
\end{align}

Thus, in accordance with \eqref{re12} we have
\begin{align}\label{are15}
	\dfrac{d}{dt}\|\mathbf{p}\|^p_{W^{1,p}(\mathbb{R}^2)} \leq C_1(M_0^2+1) \|\mathbf{p}\|^p_{W^{1,p}(\mathbb{R}^2)} +C_2.
\end{align}
By \eqref{are15} and Gronwall's inequality, we deduce that
\begin{align}\label{are16}
	\sup_{t\in (0,\bar T)}\|\mathbf{p}\|^p_{W^{1,p}(\mathbb{R}^2)} &\leq
	(\|\mathbf{p}_0\|^p_{W^{1,p}(\mathbb{R}^2)} +C_2\bar{T})e^{C_1(M_0^2+1)\bar{T}}
	\leq(M_0^p +C_2\bar{T})e^{C_1(M_0^2+1)\bar{T}}\leq 2^pM_0^p,
\end{align}
where the last inequality is valid by choosing $\bar T>0$ small enough. Therefore, we build a map $\mathcal{X}_{\bar{T}}\rightarrow\mathcal{X}_{\bar{T}}$.

In the next, we prove that this map is a contraction in $L^\infty(0,T^*;W^{1,p}(\R^2))$ for some small $T^*\leq \bar{T}$.
Take $\mathbf{q}_1,\mathbf{q}_2\in \mathcal{X}_{\bar{T}}$, let $\mathbf{p}_1,\mathbf{p}_2$ be the corresponding solution of \eqref{pqproblem}.
Then $\bar{\mathbf{p}}=\mathbf{p}_1-\mathbf{p}_2$ solves the problem with $\bar{\mathbf{q}}=\mathbf{q}_1-\mathbf{q}_2$
\begin{align}\label{errorproblem}
	\begin{cases}
		\partial_t \bar{\mathbf{p}}=&(e^{-v^\varepsilon}+1)\Delta\bar{\mathbf{p}}-e^{-v^\varepsilon}\nabla v^\varepsilon\nabla\cdot\bar{\mathbf{p}}
		+2(e^{-v^\varepsilon}+1)\mathbf{q}_1\cdot\nabla \bar{\mathbf{p}}-e^{-v^\varepsilon}\nabla v^\varepsilon\mathbf{q}_1\cdot\bar{\mathbf{p}}\\
		&	+2(e^{-v^\varepsilon}+1)\bar{\mathbf{q}}\cdot\nabla \mathbf{p}_2-e^{-v^\varepsilon}\nabla v^\varepsilon\bar{\mathbf{q}}\cdot\mathbf{p}_2\\
		&+2e^{-v^\varepsilon}\nabla v^\varepsilon\nabla v^\varepsilon\cdot\bar{\mathbf{p}}
		-2e^{-v^\varepsilon}D^2 v^\varepsilon\cdot \bar{\mathbf{p}}
		-2e^{-v^\varepsilon}\nabla v^\varepsilon\cdot\nabla\bar{\mathbf{p}},\\
		\bar{\mathbf{p}}|_{t=0}=&0.
	\end{cases}
\end{align}
Then similar to the discussion for $L^p$ estimate above, by using $|\bar{\mathbf{p}}_i|^{p-2}\bar{\mathbf{p}}_i$ $(i=1,2)$ as test function, notice that $\mathbf{p}_2\in \mathcal{X}_{\bar{T}}$ we have by Sobolev embedding that
\begin{align*}
&\int_{\mathbb{R}^2} \Big(2(e^{-v^\varepsilon}+1)\bar{\mathbf{q}}\cdot\nabla (\mathbf{p}_2)_i-e^{-v^\varepsilon}\partial_i v^\varepsilon\bar{\mathbf{q}}\cdot\mathbf{p}_2\Big) |\bar{\mathbf{p}}_i|^{p-2}\bar{\mathbf{p}}_i dx\\
\leq  &\int_{\R^2} |\bar{\mathbf{p}}|^pdx +C\int_{\R^2} |\bar{\mathbf{q}}|^p(|\nabla\mathbf{p}_2|^p+|\mathbf{p}_2|^p)dx\leq \int_{\R^2} |\bar{\mathbf{p}}|^pdx
+C \|\bar{\mathbf{q}}(t,\cdot)\|_{L^\infty(\R^2)}^p\|\mathbf{p}_2\|^p_{W^{1,p}(\R^2)}\\
\leq &\int_{\R^2} |\bar{\mathbf{p}}|^pdx +CM_0^p \|\bar{\mathbf{q}}(t,\cdot)\|_{W^{1,p}(\R^2)}^p.
\end{align*}
Thus, we obtain that
\begin{align}\label{contraction}
	&\frac{1}{p}\frac{d}{dt}\int_{\mathbb{R}^2}|\bar{\mathbf{p}}|^p dx
  +\frac{p-1}{p^2}\int_{\mathbb{R}^2}\sum_{i=1}^2\Big|\nabla |\bar{\mathbf{p}}_i|^{\frac{p}{2}}\Big|^2  dx
	\leq C(M_0^2+1)\int_{\mathbb{R}^2}|\bar{\mathbf{p}}|^p dx+CM_0^p \|\bar{\mathbf{q}}(t,\cdot)\|_{W^{1,p}(\R^2)}^p.
\end{align}
Furthermore we do also the $L^p$ estimate for $\nabla\bar{\mathbf{p}}$.
By taking derivative $\partial_j$ $(j=1,2)$ for \eqref{errorproblem} and multiplying the resulting equation by $|\partial_j\bar{\mathbf{p}}_i|^{p-2}\partial_j\bar{\mathbf{p}}_i$ $(i,j=1,2)$, considering the nonlinear term
\begin{align*}
	&\int_{\mathbb{R}^2} \partial_j\Big(2(e^{-v^\varepsilon}+1)\bar{\mathbf{q}}\cdot\nabla (\mathbf{p}_2)_i-e^{-v^\varepsilon}\partial_i v^\varepsilon\bar{\mathbf{q}}\cdot\mathbf{p}_2\Big) |\partial_j\bar{\mathbf{p}}_i|^{p-2}\partial_j\bar{\mathbf{p}}_i dx\\
	=&-\int_{\mathbb{R}^2} \Big(2(e^{-v^\varepsilon}+1)\bar{\mathbf{q}}\cdot\nabla (\mathbf{p}_2)_i-e^{-v^\varepsilon}\partial_i v^\varepsilon\bar{\mathbf{q}}\cdot\mathbf{p}_2\Big)\partial_j (|\partial_j\bar{\mathbf{p}}_i|^{p-2}\partial_j\bar{\mathbf{p}}_i) dx\\
	\leq& \frac{p-1}{4p^2}\int_{\R^2}\Big|\nabla|\partial_j\bar{\mathbf{p}}_i|^{\frac{p}{2}}\Big|^2dx
 +C\int_{\R^2} |\bar{\mathbf{q}}|^2(|\nabla\mathbf{p}_2|^2+|\mathbf{p}_2|^2)|\nabla\bar{\mathbf{p}}|^{p-2}dx\\
	\leq& \frac{p-1}{4p^2}\int_{\R^2}\Big|\nabla|\partial_j\bar{\mathbf{p}}_i|^{\frac{p}{2}}\Big|^2dx
  +C\int_{\R^2}|\nabla\bar{\mathbf{p}}|^pdx+C\int_{\R^2}|\bar{\mathbf{q}}|^p(|\nabla\mathbf{p}_2|^p+|\mathbf{p}_2|^p)dx\\
	\leq &\frac{p-1}{4p^2}\int_{\R^2}\Big|\nabla|\partial_j\bar{\mathbf{p}}_i|^{\frac{p}{2}}\Big|^2dx
  +C\int_{\R^2}|\nabla\bar{\mathbf{p}}|^pdx+CM_0^p \|\bar{\mathbf{q}}(t,\cdot)\|_{W^{1,p}(\R^2)}^p,
\end{align*}
and taking summation $\sum_{i,j=1}^2$, we have that
\begin{align}\label{contraction2}
	&\frac{1}{p}\frac{d}{dt}\int_{\mathbb{R}^2}|\nabla\bar{\mathbf{p}}|^pdx
   +\frac{p-1}{p^2}\int_{\mathbb{R}^2}\sum_{i,j=1}^2\Big|\nabla|\partial_j\bar{\mathbf{p}}_i|^{\frac{p}{2}}\Big|^2  dx
	\leq C(M_0^2+1)\|\bar{\mathbf{p}}\|^p_{W^{1,p}(\R^2)}
   +CM_0^p \|\bar{\mathbf{q}}(t,\cdot)\|_{W^{1,p}(\R^2)}^p.
\end{align}
We combine \eqref{contraction} and \eqref{contraction2} to get
\begin{align}
\frac{d}{dt}\|\bar{\mathbf{p}}\|^p_{W^{1,p}(\R^2)}
\leq C_3(M_0^2+1)\|\bar{\mathbf{p}}\|^p_{W^{1,p}(\R^2)}
   +C_4M_0^p \|\bar{\mathbf{q}}(t,\cdot)\|_{W^{1,p}(\R^2)}^p.
\end{align}
This, together with Gronwall's inequality, we have
\begin{align*}
	\sup_{t\in (0, T^*)}\|\bar{\mathbf{p}}(t,\cdot)\|_{W^{1,p}(\R^2)}^p
   \leq C_4M_0^p T^*e^{C_3(M_0^2+1)T^*} \|\bar{\mathbf{q}}\|^p_{L^\infty(0,T^*;W^{1,p}(\R^2))}.
\end{align*}
If we choose $T^*\in (0,\bar{T})$ small enough such that $C_4M_0^p T^*e^{C_3(M_0^2+1)T^*}<1$,
we obtain that the map is a contraction.
Therefore, we conclude that \eqref{pproblem} has a unique weak solution $\mathbf{p}$ in $L^\infty(0,T^*;W^{1,p}(\R^2))$ $(p>2)$
with the help of Banach fixed point theorem. Moreover, $\mathbf{p}$ satisfies
\begin{align*}
\|\mathbf{p}\|_{L^\infty(0,T^*;W^{1,p}(\R^2))}\leq 2M_0,
\end{align*}
finishing the proof of Lemma \ref{W1p}.
\end{proof}

\subsubsection{Propagation of chaos}
Based on the uniform $L^\infty(0,T^*;H^2(\mathbb{R}^2))$ bound for $\nabla\log u^\varepsilon$,
we can expect the quantitative propagation of chaos result in $L^\infty(0,T^*;L^1(\mathbb{R}^{2k}))$-norm.
The Kolmogorov forward equation of \eqref{generalized_regularized_particle_model} reads as
\begin{align}
\label{relativeuN}
\begin{cases}
\partial_t u_N^\varepsilon=\displaystyle\sum_{i=1}^N\Delta_{x_i}\Big( \exp \Big(-\frac{1}{N}\displaystyle\sum_{j=1}^N\Phi^\varepsilon(x_i-x_j) \Big)u_N^\varepsilon+u_N^\varepsilon\Big),     \\
u_N^\varepsilon(0,x_1,\cdots,x_N)={u_0}^{\otimes N}=u_0(x_1)\cdots u_0(x_N),
\end{cases}
\end{align}
where $u_N^\varepsilon(t,x_1,\cdots,x_N)$ is the joint law of $N$ particles $(X^\varepsilon_{N,i})_{1\leq i\leq N}$.
For fixed $\varepsilon$, the global existence and uniqueness of a classical solution to the linear parabolic problem \eqref{relativeuN} can be obtained
as a trivial consequence of classical parabolic theory.
To derive a quantitative convergence result, we start with the error estimate between $u_N^\varepsilon$ and the product density ${u^\varepsilon}^{\otimes N}$, where $u^\varepsilon$ is the solution to \eqref{kellersegelmedium}.
It is easy to see
\begin{align}
\label{relativeuvare}
\begin{cases}\partial_t {u^\varepsilon}^{\otimes N}
=\displaystyle\sum_{i=1}^N\Delta_{x_i}\big(\exp (-\Phi^\varepsilon\ast u^\varepsilon(t,x_i) ){u^\varepsilon}^{\otimes N}+{u^\varepsilon}^{\otimes N}\big),     \\
{u^\varepsilon}^{\otimes N}(0,x_1,\cdots,x_N)=u_{0}(x_1)\cdots u_{0}(x_N).
\end{cases}
\end{align}
For $k\in\N$ and two probability density $f$ and $g$ on $\R^{2k}$, let
\begin{align*}
\mathcal{H}_k(f|g)
:=\int_{\R^{2k}}f\log\frac{f}{g}\,dx_1\cdots dx_k.
\end{align*}
We define the relative entropy between $u_N^\varepsilon$ and ${u^\varepsilon}^{\otimes N}$ as follows
\begin{align*}
\mathcal{H}(u_N^\varepsilon|{u^\varepsilon}^{\otimes N}):=\frac{1}{N}\mathcal{H}_N(u_N^\varepsilon|{u^\varepsilon}^{\otimes N})(t)
=\frac{1}{N}\int_{\R^{2N}}u_N^\varepsilon\log\frac{u_N^\varepsilon}{{u^\varepsilon}^{\otimes N}}\,dx_1\cdots dx_N.
\end{align*}
From \eqref{relativeuN} and \eqref{relativeuvare}, we deal with the evolution of the relative entropy as follows:
\begin{align}\label{relative1}
\frac{d}{dt}\mathcal{H}(u_N^\varepsilon|{u^\varepsilon}^{\otimes N})
=&\frac{1}{N}\int_{\mathbb{R}^{2N}}\Big(\partial_t u_N^\varepsilon\log \frac{u_N^\varepsilon}{{u^\varepsilon}^{\otimes N}}
-\frac{u_N^\varepsilon}{{u^\varepsilon}^{\otimes N}}\partial_t {u^\varepsilon}^{\otimes N}\Big)\,dx_1\cdots dx_N\nonumber\\
=& -\frac{1}{N}\sum_{i=1}^N\int_{\mathbb{R}^{2N}}\Big(\nabla_{x_i}u_N^\varepsilon-\nabla_{x_i}{u^\varepsilon}^{\otimes N}\frac{u_N^\varepsilon}{{u^\varepsilon}^{\otimes N}} \Big)
\cdot \nabla_{x_i}\log \frac{u_N^\varepsilon}{{u^\varepsilon}^{\otimes N}}\,dx_1\cdots dx_N \nonumber\\
&-\frac{1}{N}\sum_{i=1}^N\int_{\mathbb{R}^{2N}}\Big[\nabla_{x_i}\Big(\exp\Big(-\frac{1}{N}\sum_{j=1}^N\Phi^\varepsilon(x_i-x_j)\Big)u_N^\varepsilon \Big)
-\frac{u_N^\varepsilon}{{u^\varepsilon}^{\otimes N}}\nabla_{x_i}({u^\varepsilon}^{\otimes N}\exp(-\Phi^\varepsilon\ast u^\varepsilon(t,x_i))) \Big]\nonumber\\
&\phantom{xxxxxx}\cdot \nabla_{x_i}\log \frac{u_N^\varepsilon}{{u^\varepsilon}^{\otimes N}}\,dx_1\cdots dx_N\nonumber\\
=&-\frac{1}{N}\sum_{i=1}^N\int_{\mathbb{R}^{2N}}u_N^\varepsilon\Big|\nabla_{x_i}\log\frac{u_N^\varepsilon}{{u^\varepsilon}^{\otimes N}} \Big|^2\,dx_1\cdots dx_N\nonumber\\
&-\frac{1}{N}\sum_{i=1}^N\int_{\mathbb{R}^{2N}}\Big[u_N^\varepsilon\exp\Big(-\frac{1}{N}\sum_{j=1}^N\Phi^\varepsilon(x_i-x_j) \Big)
\Big(\nabla_{x_i}\log u_N^\varepsilon-\frac{1}{N}\sum_{j=1}^N\nabla_{x_i}\Phi^\varepsilon(x_i-x_j) \Big) \nonumber\\
&\phantom{xxxxx}-u_N^\varepsilon\exp(-\Phi^\varepsilon\ast u^\varepsilon(t,x_i))
\big(\nabla_{x_i}\log {u^\varepsilon}^{\otimes N}-\nabla_{x_i}\Phi^\varepsilon\ast u^\varepsilon(t,x_i) \big) \Big]
\cdot\nabla_{x_i}\log \frac{u_N^\varepsilon}{{u^\varepsilon}^{\otimes N}}\,dx_1\cdots dx_N.
\end{align}
This shows that
\begin{align}\label{relative2}
\frac{d}{dt}\mathcal{H}(u_N^\varepsilon|{u^\varepsilon}^{\otimes N})
=&-\frac{1}{N}\sum_{i=1}^N\int_{\mathbb{R}^{2N}}u_N^\varepsilon\Big|\nabla_{x_i}\log \frac{u_N^\varepsilon}{{u^\varepsilon}^{\otimes N}} \Big|^2\,dx_1\cdots dx_N\nonumber\\
&-\frac{1}{N}\sum_{i=1}^N\int_{\mathbb{R}^{2N}}u_N^\varepsilon\exp\Big(-\frac{1}{N}\sum_{j=1}^N\Phi^\varepsilon(x_i-x_j) \Big)
\Big|\nabla_{x_i}\log\frac{u_N^\varepsilon}{{u^\varepsilon}^{\otimes N}} \Big|^2\,dx_1\cdots dx_N
+I_1+I_2,
\end{align}
where
\begin{align*}
I_1:=&-\frac{1}{N}\sum_{i=1}^N\int_{\mathbb{R}^{2N}}u_N^\varepsilon\exp\Big(-\frac{1}{N}\sum_{j=1}^N\Phi^\varepsilon(x_i-x_j) \Big)\nonumber\\
&\phantom{xxxxxxx}\cdot\Big(\nabla_{x_i}\Phi^\varepsilon\ast u^\varepsilon(t,x_i)-\frac{1}{N}\sum_{j=1}^N\nabla_{x_i}\Phi^\varepsilon(x_i-x_j) \Big)
\cdot\nabla_{x_i}\log \frac{u_N^\varepsilon}{{u^\varepsilon}^{\otimes N}}\,dx_1\cdots dx_N,
\end{align*}
and
\begin{align*}
I_2=:&-\frac{1}{N}\sum_{i=1}^N\int_{\mathbb{R}^{2N}}u_N^\varepsilon
\Big(\exp\Big(-\frac{1}{N}\sum_{j=1}^N\Phi^\varepsilon(x_i-x_j) \Big)-\exp(-\Phi^\varepsilon\ast u^\varepsilon(t,x_i)) \Big)\nonumber\\
&\phantom{xxxxxxx}\cdot \big(\nabla_{x_i}\log {u^\varepsilon}^{\otimes N}-\nabla_{x_i}\Phi^\varepsilon\ast u^\varepsilon (t,x_i)\big)
\cdot \nabla_{x_i}\log \frac{u_N^\varepsilon}{{u^\varepsilon}^{\otimes N}}\,dx_1\cdots dx_N.
\end{align*}

The term $I_1$ can be divided into four parts
\begin{align}\label{relative3}
I_1\leq& \frac{1}{2N}\sum_{i=1}^N\int_{\mathbb{R}^{2N}}u_N^\varepsilon\exp\Big(-\frac{1}{N}\sum_{i=1}^N\Phi^\varepsilon(x_i-x_j) \Big)
\Big|\nabla_{x_i}\log \frac{u_N^\varepsilon}{{u^\varepsilon}^{\otimes N}} \Big|^2\,dx_1\cdots dx_N\nonumber\\
&+C\frac{1}{N}\sum_{i=1}^N\mathbb{E}\big(\big|\nabla_{x_i}\Phi^\varepsilon\ast u^\varepsilon(t,X_{N,i}^{\varepsilon})
-\nabla_{x_i}\Phi^\varepsilon\ast u^\varepsilon(t,\bar{X}_i^\varepsilon) \big|^2 \big)\nonumber\\
&+C\frac{1}{N}\sum_{i=1}^N\mathbb{E}\Big(\Big|\nabla_{x_i}\Phi^\varepsilon\ast u^\varepsilon(t,\bar{X}_i^\varepsilon)
-\frac{1}{N}\sum_{j=1}^N\nabla_{x_i}\Phi^\varepsilon(\bar{X}_i^\varepsilon-\bar{X}_j^\varepsilon) \Big|^2 \Big)\nonumber\\
&+C\frac{1}{N}\sum_{i=1}^N\mathbb{E}\Big(\Big|\frac{1}{N}\sum_{j=1}^N\big|\nabla_{x_i}\Phi^\varepsilon(\bar{X}_i^\varepsilon-\bar{X}_j^\varepsilon)
-\nabla_{x_i}\Phi^\varepsilon(X_{N,i}^\varepsilon-X_{N,j}^\varepsilon)\big| \Big|^2 \Big)\nonumber\\
=:&\frac{1}{2N}\sum_{i=1}^N\int_{\mathbb{R}^{2N}}u_N^\varepsilon\exp\Big(-\frac{1}{N}\sum_{i=1}^N\Phi^\varepsilon(x_i-x_j) \Big)
\Big|\nabla_{x_i}\log \frac{u_N^\varepsilon}{{u^\varepsilon}^{\otimes N}} \Big|^2\,dx_1\cdots dx_N
+I_{11}+I_{12}+I_{13}.
\end{align}
It follows from \eqref{Lp}, the bound for $\nabla\Phi$ in $L^{3/2}(\mathbb{R}^2)$ and Young's evolution inequality that
\begin{align*}
\|D^2\Phi^\varepsilon\ast u^\varepsilon\|_{L^\infty((0,T^*)\times\mathbb{R}^2)}
\leq \|D^2\Phi^\varepsilon\|_{L^{3/2}(\mathbb{R}^2)}\|u^\varepsilon\|_{L^\infty(0,T^*;L^3(\mathbb{R}^2))}
\leq C\|\nabla\Phi\|_{L^{3/2}(\mathbb{R}^2)}\|\nabla j^\varepsilon\|_{L^1(\mathbb{R}^2)}\leq \frac{C}{\varepsilon},
\end{align*}
where $C$ appeared in this section is a positive constant independent of $\varepsilon$ and $N$.
Together with Proposition \ref{mean-field_est_1}, we get for any $\beta\in (0,1)$
\begin{align}\label{relative4}
I_{11}\leq C\|D^2\Phi^\varepsilon\ast u^\varepsilon\|_{L^\infty((0,T^*)\times\mathbb{R}^2)}^2
\max_{i=1,\cdots, N}\mathbb{E}(|X_{N,i}^\varepsilon-\bar{X}_i^\varepsilon|^2)
\leq \frac{C}{\varepsilon^2}N^{-\beta}.
\end{align}
A simple calculation gives
$\|D^2\Phi^\varepsilon\|_{L^\infty(\mathbb{R}^2)}\leq \|\nabla\Phi\|_{L^1(\mathbb{R}^2)}\|\nabla j^\varepsilon\|_{L^\infty(\mathbb{R}^2)}
\leq \frac{C}{\varepsilon^3}$, and combining Proposition \ref{mean-field_est_1}, it holds
\begin{align}\label{relative5}
I_{13}\leq C\|D^2\Phi^\varepsilon\|_{L^\infty(\mathbb{R}^2)}^2
\max_{ i=1,\cdots, N}\mathbb{E}(|X_{N,i}^\varepsilon-\bar{X}_i^\varepsilon|^2)
\leq \frac{C}{\varepsilon^6}N^{-\beta}.
\end{align}
We proceed similarly as \eqref{A3r} to infer that
\begin{align}\label{relative6}
I_{12}\leq C\frac{\|\nabla\Phi^\varepsilon\|^2_{L^\infty(\mathbb{R}^2)}}{N}
\leq \frac{C}{\varepsilon^4}N^{-1},
\end{align}
where we have used
$\|\nabla\Phi^\varepsilon\|_{L^\infty(\mathbb{R}^2)}\leq \|\nabla\Phi\|_{L^1(\mathbb{R}^2)}\|j^\varepsilon\|_{L^\infty(\mathbb{R}^2)}\leq \frac{C}{\varepsilon^2}$.
Inserting \eqref{relative4}-\eqref{relative6} into \eqref{relative3}, we deduce that
\begin{align}\label{relative7}
I_1\leq& \frac{1}{2N}\sum_{i=1}^N\int_{\mathbb{R}^{2N}}u_N^\varepsilon\exp\Big(-\frac{1}{N}\sum_{i=1}^N\Phi^\varepsilon(x_i-x_j) \Big)
\Big|\nabla_{x_i}\log \frac{u_N^\varepsilon}{{u^\varepsilon}^{\otimes N}} \Big|^2\,dx_1\cdots dx_N
+\frac{C}{\varepsilon^6}N^{-\beta}+\frac{C}{\varepsilon^4}N^{-1}.
\end{align}

Now we focus on the term $I_2$.
By the mean value theorem, it holds
\begin{align*}
I_2\leq&\frac{C}{N}\sum_{i=1}^N\int_{\mathbb{R}^{2N}}u_N^\varepsilon
\Big|\frac{1}{N}\sum_{j=1}^N\Phi^\varepsilon(x_i-x_j) -\Phi^\varepsilon\ast u^\varepsilon(t,x_i) \Big|
\Big| \nabla_{x_i}\log \frac{u_N^\varepsilon}{{u^\varepsilon}^{\otimes N}}\Big|\,dx_1\cdots dx_N,
\end{align*}
where we have used $\nabla_{x_i}\log {u^\varepsilon}^{\otimes N}=\nabla_{x_i}\log u^\varepsilon(x_i)$ and
the uniform $L^\infty((0,T^*)\times\mathbb{R}^2)$ bound for $\nabla\log u^\varepsilon$
 together with the fact
$$
\|\nabla\Phi^\varepsilon\ast u^\varepsilon\|_{L^\infty(\mathbb{R}^2)}
\leq \|\nabla\Phi^\varepsilon\|_{L^{3/2}(\mathbb{R}^2)}\|u^\varepsilon\|_{L^3(\mathbb{R}^2)}
\leq C\|\nabla\Phi\|_{L^{3/2}(\mathbb{R}^2)}\|j^\varepsilon\|_{L^1(\mathbb{R}^2)}\leq C.
$$
Due to H\"older's inequality, we have
\begin{align*}
I_2\leq& \frac{1}{2N}\sum_{i=1}^N\int_{\mathbb{R}^{2N}}u_N^\varepsilon\Big|\nabla_{x_i}\log \frac{u_N^\varepsilon}{{u^\varepsilon}^{\otimes N}} \Big|^2\,dx_1\cdots dx_N\nonumber\\
&+\frac{C}{N}\sum_{i=1}^N\mathbb{E}\Big(\Big|\frac{1}{N}\sum_{j=1}^N\big|\Phi^\varepsilon(X_{N,i}^\varepsilon-X_{N,j}^\varepsilon)-
\Phi^\varepsilon(\bar{X}_i^\varepsilon-\bar{X}_j^\varepsilon) \big|\Big|^2 \Big)\nonumber\\
&+\frac{C}{N}\sum_{i=1}^N\mathbb{E}\Big(\Big|\frac{1}{N}\sum_{j=1}^N\big|\Phi^\varepsilon(\bar{X}_i^\varepsilon-\bar{X}_j^\varepsilon)
 -\Phi^\varepsilon\ast u^\varepsilon(\bar{X}_i^\varepsilon)\big|\Big|^2 \Big)\nonumber\\
 &+\frac{C}{N}\sum_{i=1}^N\mathbb{E}\big(|\Phi^\varepsilon\ast u^\varepsilon(\bar{X}_i^\varepsilon)-
 \Phi^\varepsilon\ast u^\varepsilon(X_{N,i}^\varepsilon) |^2 \big).
\end{align*}
In view of mean value theorem and Proposition \ref{mean-field_est_1}, we proceed similarly as \eqref{A3r} to deduce that for any $\beta\in (0,1)$
\begin{align} \label{relative8}
I_2\leq&\frac{1}{2N}\sum_{i=1}^N\int_{\mathbb{R}^{2N}}u_N^\varepsilon\Big|\nabla_{x_i}\log \frac{u_N^\varepsilon}{{u^\varepsilon}^{\otimes N}} \Big|^2\,dx_1\cdots dx_N\nonumber\\
&+C(\|\nabla\Phi^\varepsilon\|_{L^\infty(\mathbb{R}^2)}^2
+\|\nabla\Phi^\varepsilon\ast u^\varepsilon\|_{L^\infty((0,T^*)\times\mathbb{R}^2)}^2 )
\max_{1\leq i\leq N}\mathbb{E}(|X_{N,i}^\varepsilon-\bar{X}_i^\varepsilon|^2)
+C\frac{\|\Phi^\varepsilon\|_{L^\infty(\mathbb{R}^2)}^2}{N}\nonumber\\
\leq &\frac{1}{2N}\sum_{i=1}^N\int_{\mathbb{R}^{2N}}u_N^\varepsilon\Big|\nabla_{x_i}\log \frac{u_N^\varepsilon}{{u^\varepsilon}^{\otimes N}} \Big|^2\,dx_1\cdots dx_N
+\frac{C}{\varepsilon^{4}}N^{-\beta}+\frac{C}{\varepsilon^{4}}N^{-1}+CN^{-\beta}.
\end{align}
Combining \eqref{relative2}, \eqref{relative7} and \eqref{relative8} and keep the leading order terms in $N$, we get
\begin{align}\label{relatvie9}
\frac{d}{dt}\mathcal{H}(u_N^\varepsilon|{u^\varepsilon}^{\otimes N})
\leq&-\frac{1}{2N}\sum_{i=1}^N\int_{\mathbb{R}^{2N}}u_N^\varepsilon\Big|\nabla_{x_i}\log \frac{u_N^\varepsilon}{{u^\varepsilon}^{\otimes N}} \Big|^2\,dx_1\cdots dx_N\nonumber\\
&-\frac{1}{2N}\sum_{i=1}^N\int_{\mathbb{R}^{2N}}u_N^\varepsilon\exp\Big(-\frac{1}{N}\sum_{j=1}^N\Phi^\varepsilon(x_i-x_j) \Big)
\Big|\nabla_{x_i}\log\frac{u_N^\varepsilon}{{u^\varepsilon}^{\otimes N}} \Big|^2\,dx_1\cdots dx_N\nonumber\\
&+\frac{C}{\varepsilon^6}N^{-\beta}+\frac{C}{\varepsilon^{4}}N^{-\beta},
\end{align}
which implies by using the Csisz\'ar-Kullback-Pinsker inequality \cite{Villani02}
\begin{align}\label{relative9}
\sup_{t\in (0,T^*)}\|u^\varepsilon_{N,k}(t)-{u^\varepsilon}^{\otimes k}(t)\|_{L^1(\mathbb{R}^{2k})}^2
&\leq \sup_{t\in (0,T^*)}2\mathcal{H}_k(u_{N,k}^\varepsilon|{u^\varepsilon}^{\otimes k})(t)
\leq \sup_{t\in (0,T^*)} 4k\mathcal{H}(u_N^\varepsilon| {u^\varepsilon}^{\otimes N})(t)\nonumber\\
&\leq \frac{Ck}{\varepsilon^6}N^{-\beta}+\frac{Ck}{\varepsilon^{4}}N^{-\beta}\quad {\mbox{for any }}k>0,
\end{align}
more details can be found in \cite{CHH}.
Recall that $\varepsilon=(\lambda\log N)^{-\frac{1}{4}}$ with $0<\lambda\ll 1$, it holds for any $k\in \N_+$
\begin{align*}
\|u^\varepsilon_{N,k}-{u^\varepsilon}^{\otimes k}\|_{L^\infty(0,T^*;L^1(\mathbb{R}^{2k}))}^2\leq C(k)N^{-\alpha}\quad \mbox{for any } \alpha\in (0,\beta).
\end{align*}
Thus, in accordance with \eqref{a1}, we complete the proof of propagation of chaos in Theorem \ref{relative}.

\section*{Appendix}
In the appendix we present a list of useful lemmas applied in the proofs of this paper.
\begin{lemma}\label{entropymini}{\rm{(The entropy minimization)}}.
	Assume that $\phi$ is a function such that $e^\phi\in L^1(\mathbb{R}^2)$ and $\hat{w}:=\frac{\bar Ce^\phi}{\int_{\mathbb{R}^2}e^\phi\,dx}$,
	where $\bar C$ is a positive arbitrary constant. Let $\mathcal{E}:L^1_+(\mathbb{R}^2)\rightarrow \mathbb{R}\cup \{\infty\}$ be the entropy functional
	\begin{align*}
		\mathcal{E}(w;\phi):=\int_{\mathbb{R}^2}(w\log w-w\phi)\,dx.
	\end{align*}
	Then for any $w\in L_+^1(\R^2)$ satisfying $\int_{\R^2}w(x)\,dx=\bar C$, it holds
	\begin{align*}
		\mathcal{E}(w;\phi)\geq \mathcal{E}(\hat{w};\phi).
	\end{align*}
\end{lemma}

\begin{lemma}\label{Onofri}{\rm{(Onofri's inequality in $\R^2$)}}.
	Let $n$ be a function such that $n\in L^1(\R^2,H(x)\,dx)$ and $|\nabla n|\in L^2(\R^2)$. Then it holds
	\begin{align*}
		\int_{\R^2}e^{n(x)}H(x)\,dx
		\leq \exp\Big\{\int_{\R^2}n(x)H(x)\,dx+\frac{1}{16\pi}\int_{\R^2}|\nabla n(x)|^2\,dx \Big\}.
	\end{align*}
\end{lemma}

\begin{lemma}\label{f-}{\rm{(Estimate for $f(\log f)_-$).}}
	Let $\phi$ be a function such that $e^\phi\in L^1(\mathbb{R}^2)$ and $f$ be a non-negative function such that
	$(f\chi_{\{f\leq 1\}})\in L^1(\mathbb{R}^2)\cap L^1(\mathbb{R}^2;|\phi(x)|dx)$.
	Then there exists a constant $\hat C>0$ independent of $f$ such that
	\begin{align*}
		\int_{\mathbb{R}^2}f(\log f)_-\,dx\leq \hat C-\int_{\{f\leq 1\}}f(x)\phi(x)\,dx.
	\end{align*}
\end{lemma}
For a rigorous proof of Lemmas \ref{entropymini}-\ref{f-}, the reader is referred to \cite{CCC}.
\vskip5mm
Furthermore, we need to use the following version of the Gagliardo-Nirenberg inequality to proceed further $L^p$ estimate of $u^\varepsilon$. Although the proof is similar to \cite{TW2}, in which the result for the bounded domain was derived,
we present a detialed proof for the convenience of the reader.
\begin{lemma}{\rm{(Generalization of the Gagliardo-Nirenberg inequality).}}\label{GN}
	Let $q\in (1,\infty)$, $r\in (0,q)$ and $\alpha>0$. Then for any $\delta>0$, there exists $C_\delta>0$ such that
	\begin{align*}
		\|\varphi\|_{L^q(\R^2)}^q\leq \delta \|\nabla\varphi\|_{L^2(\R^2)}^{q-r}\|\varphi(\log |\varphi|)^\alpha\|_{L^r(\R^2)}^r
		+C_\delta\|\varphi\|_{L^r(\R^2)}^r,
	\end{align*}
	for any $\varphi\in L^r(\R^2)$ satisfying $\nabla \varphi\in L^2(\R^2)$ and $\varphi(\log |\varphi|)^\alpha\in L^r(\R^2)$.
\end{lemma}
\begin{proof}
	Taking into account the Gagliardo-Nirenberg inequality, we have
	\begin{align*}
		\|\varphi\|_{L^q(\R^2)}^q\leq \tilde C\|\nabla \varphi\|_{L^2(\R^2)}^{q-r}\|\varphi\|_{L^r(\R^2)}^r,
	\end{align*}
	where $\tilde C$ is a positive constant. We introduce $f\in W^{1,\infty}_{loc}(\R)$
	\begin{align*}
		f(s)=\begin{cases}
			0,\quad {\rm{if}}\;|s|\leq N,    \\
			2(|s|-N),\quad {\rm{if}}\;N<|s|<2N, \\
			|s|, \quad {\rm{if}}\;|s|\geq 2N.
		\end{cases}
	\end{align*}
	It is easy to see $0\leq f(s)\leq |s|$ and $|f'(s)|\leq 2$ for a.e. $s\in \R$.
	It follows that
	\begin{align*}
		&\|f(\varphi)\|_{L^r(\R^2)}^q\leq \|\varphi\|_{L^r(\R^2)}^q,\\
		&\|\nabla f(\varphi)\|_{L^2(\R^2)}^{q-r}=\|f'(\varphi)\nabla\varphi\|_{L^2(\R^2)}^{q-r}\leq 2^{q-r}\|\nabla\varphi\|_{L^2(\R^2)}^{q-r},\\
		&\|f(\varphi)\|_{L^r(\R^2)}^r\leq \int_{\{|\varphi|>N\}}|f(\varphi)|^r\,dx
		\leq \int_{\{|\varphi|>N\}}|\varphi|^r\frac{(\log |\varphi|)^{\alpha r}}{(\log |\varphi|)^{\alpha r}}\,dx
		\leq \frac{1}{(\log N)^{\alpha r}}\|\varphi(\log|\varphi|)^\alpha\|_{L^r(\R^2)}^r.
	\end{align*}
	Due to $0\leq |s|-f(s)\leq 2N$ on $\R$, we derive that
	\begin{align*}
		\||\varphi|-f(\varphi)\|_{L^q(\R^2)}^q
		\leq (2N)^{q-r}\int_{\R^2}||\varphi|-f(\varphi)|^r\,dx
		\leq (2N)^{q-r}\int_{\R^2}|\varphi|^r\,dx
		=(2N)^{q-r}\|\varphi\|_{L^r(\R^2)}^r.
	\end{align*}
	Summarizing, this shows that
	\begin{align*}
		\|\varphi\|_{L^q(\R^2)}^q
		&\leq 2^q\|f(\varphi)\|_{L^q(\R^2)}^q+2^q\||\varphi|-f(\varphi)\|_{L^q(\R^2)}^q\\
		&\leq 2^q\tilde C\|\nabla f(\varphi)\|_{L^2(\R^2)}^{q-r}\|f(\varphi)\|_{L^r(\R^2)}^r
		+2^q\||\varphi|-f(\varphi)\|_{L^q(\R^2)}^q\\
		&\leq \frac{ 2^{2q-r}\tilde C}{(\log N)^{\alpha r}}\|\nabla\varphi\|_{L^2(\R^2)}^{q-r}\|\varphi(\log|\varphi|)^\alpha\|_{L^r(\R^2)}^r
		+2^q(2N)^{q-r}\|\varphi\|_{L^r(\R^2)}^r\\
		&\leq \delta\|\nabla\varphi\|_{L^2(\R^2)}^{q-r}\|\varphi(\log |\varphi|)^\alpha\|_{L^r(\R^2)}^r
		+C_\delta\|\varphi\|_{L^r(\R^2)}^r,
	\end{align*}
	where we take $N>1$ large enough such that $\frac{2^{2q-r}\tilde C}{(\log N)^{\alpha r}}\leq \delta$.
\end{proof}

\begin{lemma} \label{liminf_ulogu}
Let $(u_n)_{n \in \N} \subset L^2(\R^2)$ be a non-negative sequence and $u_n\rightarrow u$ in $L^2(\R^2)$ as $n\rightarrow \infty$.
We assume that there exists a positive constant $C$ independent of $n$ such that, for any $n \in \N$,
$$ \int_{\R^2} u_n(x) |x|^2 \, dx  \leq C \mbox{ and } \int_{\R^2} u_n |\log u_n|(x) \, dx \leq C. $$
Then $u \log u \in L^1(\R^2)$ and up to a subsequence (not relabeled)
$$\int_{\R^2} u \log u(x) \, dx\leq \liminf_{n \to \infty} \int_{\R^2} u_n \log u_n(x) \, dx.$$
\end{lemma}

\begin{proof}
Define
$$ h(s):=s \log s -s +1, \quad s \geq 0,$$
which allows us to rewrite $u_n\log u_n$ as follows
\begin{align*}
u_n \log  u_n & = h \left(\frac{u_n}{e^{-|x|}} \right) e^{-|x|} + u_n \log (e^{-|x|}) + u_n - e^{-|x|} \\
& = h \left(\frac{u_n}{e^{-|x|}} \right) e^{-|x|} + (- u_n |x|) + u_n - e^{-|x|} .
\end{align*}
It follows from the assumptions that $u_n$ and $u_n |x|$ are uniformly bounded in $L^1(\R^2)$ (Soon we will see a similar argument, so we omit it here).
We combine this result and $u_n \log  u_n, e^{-|x|} \in L^1(\R^2)$ to get $h \left(\frac{u_n}{e^{-|x|}} \right) e^{-|x|} \in  L^1(\R^2)$ as well.
Since $u_n$ converges strongly to $u$ in $L^2(\R^2)$, we obtain that $u_n(x)\rightarrow u(x)$ for a.e. $x\in \R^2$.
In view of the non-negativity and continuity of $h$, we apply  Fatou's lemma to conclude
$\int_{\R^2}h(\frac{u}{e^{-|x|}})e^{-|x|}dx\leq \liminf_{n\rightarrow \infty}\int_{\R^2}h(\frac{u_n}{e^{-|x|}})e^{-|x|}dx$.
Due to the a.e. convergence of $u_n$ and Fatou's lemma, we have $\int_{\R^2}u dx\leq \liminf_{n\rightarrow \infty}\int_{\R^2}u_n dx$.
A similar argument proves that $\int_{\R^2}u|x|^2 dx\leq \liminf_{n\rightarrow \infty}\int_{\R^2}u_n|x|^2 dx$.
The only term which can't be treated in this way, due to the wrong sign, is $- u_n |x|$. We now discuss its convergence.
For every $R>0$, it holds
\begin{align*}
& \int_{\R^2} \big|u_n(x)|x| -u(x)|x|\big| \, dx \\
 \leq & \int_{B_R(0)} |u_n(x) -u(x)||x| \, dx  + \int_{B_R(0)^c} \frac{1}{|x|}|u_n(x) - u(x)||x|^2 \, dx \\
 \leq & R  \left| B_R(0) \right|^{\frac{1}{2}}  \Big( \int_{B_R(0)} |u_n(x) -u(x)|^2 \, dx \Big)^{\frac{1}{2}}
 + \frac{1}{R} \Big( \int_{B_R(0)^c} |u_n(x)||x|^2 \, dx + \int_{B_R(0)^c} |u(x)||x|^2 \, dx \Big) \\
 \leq & R  \left| B_R(0) \right|^{\frac{1}{2}} \Vert u_n -u \Vert_{L^2(\R^2)}  + \frac{C}{R}.
\end{align*}
Therefore, we infer that $u|x| \in L^1(\R^2)$ and
$$0 \leq  \limsup_{n \to \infty} \Vert u_n|x| -u|x| \Vert_{L^1(\R^2)} \leq \frac{C}{R} \rightarrow 0\mbox{ as }R\rightarrow \infty,$$
which implies $\lim_{n\rightarrow \infty}\int_{\R^2} -u_n(x)|x| \, dx =   \int_{\R^2}  -u(x)|x| \, dx$.\\
Summarizing, we deduce that
\begin{align*}
\infty > \liminf_{n \to \infty} \int_{\R^2} u_n \log u_n(x) \, dx \geq &\int_{\R^2} h \left(\frac{u}{e^{-|x|}} \right) e^{-|x|} \, dx  + \int_{\R^2} (- u |x|) \, dx  + \int_{\R^2} u \, dx  - \int_{\R^2} e^{-|x|} \, dx\\
=& \int_{\R^2} u \log u(x) \, dx,
\end{align*}
where we have used $u \log  u \in L^1(\R^2)$.
\end{proof}

\medskip
\indent
{\bf Acknowledgements:}
Yue Li acknowledge partial support from the Austrian Science Fund (FWF), grants P33010 and F65. This work has received funding from DFG grant CH955/8-1 and the European Research Council (ERC) under the European Union's Horizon 2020 research and innovation programme, ERC Advanced Grant no.~101018153.
The authors thank the reviewers for their constructive comments and helpful suggestions.

\end{document}